\documentclass[12pt, reqno]{amsart}
\usepackage{}
\usepackage{cancel}
\usepackage{stmaryrd}
\usepackage{amsmath}
\usepackage{amsfonts}
\usepackage{amssymb}
\usepackage{xcolor} 

\usepackage[all,cmtip]{xy}           
\usepackage{xspace}
\usepackage{bbding}
\usepackage{makecell}
\usepackage{txfonts}
\usepackage{enumerate}
\usepackage[shortlabels]{enumitem}
\usepackage{cite}

\usepackage{tikz}
\usepackage{titletoc}

\usepackage{ifpdf}
\ifpdf
 \usepackage[colorlinks,final,hyperindex]{hyperref}
\else
 \usepackage[colorlinks,final,backref=page,hyperindex,hypertex]{hyperref}
\fi
\usepackage[active]{srcltx}

\makeatletter

\begin{document}
\newcommand {\emptycomment}[1]{} 

\baselineskip=15pt
\newcommand{\nc}{\newcommand}
\newcommand{\delete}[1]{}
\nc{\mfootnote}[1]{\footnote{#1}} 
\nc{\todo}[1]{\tred{To do:} #1}

\nc{\mlabel}[1]{\label{#1}}  
\nc{\mcite}[1]{\cite{#1}}  
\nc{\mref}[1]{\ref{#1}}  
\nc{\meqref}[1]{\eqref{#1}} 
\nc{\mbibitem}[1]{\bibitem{#1}} 

\delete{
\nc{\mlabel}[1]{\label{#1}  
{\hfill \hspace{1cm}{\bf{{\ }\hfill(#1)}}}}
\nc{\mcite}[1]{\cite{#1}{{\bf{{\ }(#1)}}}}  
\nc{\mref}[1]{\ref{#1}{{\bf{{\ }(#1)}}}}  
\nc{\meqref}[1]{\eqref{#1}{{\bf{{\ }(#1)}}}} 
\nc{\mbibitem}[1]{\bibitem[\bf #1]{#1}} 
}

\newcommand {\comment}[1]{{\marginpar{*}\scriptsize\textbf{Comments:} #1}}
\nc{\mrm}[1]{{\rm #1}}
\nc{\id}{\mrm{id}}  \nc{\Id}{\mrm{Id}}
\nc{\admset}{\{\pm x\}\cup (-x+K^{\times}) \cup K^{\times} x^{-1}}

\def\a{\alpha}
\def\ad{associative D-}
\def\padm{$P$-admissible~}
\def\asi{ASI~}
\def\aybe{aYBe~}
\def\b{\beta}
\def\bd{\boxdot}
\def\bbf{\overline{f}}
\def\bF{\overline{F}}
\def\bbF{\overline{\overline{F}}}
\def\bbbf{\overline{\overline{f}}}
\def\bg{\overline{g}}
\def\bG{\overline{G}}
\def\bbG{\overline{\overline{G}}}
\def\bbg{\overline{\overline{g}}}
\def\bT{\overline{T}}
\def\bt{\overline{t}}
\def\bbT{\overline{\overline{T}}}
\def\bbt{\overline{\overline{t}}}
\def\bR{\overline{R}}
\def\br{\overline{r}}
\def\bbR{\overline{\overline{R}}}
\def\bbr{\overline{\overline{r}}}
\def\bu{\overline{u}}
\def\bU{\overline{U}}
\def\bbU{\overline{\overline{U}}}
\def\bbu{\overline{\overline{u}}}
\def\bw{\overline{w}}
\def\bW{\overline{W}}
\def\bbW{\overline{\overline{W}}}
\def\bbw{\overline{\overline{w}}}
\def\btl{\blacktriangleright}
\def\btr{\blacktriangleleft}
\def\calo{\mathcal{O}}
\def\ci{\circ}
\def\d{\delta}
\def\dd{\diamondsuit}
\def\D{\Delta}
\def\frakB{\mathfrak{B}}
\def\G{\Gamma}
\def\g{\gamma}
\def\gg{\mathfrak{g}}
\def\hh{\mathfrak{h}}
\def\k{\kappa}
\def\l{\lambda}
\def\ll{\mathfrak{L}}
\def\lh{\leftharpoonup}
\def\lr{\longrightarrow}
\def\N{Nijenhuis~}
\def\o{\otimes}
\def\om{\omega}
\def\opa{\cdot_{A}}
\def\opb{\cdot_{B}}
\def\p{\psi}
\def\r{\rho}
\def\ra{\rightarrow}
\def\rbs{Rota-Baxter system}
\def\rep{representation~}
\def\rh{\rightharpoonup}
\def\rr{{\Phi_r}}
\def\s{\sigma}
\def\srbs{symmetric Rota-Baxter system}
\def\st{\star}
\def\ss{\mathfrak{sl}_2}
\def\ti{\times}
\def\tl{\triangleright}
\def\tr{\triangleleft}
\def\v{\varepsilon}
\def\vp{\varphi}
\def\vth{\vartheta}
\def\wn{\widetilde{N}}
\def\wb{\widetilde{\beta}}

\newtheorem{thm}{Theorem}[section]
\newtheorem{lem}[thm]{Lemma}
\newtheorem{cor}[thm]{Corollary}
\newtheorem{pro}[thm]{Proposition}
\theoremstyle{definition}
\newtheorem{defi}[thm]{Definition}
\newtheorem{ex}[thm]{Example}
\newtheorem{rmk}[thm]{Remark}
\newtheorem{pdef}[thm]{Proposition-Definition}
\newtheorem{condition}[thm]{Condition}
\newtheorem{question}[thm]{Question}
\renewcommand{\labelenumi}{{\rm(\alph{enumi})}}
\renewcommand{\theenumi}{\alph{enumi}}
\newcommand{\End}{\mathrm{End}}

\nc{\ts}[1]{\textcolor{purple}{MTS:#1}}
\nc{\zc}[2]{\textcolor{blue}{ZC:#2}}
\font\cyr=wncyr10



 \title{Symmetric Rota-Baxter systems and applications}

 \author[Zhao]{Chan Zhao}
 \address{School of Mathematics and Statistics, Henan Normal University, Xinxiang 453007, China}
         \email{zhaochan2024@stu.htu.edu.cn}

 \author[Li]{Haiying Li}
 \address{School of Mathematics and Statistics, Henan Normal University, Xinxiang 453007, China}
        \email{lihaiying@htu.edu.cn}

 \author[Ma]{Tianshui Ma \textsuperscript{*}}
 \address{School of Mathematics and Statistics, Henan Normal University, Xinxiang 453007, China}
         \email{matianshui@htu.edu.cn}

  \thanks{\textsuperscript{*}Corresponding author}

\date{\today}

 \begin{abstract}
 Rota-Baxter operators and bialgebras are closely connected in several applications, such as the Connes-Kreimer renormalization framework and the operator approach to the classical Yang-Baxter equation. The concept of a Rota-Baxter system was introduced in 2016 as a generalization of a Rota-Baxter operator. In this work, we introduce a bialgebra structure compatible with a symmetric Rota-Baxter system, called a symmetric Rota-Baxter antisymmetric infinitesimal (ASI) bisystem. This bialgebra is characterized by generalizations of matched pairs of algebras and double constructions of Frobenius algebras to the setting of symmetric Rota-Baxter systems. Investigating the coboundary case leads to an enriched version of the associative Yang-Baxter equation (aYBe) adapted to symmetric Rota-Baxter systems. Antisymmetric solutions of this equation are used to construct symmetric Rota-Baxter ASI bisystems. We also introduce the notion of an $\mathcal{O}$-operator on a symmetric Rota-Baxter system, which produces solutions of the admissible aYBe in such systems and thereby gives rise to symmetric Rota-Baxter ASI bisystems. A symmetric Rota-Baxter bisystem generalizes several known structures, including Rota-Baxter Lie bisystems, Rota-Baxter ASI bialgebras, Rota-Baxter Lie bialgebras, averaging ASI bialgebras, averaging Lie bialgebras, and special apre-perm bialgebras. 
 \end{abstract}
\subjclass[2020]{
17B38,  
17A30,  
16T25,   
16T10,   
17B62.   
}

\keywords{Rota-Baxter system; Rota-Baxter algebra; Rota-Baxter ASI bialgebra; Rota-Baxter Lie bialgebra; averaging ASI bialgebra; averaging Lie bialgebra; special apre-perm bialgebra}

 \maketitle

 \vspace{-0.82cm}

  \tableofcontents


 \allowdisplaybreaks

 \section{Introduction} In this paper, we investigate the bialgebraic structures on symmetric Rota-Baxter systems and examine their properties. Based on these findings, we establish connections to various related structures, such as Rota-Baxter Lie systems, Rota-Baxter ASI bialgebras \cite{BGM}, Rota-Baxter Lie bialgebras \cite{BGLM,LS,Shi}, averaging ASI bialgebras \cite{BGLZ,HC}, averaging Lie bialgebras \cite{BGLZ1,HSZ}, and special apre-perm bialgebras \cite{BGLZ,ZL}.

 \subsection{Rota-Baxter systems} A Rota-Baxter algebra of weight $\lambda$ is a pair $(A, R)$, where $A$ is an algebra and $R: A\rightarrow A$ is a linear map such that the following equation holds:
 \begin{eqnarray}
 R(a) R(b)=R(R(a) b+a R(b)+\lambda a b),\label{eq:cee}
 \end{eqnarray}
 for all $a, b\in A$. Rota-Baxter algebra originated from the work of Baxter and Rota in fluctuation theory and combinatorics \cite{B,Ro}. The theory now plays a pivotal role in modern mathematics, particularly in the study of integrable systems, operads, and the renormalization of quantum field theories \cite{CK,GL,ZD,Guo1,GPZ,BLST,MLLC,MS22,MS25,SW,WBLS}.

 As a natural generalization of Rota-Baxter algebras, the notion of a Rota-Baxter system was introduced by T. Brzezi\'{n}ski in \cite{Br1}. It consists of a triple $(A, R, S)$, where $A$ is an algebra and $R, S: A \to A$ are linear operators satisfying  \vskip-6mm
 \begin{eqnarray*}
 &R(a)R(b)= R(R(a)b + a S(b)),\quad S(a)S(b)= S(R(a)b + a S(b)),&\label{eq:rbs1}\label{eq:rbs2} \end{eqnarray*}
 for all $a, b \in A$. Much like in the case of a Rota-Baxter algebra, such a system gives rise to a dendriform algebra structure. Furthermore, examples of Rota-Baxter systems can be constructed from quasitriangular covariant bialgebras, which can be seen as a natural generalization of infinitesimal bialgebras studied in \cite{Ag1,Ag2,EF}. Other studies on the Rota-Baxter system can be found in the literature \cite{BGM,Br2,MMS,MMS1,BGGZ,LW,DG,OPV}.

 By \cite[Lemma 2.2]{Br1}, if $R$ is a Rota-Baxter operator of weight $\lambda$ on $A$, then both $(A, R, R+\lambda \id)$ and $(A, R + \lambda \id, R)$ form Rota-Baxter systems. Inspired by this construction, in this paper, we study a class of Rota-Baxter systems in which both $(A, R, S)$ and $(A, S, R)$ are Rota-Baxter systems. This leads to the notion below: a \textbf{symmetric Rota-Baxter system} is a triple $(A, R, S)$ such that both $(A, R, S)$ and $(A, S, R)$ are Rota-Baxter systems. Just as a Rota-Baxter algebra induces a Rota-Baxter Lie algebra, a symmetric Rota-Baxter system induces a Rota-Baxter Lie system. While this holds for the symmetric case, a general Rota-Baxter system does not, in general, induce such a Lie system structure. This provides another motivation for studying symmetric Rota-Baxter systems.

 \subsection{Bialgebras}
 A bialgebra is a mathematical structure that unifies the notions of an algebra and a coalgebra in a compatible manner, with prominent examples including Lie bialgebras \cite{Dr}, Hopf algebras \cite{Sw}, and (antisymmetric) infinitesimal bialgebras \cite{Ag1,Ag2,Bai1,Zhe}. In recent years, bialgebraic structures have been extensively explored in the context of Rota-Baxter-type algebras, that is, algebras endowed with a linear operator-leading to advances in Rota-Baxter ASI bialgebras \cite{BGM}, Rota-Baxter bialgebras \cite{MLiu,MMS,MLCW}, Rota-Baxter Lie bialgebras \cite{BGLM,LS,Shi}, as well as their averaging, differential, and Nijenhuis counterparts \cite{BGLZ,HC,BGLZ1,HSZ,LLB,LM,Lar,MLo}.

 In the associative setting, the notion of a Rota-Baxter bialgebra was introduced in \cite{MLiu}, where examples were constructed via the Radford biproduct from Hopf algebra theory. This framework was later extended to Rota-Baxter bisystems and Rota-Baxter Hopf $\pi$-algebras in \cite{MMS} and \cite{MLCW}, respectively. However, a coherent compatibility condition between Rota-Baxter operators on algebras and Rota-Baxter co-operators on coalgebras remained unclear. To address this gap, Rota-Baxter ASI bialgebras were introduced in \cite{BGM}, which established such compatibility conditions and can be characterized equivalently by matched pairs and double constructions. To our knowledge, the study of bialgebraic structures involving multiple operators remains an open problem. This paper presents a solution by developing the theory for algebras equipped with two operators.

 \subsection{Main results and layout of the paper} The main notions and constructions in this paper are summarized in the following diagram.
$$ 
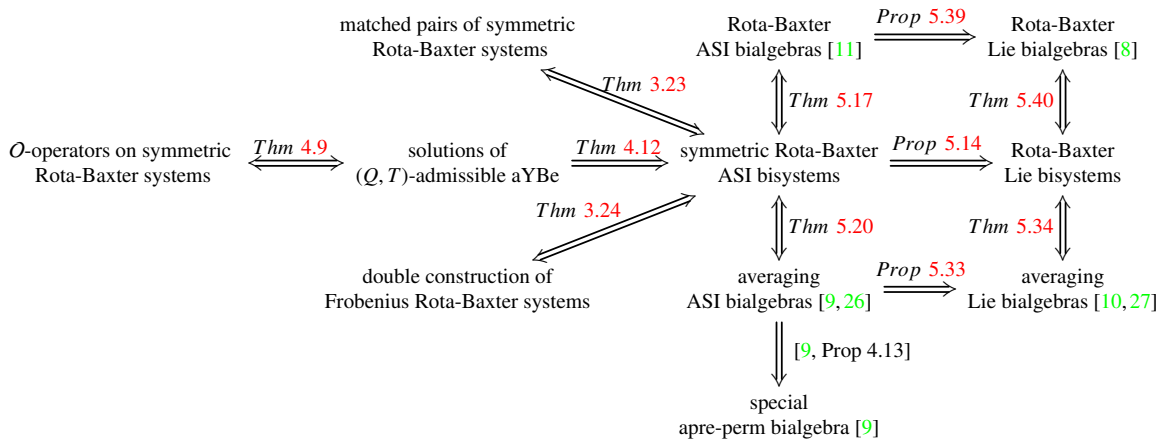
\xymatrix{
 &\text{matched pairs of symmetric}\atop {\text{ Rota-Baxter systems}}
 &\text{Rota-Baxter}\atop {\text{\asi bialgebras\cite{BGM}}}\ar@2{->}^{Prop~\ref{pro:hb}}[r]
 &\text{Rota-Baxter}\atop {\text{Lie bialgebras\cite{BGLM,LS}}}\\
\mathcal{O}\text{-operators on symmetric }\atop{\text{Rota-Baxter systems}} \ar@2{<->}^{Thm~\ref{thm:dm}}[r]
&
\text{solutions of}\atop {\text{$(Q,T)$-admissible aYBe}}
\ar@2{->}^{Thm~\ref{thm:do}}[r]
& \text{symmetric Rota-Baxter}\atop {\text{\asi bisystems}} \ar@2{->}^{\quad Prop~\ref{pro:en}}[r] \ar@2{<->}_{~~~Thm~\ref{thm:da}\quad}[dl] \ar@2{<->}_{~~Thm~\ref{thm:cz}}[ul]\ar@2{<->}_{Thm ~\ref{thm:es}}[u]
\ar@2{<->}^{Thm~\ref{thm:eu}}[d]
&\text{Rota-Baxter}\atop {\text{Lie bisystems}} \ar@2{<->}^{Thm~\ref{thm:bc}}[u]\ar@2{<->}_{Thm~\ref{thm:eu-1}
}[d]
& \\
&\text{double construction of}\atop {\text{Frobenius Rota-Baxter
systems}}
&\text{averaging}\atop {\text{\asi bialgebras \cite{BGLZ,HC}}}\ar@2{->}^{Prop~\ref{pro:ex}}[r] \ar@2{->}^{\tiny\hbox{\cite[Prop 4.13]{BGLZ}}}[d]
&\text{averaging}\atop {\text{Lie bialgebras \cite{BGLZ1,HSZ}}} \\
&
&\text{special}\atop {\text{apre-perm bialgebra \cite{BGLZ}}}
}
$$

 In Section \ref{se:srbsrep}, we introduce the notion of a symmetric Rota-Baxter system $(A, R, S)$, defined by the condition that both $(A, R, S)$ and $(A, S, R)$ are Rota-Baxter systems in the sense of Brzezi\'{n}ski \cite{Br1}. We then provide several constructions of such systems. Furthermore, we introduce the concept of a dual representation of a representation of a symmetric Rota-Baxter system, which plays a key role in understanding the bialgebraic structures associated with these systems.

 Section \ref{se:srbbialg} is devoted to symmetric Rota-Baxter ASI bisystems. We define this notion and give equivalent characterizations in terms of matched pairs and double constructions of symmetric Rota-Baxter systems.

 In Section \ref{se:aybe}, we focus on the coboundary case of symmetric Rota-Baxter ASI bisystems. Explicit constructions are provided via solutions to an admissible aYBe and by means of $\mathcal{O}$-operators in a symmetric Rota-Baxter system.

 Finally, in Section \ref{se:appl}, we show how our results recover and unify several existing bialgebraic structures, including Rota-Baxter Lie systems, Rota-Baxter ASI bialgebras \cite{BGM}, Rota-Baxter Lie bialgebras \cite{BGLM,LS,Shi}, averaging ASI bialgebras \cite{BGLZ,HC}, averaging Lie bialgebras \cite{BGLZ1,HSZ}, and special apre-perm bialgebras \cite{BGLZ,ZL}.

 \smallskip	
 \noindent{\bf Notations:} Throughout this paper, all vector spaces, tensor products and linear homomorphisms are over a field $K$ of characteristic zero. We denote by $\id_M$ the identity map from $M$ to $M$, $\sigma: M\otimes N\to N\otimes M$ by the flip map. All (co)algebras in this paper are assumed to be (co)associative. And we freely use the Sweedler notation for a coalgebra in \cite{Sw}. For brevity, let $(C, \D)$ be a coalgebra, we write a comultiplication $\D(c)$ as $c_{(1)}\otimes c_{(2)}, ~\forall~c\in C$ without the summation sign.
	
 \section{Symmetric Rota-Baxter systems and their representations}\label{se:srbsrep}

 \subsection{Symmetric Rota-Baxter systems} Recall that an \textbf{(associative) algebra} is a pair $(A,\cdot)$ consisting of a vector space $A$ and a binary operation $\cdot: A\otimes A\rightarrow A$ satisfying $(a b) c=a (b c)$, for all $a, b, c\in A$.

 \begin{defi}\label{de:ea} A \textbf{symmetric Rota-Baxter system} is a triple $((A, \cdot), R, S)$ ($(A, R, S)$ for short), where $A$ is an algebra and $R, S: A\rightarrow A$ are linear maps satisfying the identities:
 \begin{eqnarray}
 R(a) R(b)=R(R(a) b+a S(b))=R(S(a) b+a R(b)),\label{eq:ea0}\\
 S(a) S(b)=S(R(a) b+a S(b))=S(S(a) b+a R(b)),\label{eq:ea1}
 \end{eqnarray}
 where $a, b\in A$. A morphism between two symmetric Rota-Baxter systems $(A, R_{A}, S_{A})$ and $(B, R_{B}, S_{B})$ is an algebra map $f : A \rightarrow B$ such that $f \circ R_{A} = R_{B} \circ f$ and $f \circ S_{A} = S_{B} \circ f$.
 \end{defi}

 \begin{rmk}\label{rmk:eb}
  \begin{enumerate}[(1)]
    \item The definition of a symmetric Rota-Baxter system is symmetric in the operators $R$ and $S$, that is, if $(A, R, S)$ is a symmetric Rota-Baxter system, then so is $(A, S, R)$.

    \item If $A$ is a commutative algebra, then a Rota-Baxter system $(A, R, S)$ is symmetric.

    \item If $R$ or $S$ is bijective, then $R(a) b+a S(b)=S(a) b+a R(b)$.

    \item If $(A, R, S)$ is a symmetric Rota-Baxter system, then both $(A, R, 0)$ and $(A, 0, S)$ are averaging algebras.
  \end{enumerate}
 \end{rmk}

 \begin{ex}\label{ex:cee} Let $(A, \cdot)$ be a $2$-dimensional algebra with basis $e, f$, where $\cdot$ is given by
 \begin{center}
 \begin{tabular}{r|rr}
 $\cdot$ & $e$  & $f$ \\
 \hline
 $e$ & $0$  & $0$  \\
 $f$ & $ke$  & $kf$  \\
 \end{tabular},~~ 
 where $0\neq k\in \mathbb{R}$.
 \end{center}\vskip1mm
 The following choices of $R, S$ give symmetric Rota-Baxter systems.
 \begin{enumerate}[(1)]
 \item $R(e)=p_{1}e+p_{2}f,~~ R(f)=p_{3}e+\frac{p_{2}p_{3}}{p_{1}}f,\\ S(e)=-\frac{p_{2}p_{3}}{p_{1}}e+p_{2}f,~~ S(f)=p_{3}e-p_{1}f~(p_{1}\neq 0)$,
 \item $R(e)=0,~~ R(f)=p_{1}e+p_{2}f,\quad S(e)=-p_{2}e,~~ S(f)=p_{1}e$,
 \item $R(e)=p_{1}f,~~ R(f)=p_{2}f,\quad S(e)=-p_{2}e+p_{1}f,~~ S(f)=0 ~(p_{1}\neq 0)$,
 \item $R(e)=0,~~ R(f)=p_{1}e,\quad S(e)=0,~~ S(f)=p_{2}e~(p_{2}\neq p_{1})$,
 \item $R(e)=0,~~ R(f)=p_{1}e+p_{2}f,\quad S=0~(p_{2}\neq 0)$,
 \item $R=0,\quad S(e)=0,~~ S(f)=p_{1}e+p_{2}f~(p_{2}\neq 0)$,
 \item $R=0,\quad S(e)=p_{1}e,~~ S(f)=p_{1}f~(p_{1}\neq 0)$,
 \item $R(e)=p_{1}e,~~ R(f)=p_{1}f,\quad S=0~(p_{1}\neq 0)$,
 \end{enumerate}
  where $p_{i}, i=1, 2, 3$ are parameters.
 \end{ex}

 \begin{ex}\label{ex:cggg} Let $A$ be an algebra. Suppose that $\mathfrak{r}, \mathfrak{s}$ are central in $A$ such that $\mathfrak{r} \mathfrak{s}=0$. Define $R, S: A\rightarrow A$ by
 \begin{eqnarray*}
 R:a\mapsto a~ \mathfrak{r},\quad S:a\mapsto \mathfrak{s}~a,
 \end{eqnarray*}
 then $(A, R, S)$ is a symmetric Rota-Baxter system.
 \end{ex} 

 \begin{proof} Since $\mathfrak{s}$ is central, $(A, R, S)$ is a Rota-Baxter system by \cite[Example 2.4]{Br1}. Since $\mathfrak{r}$ and $\mathfrak{s}$ are central and $\mathfrak{r} \mathfrak{s}=0$, $(A, S, R)$ is a Rota-Baxter system. Hence, $(A, R, S)$ is a symmetric Rota-Baxter system.
 \end{proof}

 \begin{lem}\label{lem:ec} Let $A$ be an algebra and $\lambda\in K$. Then $R$ is a Rota-Baxter operator of weight $\lambda$ on $A$ if and only if $(A, R, R+\lambda \id_{A})$ or $(A, R+\lambda \id_{A}, R)$ is a symmetric Rota-Baxter system.
 \end{lem}

 \begin{proof} $(\Rightarrow)$ By \cite[Lemma 2.2]{Br1}, we know that both $(A, R, R+\lambda \id_{A})$ and $(A, R+\lambda \id_{A}, R)$ are Rota-Baxter systems, i.e., $(A, R, R+\lambda \id_{A})$ or $(A, R+\lambda \id_{A}, R)$ is a symmetric Rota-Baxter system.

 $(\Leftarrow)$ For $R$ and $S=R+\lambda \id_{A}$, Eqs.(\ref{eq:ea0}) and (\ref{eq:ea1}) are exactly Eq.(\ref{eq:cee}), i.e., $R$ is a Rota-Baxter operator of weight $\lambda$ on $A$.
 \end{proof} 

 A {\bf dendriform algebra} \cite{Lo} is a system consisting of a vector space $V$ and two bilinear operations $\prec, \succ$ on $V$ such that, for all $a, b, c \in V$,
 \begin{eqnarray*}
 &(a \prec b) \prec c= a \prec (b \prec c + b \succ c),& \label{eq:1.2a} \\
 &a \succ (b \prec c)= (a \succ b) \prec c, & \label{eq:1.2b} \\
 &a \succ (b \succ c)= (a \prec b + a \succ b) \succ c. & \label{eq:1.2c}
 \end{eqnarray*}

 \begin{lem}\label{lem:ed} Let $(A, R, S)$ be a symmetric Rota-Baxter system. Define the linear maps $\prec, $ $\succ, \prec', \succ': A\otimes A \rightarrow A$ by
 \begin{eqnarray*}
  a \prec b = a S(b), \ \ \  a \succ b = R(a) b,\ \ \  a \prec' b = a R(b), \ \ \  a \succ' b = S(a) b, \forall~ a, b\in A.
 \end{eqnarray*}
 Then both $(A, \prec, \succ)$ and $(A, \prec', \succ')$ are dendriform algebras.
 \end{lem}

 \begin{proof} This follows directly from \cite[Proposition 2.5]{Br1}.
 \end{proof}

 Recall from \cite{Vin,Bai3} that a {\bf (left) pre-Lie algebra} is a pair $(A, \circ)$ consisting of a vector space $A$ and a bilinear map $\circ : A \otimes A \rightarrow A$ such that for all $x, y, z \in A$,
 \begin{equation*}
 (x \circ y) \circ z - x \circ (y \circ z) = (y \circ x) \circ z - y \circ (x \circ z).
 \end{equation*} 

 \begin{rmk}\label{rmk:ef} Let $(A, R, S)$ be a symmetric Rota-Baxter system. Then by \cite[Corollary 2.7]{Br1} and Lemma \ref{lem:ed}, both $(A, \star)$ and $(A, \star')$ are associative algebras, both $(A, \bullet)$ and $(A, \bullet')$ are pre-Lie algebras, where the linear maps $\star, \star', \bullet, \bullet': A\otimes A\rightarrow A$ are defined by
 \begin{eqnarray*}\label{eq:ef0}
 a \star b = R(a) b+a S(b),  & a \star' b = a R(b)+S(a) b, \\
 a \bullet b = R(a) b-b S(a),& a \bullet' b = S(a) b-b R(a),
 \end{eqnarray*}
 for all $a, b\in A$.
 \end{rmk}

 A symmetric Rota-Baxter system can be obtained by a symmetric associative Yang-Baxter pair.

 \begin{defi}\label{de:eh} Let $A$ be an algebra. A \textbf{symmetric associative Yang-Baxter pair} is a pair $(\mathfrak{r}, \mathfrak{s})$, where $\mathfrak{r}, \mathfrak{s} \in A\otimes A$ such that the following equations hold:
 \begin{eqnarray*}
 &\mathfrak{r}_{12}\mathfrak{r}_{23}=\mathfrak{r}_{13}\mathfrak{r}_{12}
 +\mathfrak{s}_{23}\mathfrak{r}_{13}=\mathfrak{r}_{13}\mathfrak{s}_{12}
 +\mathfrak{r}_{23}\mathfrak{r}_{13},&\\
 &\mathfrak{s}_{12}\mathfrak{s}_{23}=\mathfrak{s}_{13}\mathfrak{r}_{12}
 +\mathfrak{s}_{23}\mathfrak{s}_{13}=\mathfrak{s}_{13}\mathfrak{s}_{12}
 +\mathfrak{r}_{23}\mathfrak{s}_{13},&
 \end{eqnarray*}
 where $\mathfrak{r}_{12}=\mathfrak{r}\otimes 1, \mathfrak{r}_{23}=1\otimes \mathfrak{r}, etc.$
 \end{defi}

 \begin{rmk}\label{rmk:ei} $(\mathfrak{r}, \mathfrak{s})$ is a symmetric associative Yang-Baxter pair if and only if both $(\mathfrak{r}, \mathfrak{s})$ and $(\mathfrak{s}, \mathfrak{r})$ are associative Yang-Baxter pairs (see \cite[Definition 3.1]{Br1}).
 \end{rmk}

 \begin{pro}\label{pro:ej} Let $(\mathfrak{r}, \mathfrak{s})$ be a symmetric associative Yang-Baxter pair on $A$. Define the linear maps $R, S: A \rightarrow A$ by
 \begin{eqnarray}\label{eq:ej1}
 R(a) =\sum \mathfrak{r}^{[1]} a \mathfrak{r}^{[2]}, \ \ \  S(a) = \sum \mathfrak{s}^{[1]} a \mathfrak{s}^{[2]}, \forall~ a\in A,
 \end{eqnarray}
 here $\mathfrak{r}=\sum \mathfrak{r}^{[1]} \otimes \mathfrak{r}^{[2]},  \mathfrak{s}=\sum \mathfrak{s}^{[1]} \otimes \mathfrak{s}^{[2]}$. Then $(A, R, S)$ is a symmetric Rota-Baxter system.
 \end{pro}

 \begin{proof} Similar to the proof of \cite[Proposition 3.4]{Br1}.
 \end{proof}

 \subsection{Dual representation}

 \begin{defi}\label{de:cb} Let $A$ be an algebra. An \textbf{$A$-bimodule} is a triple $(M, \ell, r)$, where $M$ is a vector space and $\ell, r: A\rightarrow \End(M)$ are linear maps such that the following equations hold:
 \begin{eqnarray}
 &\ell(a)(\ell(b)m)=\ell(a b)m,~~ m r(a b)= (m r(a))r(b),&\label{eq:cb}\\
 &\ell(a)(m r(b))=(\ell(a)m)r(b),&\label{eq:cb1}
 \end{eqnarray}
 for all $a, b\in A$.
 \end{defi}

 \begin{ex}\label{ex:cd} Let $(A,\cdot)$ be an algebra. Define linear maps $\mathcal{L}, \mathcal{R}: A\rightarrow \End(A)$ by $\mathcal{L}(a)b=a\cdot b= a\mathcal{R}(b),$ for all $a, b\in A$. Then, $(A, \mathcal{L}, \mathcal{R})$ is an $A$-bimodule, which is called an \textbf{adjoint $A$-bimodule}.
 \end{ex}

 \begin{lem}\label{lem:semidirectpro}{\em \cite{Bai1,BGM}} Let $A$ be an algebra. Then $(M, \ell, r)$ is an $A$-bimodule if and only if there is an algebra $(A\oplus M, \cdot_{A\oplus M})$ on the direct sum $A\oplus M$ of vector spaces by defining the multiplication $\cdot_{A\oplus M}$ by
 \begin{equation}\label{eq:cd1}
 (a+m)\cdot_{A\oplus M}(b+n)=a\cdot b+\ell(a)n+m r(b),
 \end{equation}
 for all $a, b \in A$, $m, n\in M$. We denote it by $A\ltimes_{\ell, r}M$.
 \end{lem}

 \begin{defi}\label{de:cf} Let $(A, R, S)$ be a symmetric Rota-Baxter system, $(M, \ell, r)$ an $A$-bimodule, $\alpha, \beta: M\rightarrow M$ two linear maps. If for all $a\in A$ and $m\in M$,
 \begin{eqnarray}
 &&\ell(R(a))\alpha(m)=\alpha(\ell(R(a))m+\ell(a)\beta(m))
 =\alpha(\ell(S(a))m+\ell(a)\alpha(m)),\label{eq:cf}\\
 &&\alpha(m)r(R(a))=\alpha(m r(R(a))+\beta(m)r(a))
 =\alpha(m r(S(a))+\alpha(m)r(a)),\label{eq:cf2}\\
 &&\ell(S(a))\beta(m)=\beta(\ell(R(a))m+\ell(a)\beta(m))
 =\beta(\ell(S(a))m+\ell(a)\alpha(m)), \label{eq:cf1}\\
 &&\beta(m)r(S(a))=\beta(m r(R(a))+\beta(m)r(a))=\beta(m r(S(a))+\alpha(m)r(a)),\label{eq:cf3}
 \end{eqnarray}
 then we call $(M, \ell, r, \alpha, \beta)$ a \textbf{representation} of $(A, R, S)$.

 Let $(M_{1}, \ell_{1}, r_{1}, \alpha_{1}, \beta_{1})$ and $(M_{2}, \ell_{2}, r_{2}, \alpha_{2}, \beta_{2})$ be representations of $(A, R, S)$. A linear map $f: M_{1}\rightarrow M_{2}$ is called a {\bf homomorphism from $(M_{1}, \ell_{1}, r_{1}, \alpha_{1}, \beta_{1})$ to $(M_{2}, \ell_{2}, r_{2}, \alpha_{2}, \beta_{2})$} if 
 for all $m\in M_{1}$,
 \begin{eqnarray*}
 &f(\ell_{1}(a)m)=\ell_{2}(a)f(m), ~~f(m r_{1}(a))=f(m)r_{2}(a),&\\ 
 &f(\alpha_{1}(m))=\alpha_{2}(f(m)),~~f(\beta_{1}(m))=\beta_{2}(f(m)).& 
 \end{eqnarray*}
 If furthermore $f$ is bijective, then we say that the representations $(M_{1}, \ell_{1}, r_{1}, \alpha_{1}, \beta_{1})$ and $(M_{2}, \ell_{2}, r_{2}, \alpha_{2}, \beta_{2})$ are {\bf equivalent}.
 \end{defi}

 \begin{rmk}\label{rmk:ga} $(M, \ell, r, \alpha, \beta)$ is a representation of a symmetric Rota-Baxter system $(A, R, S)$ means that $(M, \ell, r, \beta, \alpha)$ is a representation of a symmetric Rota-Baxter system $(A, S, R)$, and vice versa.
 \end{rmk}

 \begin{pro}\label{pro:cff} Let $(A, R, S)$ be a symmetric Rota-Baxter system, $(M,\ell, r)$ an $A$-bimodule, $\alpha, \beta: M\rightarrow M$ linear maps. For all $a\in A$, $m\in M$, define two linear maps
 \begin{eqnarray}
 R_{A\oplus M}: A\oplus M \rightarrow A\oplus M,~~~R_{A\oplus M}(a+m):=R(a)+\alpha(m),\label{eq:cff1}\\
 S_{A\oplus M}: A\oplus M \rightarrow A\oplus M,~~~S_{A\oplus M}(a+m):=S(a)+\beta(m).\label{eq:cff2}
 \end{eqnarray}
 Then together with the multiplication defined in Eq.\eqref{eq:cd1}, $(A\oplus M, R_{A\oplus M}, S_{A\oplus M})$ is a symmetric Rota-Baxter system if and only if $(M, \ell, r, \alpha, \beta)$ is a representation of $(A, R, S)$. The resulting symmetric Rota-Baxter system is denoted by $(A\ltimes_{\ell, r}M, R+\alpha, S+\beta)$ and is called a \textbf{semi-direct product} of $(A, R, S)$ by its representation $(M, \ell, r, \alpha, \beta)$.
 \end{pro}

 \begin{proof} It is a special case of Theorem \ref{thm:cp}, so we omit the proof here.
 \end{proof}

 Let us recall from \cite{GL} that a representation of a Rota-Baxter algebra $(A, R)$ of weight $\lambda$ is a quadruple $(V, \ell, r, \alpha)$, where $(V, \ell, r)$ is an $A$-bimodule and $\alpha$ is a linear operator on $V$ such that, for all $a \in A, v \in V$,
 \begin{eqnarray*}
 &\ell(R(a))\alpha(v)= \alpha(\ell(R(a))v) + \alpha(\ell(a)\alpha(v)) + \lambda \alpha(\ell(a)v),&\\
 &\alpha(v)r(R(a)) = \alpha(\alpha(v)r(a)) + \alpha(v r(R(a))) + \lambda\alpha(v r(a)).&
 \end{eqnarray*}

 We now present the relationship between a representation of a symmetric Rota-Baxter system and that of a Rota-Baxter algebra, which extends the relationship between the algebras themselves.

 \begin{pro} \label{pro:cg} Let $(A, R)$ be a Rota-Baxter algebra of weight $\lambda$. Then
 \begin{enumerate}[(1)]
   \item \label{it:pro:cg1} $(M, \ell, r, \alpha, \alpha+ \lambda \id_{M})$ is a representation of the symmetric Rota-Baxter system $(A, R, R+\lambda \id_{A})$ if and only if $(M, \ell, r, \alpha)$ is a representation of $(A, R)$.

   \item \label{it:pro:cg2} $(M, \ell, r, \alpha+ \lambda \id_{M}, \alpha)$ is a representation of the symmetric Rota-Baxter system $(A, R+\lambda \id_{A}, R)$ if and only if $(M, \ell, r, \alpha)$ is a representation of $(A, R)$.
 \end{enumerate}
 \end{pro}

 \begin{proof} \ref{it:pro:cg1} By Lemma \ref{lem:ec}, $(A, R)$ is a Rota-Baxter algebra of weight $\lambda$ if and only if $(A, R, R+\lambda \id_{A})$ is a symmetric Rota-Baxter system. Then $(M, \ell, r, \alpha)$ is a representation of $(A, R)$ $\stackrel{\small\hbox{\cite[Prop. 2.10]{BGM}}}{\Longleftrightarrow}$ $(A\ltimes_{\ell, r}M, R+\alpha)$ is a semi-direct product Rota-Baxter algebra of weight $\lambda$ $\stackrel{\small\hbox{Lem. \ref{lem:ec}}}{\Longleftrightarrow}$ $(A\ltimes_{\ell, r}M, R+\alpha, (R+\alpha)+\lambda \id_{A\ltimes_{\ell, r}M})$ is a semi-direct product symmetric Rota-Baxter system $\stackrel{\small\hbox{Prop. \ref{pro:cff}}}{\Longleftrightarrow}$ $(M, \ell, r, \alpha, \alpha+ \lambda \id_{M})$ is a representation of a symmetric Rota-Baxter system $(A, R, R+\lambda \id_{A})$.

  \ref{it:pro:cg2} It is direct by Remark \ref{rmk:ga} and Item \ref{it:pro:cg1}.
 \end{proof}

 \begin{pro} \label{pro:cggg} Let $(M, \ell, r, \alpha, \beta)$ be a representation of a symmetric Rota-Baxter system $(A, R, S)$. Consider the $A$-bimodule $\End(M)$ with the bimodule maps $\tilde{\ell}, \tilde{r}: A\rightarrow {\large \End}(\End(M))$ given by
 \begin{eqnarray*}
 (\tilde{\ell}(a)f)(m)=\ell(a)(f(m)),\quad (f\tilde{r}(a))(m)=(f(m))r(a).
 \end{eqnarray*}
 for all $a\in A, m\in M, f\in \End(M)$. Define
 \begin{eqnarray*}
 \tilde{\alpha}: \End(M)\rightarrow \End(M),~~\tilde{\alpha}(f)(m):=\alpha(f(m)).\\
 \tilde{\beta}: \End(M)\rightarrow \End(M),~~\tilde{\beta}(f)(m):=\beta(f(m)).
 \end{eqnarray*}
 Then $(\End(M), \tilde{\ell}, \tilde{r}, \tilde{\alpha}, \tilde{\beta})$ is a representation of $(A, R, S)$.
 \end{pro}

 \begin{proof} The proof follows by a direct computation. \end{proof}

 Let us recall from \cite{Bai2} that a {\bf representation of a pre-Lie algebra $(A, \circ)$} is a triple $(V, \rho, \varphi)$, where $V$ is a vector space, $\rho, \varphi :A \rightarrow \End(V)$ are linear maps such that, for all $x, y \in A$,
 \begin{eqnarray}
 &\rho(x \circ y - y \circ x)v = \rho(x)(\rho(y)v) - \rho(y)(\rho(x)v),&\label{eq:reppreliealg1}\\
 &\varphi(x \circ y)v = \rho(x)(\varphi(y)v) - \varphi(y)(\rho(x)v) +\varphi(y)(\varphi(x)v).&\label{eq:reppreliealg2}
 \end{eqnarray}

 \begin{pro} \label{pro:ch} Let $(M, \ell, r, \alpha, \beta)$ be a representation of a symmetric Rota-Baxter system $(A, R, S)$. For all $a\in A$ and $m\in M$,
 \begin{enumerate}
   \item[(1)] \label{it:pro:ch} define $\hat{\ell}_i, \hat{r}_i: A\rightarrow \End(M), i=1,2$ by
   \begin{eqnarray*}
   &&\hat{\ell}_{1}(a)m=\ell(R(a))m+\ell(a)\beta(m),\\
   &&m\hat{r}_{1}(a)= mr(S(a))+\alpha(m)r(a),\\
   &&\hat{\ell}_{2}(a)m=\ell(S(a))m+\ell(a)\alpha(m),\\
   &&m\hat{r}_{2}(a)= mr(R(a))+\beta(m)r(a).
   \end{eqnarray*}
   Then

   \begin{enumerate}
   \item[(1a)] \label{it:pro:ch0} $(M, \hat{\ell}_{1}, \hat{r}_{1})$ is an $(A, \star)$-bimodule, where the associative product $\star$ is given by $a\star b= R(a)b+ a S(b)$.
   \item[(1b)] \label{it:pro:ch1} $(M, \hat{\ell}_{2}, \hat{r}_{2})$ is an $(A, \star')$-bimodule, where  the associative product $\star'$ is given by $a\star' b= S(a)b+ aR(b)$.
   \end{enumerate}

   \item [(2)]\label{it:pro:ch4} define $\hat{\rho}_i, \hat{\varphi}_i: A\rightarrow \End(M), i=1,2$ by
   \begin{eqnarray*}
   &&\hat{\rho}_{1}(a)m=\ell(R(a))m-m r(S(a)),\\
   &&\hat{\varphi}_{1}(a)m= \alpha(m)r(a)-\ell(a)\beta(m),\\
   &&\hat{\rho}_{2}(a)m=\ell(S(a))m-m r(R(a)),\\
   &&\hat{\varphi}_{2}(a)m= \beta(m)r(a)-\ell(a)\alpha(m).
   \end{eqnarray*}
   Then
   \begin{enumerate}
   \item[(2a)] \label{it:pro:ch5} $(M,\hat{\rho}_{1}, \hat{\varphi}_{1})$ is a representation of the pre-Lie algebra $(A, \bullet)$, where $a \bullet b= R(a)b- b S(a)$.

   \item[(2b)] \label{it:pro:ch6} $(M,\hat{\rho}_{2}, \hat{\varphi}_{2})$ is a representation of the pre-Lie algebra $(A, \bullet')$, where $a \bullet' b= S(a)b- b R(a)$.
   \end{enumerate}
  \end{enumerate}
 \end{pro}

 \begin{proof} {\bf (1a):} By Remark \ref{rmk:ef}, $(A, \star)$ is an associative algebra.
 $(M, \ell, r, \alpha, \beta)$ is a representation of a symmetric Rota-Baxter system $(A, R, S)$ $\stackrel{\small\hbox{Prop. \ref{pro:cff}}}{\Longleftrightarrow}$ $(A\ltimes_{\ell, r}M, R+\alpha, S+\beta)$ is a semi-direct product symmetric Rota-Baxter system $\stackrel{\small\hbox{Rmk. \ref{rmk:ef}}}{\Longrightarrow}$ $(A\oplus M, \star_{A\oplus M})$ is an associative algebra $\stackrel{\small\hbox{Lem. \ref{lem:semidirectpro}}}{\Longleftrightarrow}$ $(M, \hat{\ell}_{1}, \hat{r}_{1})$ is an $(A, \star)$-bimodule.	

  {\bf (1b):}  Similar to the proof of Item  {\bf (1a)}.

 {\bf (2a):} By Remark \ref{rmk:ef}, $(A, \bullet)$ is a pre-Lie algebra.
 $(M, \ell, r, \alpha, \beta)$ is a representation of a symmetric Rota-Baxter system $(A, R, S)$ $\stackrel{\small\hbox{Prop. \ref{pro:cff}}}{\Longleftrightarrow}$ $(A\ltimes_{\ell, r}M, R+\alpha, S+\beta)$ is a semi-direct product symmetric Rota-Baxter system $\stackrel{\small\hbox{Rmk. \ref{rmk:ef}}}{\Longrightarrow}$ $(A\oplus M, \bullet_{A\oplus M})$ is a pre-Lie algebra $\stackrel{\small\hbox{\cite[Prop. 3.1]{Bai2}}}{\Longleftrightarrow}$ $(M,\hat{\rho}_{1}, \hat{\varphi}_{1})$ is a representation of the pre-Lie algebra $(A, \bullet)$.

  {\bf (2b):}  Similar to the proof of Item {\bf (2a)}.
 \end{proof}

 In what follows, we introduce the notion of a dual representation of a representation of a Rota-Baxter system.

 Let $(M, \ell, r)$ be an $(A,\cdot)$-bimodule, define linear maps $\ell^{*}, r^{*}: A\rightarrow \End(M^{*})$ by
 \begin{eqnarray*}
 &\langle \ell^{*}(a) m^{*}, n\rangle=\langle m^{*},\ell(a)n\rangle, \quad \langle m^{*}r^{*}(a), n\rangle=\langle m^{*}, nr(a)\rangle,~~~\forall~ a\in A, m^{*}\in M^{*}, n\in M.&
 \end{eqnarray*}
 Then $(M^{*}, r^{*}, \ell^{*})$ is also an $A$-bimodule, called the \textbf{dual bimodule} of $(M, \ell, r)$. In particular, $(A^{*}, \mathcal{R^{*}}, \mathcal{L^{*}})$ is an $A$-bimodule.

 \begin{lem}\label{lem:cj} Let $(A, R, S)$ be a symmetric Rota-Baxter system, $(M,\ell, r)$ an $A$-bimodule and $\xi, \zeta: M\rightarrow M$ linear maps. The quintuple $(M^{*}, r^{*}, \ell^{*}, \xi^{*}, \zeta^{*})$ is a representation of $(A, R, S)$ if and only if, for all $a\in A, m\in M,$ the following equations hold:
 \begin{eqnarray}
 &\xi(\ell(R(a))m)=\xi(\ell(a)\xi(m))+\ell(S(a))\xi(m)
 =\ell(R(a))\xi(m)+\zeta(\ell(a)\xi(m)),& \label{eq:cj}\\
 &\xi(mr(R(a)))=\xi(\xi(m)r(a))+\xi(m)r(S(a))
 =\xi(m)r(R(a))+\zeta(\xi(m)r(a)),& \label{eq:cj1}\\
 &\zeta(\ell(S(a))m)=\xi(\ell(a)\zeta(m))+\ell(S(a))\zeta(m)
 =\ell(R(a))\zeta(m)+\zeta(\ell(a)\zeta(m)),& \label{eq:cj2}\\
 &\zeta(mr(S(a)))=\xi(\zeta(m)r(a))+\zeta(m)r(S(a))
 =\zeta(m)r(R(a))+\zeta(\zeta(m)r(a)).& \label{eq:cj3}
 \end{eqnarray}
 \end{lem}

 \begin{proof} Since $(M^{*}, r^{*}, \ell^{*})$ is an $A$-bimodule, the lemma can be proved by the following equations:
  {\small\begin{eqnarray*}
 &&\hspace{-6mm}\langle\xi^{*}(m^{*})\ell^{*}(R(a))-\xi^{*}(\xi^{*}(m^{*})\ell^{*}(a)
 +m^{*}\ell^{*}(S(a))), n\rangle=\langle m^{*}, \xi(\ell(R(a))n)-\xi(\ell(a)\xi(n))-\ell(S(a)) \xi(n)\rangle,\\
 &&\hspace{-6mm}\langle\xi^{*}(m^{*})\ell^{*}(R(a))-\xi^{*}(\zeta^{*}(m^{*})\ell^{*}(a)
 +m^{*}\ell^{*}(R(a))), n\rangle=\langle m^{*}, \xi(\ell(R(a))n)-\zeta(\ell(a)\xi(n))-\ell(R(a)) \xi(n)\rangle,\\
 &&\hspace{-6mm}\langle r^{*}(R(a))\xi^{*}(m^{*})-\xi^{*}(r^{*}(R(a))m^{*}+r^{*}(a)\zeta^{*}(m^{*})), n\rangle=\langle m^{*}, \xi(n r(R(a)))-\xi(n)r(R(a))-\zeta(\xi(n)r(a))\rangle,\\
 &&\hspace{-6mm}\langle r^{*}(R(a))\xi^{*}(m^{*})-\xi^{*}(r^{*}(S(a))m^{*}+r^{*}(a)\xi^{*}(m^{*})), n\rangle=\langle m^{*}, \xi(n r(R(a)))-\xi(n)r(S(a))-\xi(\xi(n)r(a))\rangle,\\
 &&\hspace{-6mm}\langle\zeta^{*}(m^{*})\ell^{*}(S(a))-\zeta^{*}(\xi^{*}(m^{*})\ell^{*}(a)
 +m^{*}\ell^{*}(S(a))), n\rangle=\langle m^{*}, \zeta(\ell(S(a))n)-\xi(\ell(a)\zeta(n))-\ell(S(a)) \zeta(n)\rangle,\\
 &&\hspace{-6mm}\langle\zeta^{*}(m^{*})\ell^{*}(S(a))-\zeta^{*}(\zeta^{*}(m^{*})\ell^{*}(a)
 +m^{*}\ell^{*}(R(a))), n\rangle=\langle m^{*}, \zeta(\ell(S(a))n)-\zeta(\ell(a)\zeta(n))-\ell(R(a)) \zeta(n)\rangle,\\
 &&\hspace{-6mm}\langle r^{*}(S(a))\zeta^{*}(m^{*})-\zeta^{*}(r^{*}(R(a))m^{*}+r^{*}(a)\zeta^{*}(m^{*})), n\rangle=\langle m^{*}, \zeta(n r(S(a)))-\zeta(\zeta(n)r(a))-\zeta(n)r(R(a))\rangle,\\
 &&\hspace{-6mm}\langle r^{*}(S(a))\zeta^{*}(m^{*})-\zeta^{*}(r^{*}(S(a))m^{*}+r^{*}(a)\xi^{*}(m^{*})), n\rangle=\langle m^{*}, \zeta(nr(S(a)))-\xi(\zeta(n)r(a))-\zeta(n)r(S(a))\rangle,
 \end{eqnarray*}
 }where $a\in A, n\in M$ and $m^*\in M^*$.
 \end{proof}

 Let $(M, \ell, r)=(A, \mathcal{L}, \mathcal{R})$, then we have
 \begin{cor}\label{cor:ck} Let $(A, R, S)$ be a symmetric Rota-Baxter system and $Q, T: A\rightarrow A$ linear maps. The quintuple $(A^{*}, \mathcal{R}^{*}, \mathcal{L}^{*}, Q^{*}, T^{*})$ is a representation of $(A, R, S)$ if and only if, for all $a, b\in A,$ the following equations hold:
 \begin{eqnarray}
 &&Q(R(a)b)=Q(a Q(b))+S(a)Q(b)=R(a)Q(b)+T(a Q(b)), \label{eq:ck}\\
 &&Q(a R(b))=Q(Q(a)b)+Q(a)S(b)=Q(a)R(b)+T(Q(a)b), \label{eq:ck1}\\
 &&T(S(a)b)=Q(a T(b))+S(a)T(b)=T(a T(b))+R(a)T(b), \label{eq:ck2}\\
 &&T(a S(b))=Q(T(a)b)+T(a)S(b)=T(T(a)b)+T(a)R(b). \label{eq:ck3}
 \end{eqnarray}
 \end{cor}

 \begin{defi}\label{de:cl} With notations in Lemma \ref{lem:cj} and Corollary \ref{cor:ck}, if Eqs.\eqref{eq:cj}-\eqref{eq:cj3} hold, then we say that $(\xi, \zeta)$ is \textbf{admissible to $(A, R, S)$ with respect to $(M, \ell, r)$}. If Eqs.\eqref{eq:ck}-\eqref{eq:ck3} hold, then we say that $(Q, T)$ is \textbf{adjoint admissible to $(A, R, S)$} or \textbf{ $(A, R, S)$ is $(Q, T)$-adjoint admissible}.
 \end{defi}

 \section{Symmetric Rota-Baxter ASI bisystem}\label{se:srbbialg} This section is devoted to the bialgebraic structures on a symmetric Rota-Baxter system. Specifically, the notion of a symmetric Rota-Baxter ASI bisystem is introduced, and equivalent characterizations are established in terms of matched pairs and the double construction.

 \subsection{Matched pair of symmetric Rota-Baxter systems} To begin with, we introduce the notion of a matched pair of symmetric Rota-Baxter systems, which constitutes a generalization of the corresponding concept for ordinary algebras.

 \begin{defi} \cite{Bai1} \label{de:cm} A \textbf{matched pair of algebras $(A, \cdot_{A})$ and $(B, \cdot_{B})$} is a six-tuple $((A, \cdot_{A}), (B, \cdot_{B}),$ $\ell_{A}, r_{A}, \ell_{B}, r_{B})$, where $(A, \ell_{B}, r_{B})$ is a bimodule of $(B, \cdot_{B})$ and $(B, \ell_{A},r_{A})$ is a bimodule of $(A, \cdot_{A})$ such that six compatibility conditions hold.
 \end{defi}

 \begin{pro}\label{pro:cn} {\em \cite{Bai1}} For algebras $(A, \cdot_{A})$, $(B, \cdot_{B})$ and linear maps $\ell_{A}, r_{A} : A \rightarrow \End(B)$, $\ell_{B}, r_{B}:  B \rightarrow \End(A)$, define a multiplication on the direct sum $A \oplus B$ by
 \begin{equation}
 (a + b) \cdot_{A \oplus B} (a' + b'):= (a \cdot_{A} a' + ar_{B}(b') + \ell_{B}(b)a') + (b \cdot_{B} b' + \ell_{A}(a)b' + br_{A}(a')),\label{eq:cn}
 \end{equation}
 for $a, a' \in A$ and $b, b'\in B$. Then $(A\oplus B,\cdot_{A \oplus B})$ is an algebra if and only if $((A, \cdot_{A}), (B, \cdot_{B})$, $\ell_{A}, r_{A}, \ell_{B}, r_{B})$ is a matched pair of $(A, \cdot_{A})$ and $(B, \cdot_{B})$. We denote the resulting algebra $(A \oplus B, \cdot_{A \oplus B})$ by $A \bowtie^{\ell_{B}, r_{B}}_{\ell_{A}, r_{A}} B$ or simply $A \bowtie B$.
 \end{pro}

 \begin{ex}\label{ex:fa} Let $((A, \cdot_{A}), (B, \cdot_{B}),$ $\ell_{A}, r_{A}, \ell_{B}, r_{B})$ be a matched pair of algebras $(A, \cdot_{A})$ and $(B, \cdot_{B})$. Define linear maps $\mathbb{R}, \mathbb{S}: A\oplus B \rightarrow A\oplus B$ by
 \begin{equation*}\label{eq:fa}
 \mathbb{R}(a+b):=a,\ \ \mathbb{S}(a+b):=-b,~~\forall~~a\in A, b\in B.
 \end{equation*}
 Then $(A \bowtie B, \mathbb{R}, \mathbb{S})$ is a symmetric Rota-Baxter system.
 \end{ex}

 We extend the matched pair of algebras to symmetric Rota-Baxter systems.
 \begin{defi}\label{de:co} Let $(A, R_{A}, S_{A})$ and $(B, R_{B},$ $S_{B})$ be two symmetric Rota-Baxter systems. A \textbf{matched pair of $(A, R_{A}, S_{A})$ and $(B, R_{B}, S_{B})$} is a six-tuple $((A, R_{A}, S_{A}), (B, R_{B},$ $S_{B}), \ell_{A}, r_{A},$ $\ell_{B}, r_{B})$, where
 $(A, \ell_{B}, r_{B}, R_{A}, S_{A})$ is a representation of $(B, R_{B}, S_{B})$, $(B, \ell_{A},r_{A}, R_{B}, S_{B})$ is a representation of $(A, R_{A}, S_{A})$ and $((A, \cdot_{A}), (B, \cdot_{B}), \ell_{A},$ $r_{A},$ $\ell_{B}, r_{B})$ is a matched pair of $(A, \cdot_{A})$ and $(B, \cdot_{B})$.
 \end{defi}

 \begin{thm}\label{thm:cp} Let $(A, R_{A}, S_{A})$ and $(B, R_{B}, S_{B})$ be   symmetric Rota-Baxter systems. Then $((A, R_{A}, S_{A})$, $(B, R_{B}, S_{B})$, $\ell_{A}, r_{A}, \ell_{B}, r_{B})$ is a matched pair of $(A, R_{A}, S_{A})$ and $(B, R_{B}, S_{B})$ if and only if $(A \bowtie B, R_{A} + R_{B},  S_{A} + S_{B})$ is a symmetric Rota-Baxter system by defining the multiplication on $A \oplus B$ by Eq.\eqref{eq:cn} and linear maps $R_{A} + R_{B},  S_{A} + S_{B} : A \oplus B \rightarrow A \oplus B$ by
 \begin{eqnarray*}
 &(R_{A} + R_{B})(a + b) := R_{A}(a) + R_{B}(b),~~~
 (S_{A} + S_{B})(a + b) := S_{A}(a) + S_{B}(b),~~ \forall~a\in A, b\in B.&
 \end{eqnarray*}
 \end{thm}

 \begin{proof} By Proposition \ref{pro:cn}, the proof is finished by comparing the following equations:
 {\small\begin{eqnarray*}
 &&\hspace{-13mm}(R_{A} + R_{B})(a + b)\cdot_{A \oplus B} (R_{A} + R_{B})(a' + b')-(R_{A} + R_{B})((R_{A} + R_{B})(a + b)\cdot_{A \oplus B} (a' + b')\\
 &&+(a + b)\cdot_{A \oplus B} (S_{A} + S_{B})(a' + b'))\\
 &=& R_{A}(a) \cdot_{A} R_{A}(a')-R_{A}(R_{A}(a) \cdot_{A} a' +a \cdot_{A} S_{A}(a'))+ R_{B}(b) \cdot_{B} R_{B}(b')- R_{B}(R_{B}(b) \cdot_{B} b'\\
 &&+ b \cdot_{B} S_{B}(b')) + R_{A}(a)r_{B}(R_{B}(b'))- R_{A}(R_{A}(a)r_{B}(b')+a r_{B}(S_{B}(b')) )+ \ell_{B}(R_{B}(b))R_{A}(a')\\
 &&- R_{A}(\ell_{B}(R_{B}(b))a'+ \ell_{B}(b)S_{A}(a'))  + \ell_{A}(R_{A}(a))R_{B}(b')-R_{B}(\ell_{A}(R_{A}(a))b' + \ell_{A}(a)S_{B}(b'))\\
 &&+ R_{B}(b)r_{A}(R_{A}(a'))-R_{B}(R_{B}(b)r_{A}(a')+ b r_{A}(S_{A}(a'))).\\
 &&\hspace{-13mm}(R_{A} + R_{B})(a + b)\cdot_{A \oplus B} (R_{A} + R_{B})(a' + b')-(R_{A} + R_{B})((S_{A} + S_{B})(a + b)\cdot_{A \oplus B} (a' + b')\\
 &&+(a + b)\cdot_{A \oplus B} (R_{A} + R_{B})(a' + b'))\\
 &=& R_{A}(a) \cdot_{A} R_{A}(a')-R_{A}(S_{A}(a) \cdot_{A} a' +a \cdot_{A} R_{A}(a'))+ R_{B}(b) \cdot_{B} R_{B}(b')- R_{B}(S_{B}(b) \cdot_{B} b'\\
 &&+ b \cdot_{B} R_{B}(b')) + R_{A}(a)r_{B}(R_{B}(b'))- R_{A}(S_{A}(a)r_{B}(b')+a r_{B}(R_{B}(b')) )+ \ell_{B}(R_{B}(b))R_{A}(a')\\
 &&- R_{A}(\ell_{B}(S_{B}(b))a'+ \ell_{B}(b)R_{A}(a'))  + \ell_{A}(R_{A}(a))R_{B}(b')-R_{B}(\ell_{A}(S_{A}(a))b' + \ell_{A}(a)R_{B}(b'))\\
 &&+ R_{B}(b)r_{A}(R_{A}(a'))-R_{B}(S_{B}(b)r_{A}(a')+ b r_{A}(R_{A}(a'))).\\
 &&\hspace{-13mm}(S_{A} + S_{B})(a + b)\cdot_{A \oplus B} (S_{A} + S_{B})(a' + b')-(S_{A} + S_{B})((R_{A} + R_{B})(a + b)\cdot_{A \oplus B} (a' + b')\\
 &&+(a + b)\cdot_{A \oplus B} (S_{A} + S_{B})(a' + b'))\\
 &=& S_{A}(a) \cdot_{A} S_{A}(a')-S_{A}(R_{A}(a) \cdot_{A} a' +a \cdot_{A} S_{A}(a'))+ S_{B}(b) \cdot_{B} S_{B}(b')- S_{B}(R_{B}(b) \cdot_{B} b'\\
 &&+ b \cdot_{B} S_{B}(b')) + S_{A}(a)r_{B}(S_{B}(b'))- S_{A}(R_{A}(a)r_{B}(b')+a r_{B}(S_{B}(b')) )+ \ell_{B}(S_{B}(b))S_{A}(a')\\
 &&- S_{A}(\ell_{B}(R_{B}(b))a'+ \ell_{B}(b)S_{A}(a')) + \ell_{A}(S_{A}(a))S_{B}(b')-S_{B}(\ell_{A}(R_{A}(a))b' + \ell_{A}(a)S_{B}(b'))\\
 &&+ S_{B}(b)r_{A}(S_{A}(a'))-S_{B}(R_{B}(b)r_{A}(a')+ b r_{A}(S_{A}(a'))).\\
 &&\hspace{-13mm}(S_{A} + S_{B})(a + b)\cdot_{A \oplus B} (S_{A} + S_{B})(a' + b')-(S_{A} + S_{B})((S_{A} + S_{B})(a + b)\cdot_{A \oplus B} (a' + b')\\
 &&+(a + b)\cdot_{A \oplus B} (R_{A} + R_{B})(a' + b'))\\
 &=& S_{A}(a) \cdot_{A} S_{A}(a')-S_{A}(S_{A}(a) \cdot_{A} a' +a \cdot_{A} R_{A}(a'))+ S_{B}(b) \cdot_{B} S_{B}(b')- S_{B}(S_{B}(b) \cdot_{B} b'\\
 &&+ b \cdot_{B} R_{B}(b')) + S_{A}(a)r_{B}(S_{B}(b'))- S_{A}(S_{A}(a)r_{B}(b')+a r_{B}(R_{B}(b')) )+ \ell_{B}(S_{B}(b))S_{A}(a')\\
 &&- S_{A}(\ell_{B}(S_{B}(b))a'+ \ell_{B}(b)R_{A}(a'))  + \ell_{A}(S_{A}(a))S_{B}(b')-S_{B}(\ell_{A}(S_{A}(a))b' + \ell_{A}(a)R_{B}(b'))\\
 &&+ S_{B}(b)r_{A}(S_{A}(a'))-S_{B}(S_{B}(b)r_{A}(a')+ b r_{A}(R_{A}(a'))),
  \end{eqnarray*}
 }where $a, a'\in A$, $b, b'\in B$.
 \end{proof}

 \subsection{Double construction of symmetric Rota-Baxter Frobenius system}
 \begin{defi}\label{de:cq}\cite{Bai1} A \textbf{Frobenius algebra $(A, \mathcal{B})$} is an algebra $A$ with a nondegenerate invariant bilinear form $\mathcal{B}$, which means for all $a, b, c \in A$, $\mathcal{B}(a b, c) = \mathcal{B}(a, b c)$. A Frobenius algebra $(A, \mathcal{B})$ is called {\bf symmetric} if $\mathcal{B}$ is symmetric.

 Let $(A, \cdot_{A})$ be an algebra. Suppose that there is an algebra structure $\cdot_{A^{*}}$ on its dual space $A^{*}$, and an algebra structure on the direct sum $A \oplus A^{*}$ of the underlying vector spaces of $A$ and $A^{*}$ which contains both $(A, \cdot_{A})$ and $(A^{*}, \cdot_{A^{*}})$ as subalgebras. Define a bilinear form on $A \oplus A^{*}$ by
 \begin{eqnarray}
 \mathcal{B}_{d} (x + a^{*}, y + b^{*}) = \langle x, b^{*} \rangle + \langle a^{*}, y\rangle, \label{eq:cq}
 \end{eqnarray}
 for all $a^{*}, b^{*}\in A^{*}, x, y \in A$. If $\mathcal{B}_{d}$ is invariant, then $(A \oplus A^{*}, \mathcal{B}_{d})$ is a symmetric Frobenius algebra. This Frobenius algebra is called \textbf{a double construction of Frobenius algebra associated to $(A, \cdot_{A})$ and $(A^{*}, \cdot_{A^{*}})$} and denoted by $(A \bowtie A^{*}, \mathcal{B}_{d})$.
 \end{defi}

 \begin{defi} \label{de:cqq} A \textbf{symmetric Rota-Baxter Frobenius system} is a quadruple $(A, R, S, \mathcal{B})$ where $(A, R, S)$ is a symmetric Rota-Baxter system and $(A, \mathcal{B})$ is a Frobenius algebra.
 \end{defi}

 Assume that $\widehat{R}, \widehat{S}$ are the adjoint linear maps of $R, S$ with respect to $\mathcal{B}$ respectively, characterized
 by
 \begin{eqnarray}
 \mathcal{B}(R(a), b) = \mathcal{B}(a, \widehat{R}(b)),\quad \mathcal{B}(S(a), b) = \mathcal{B}(a, \widehat{S}(b)), \label{eq:cqq}
 \end{eqnarray}
 for all $a, b\in A$.

 The symmetric Frobenius property of a symmetric Rota-Baxter system $(A, R, S)$ ensures that it naturally induces a representation on the dual space $A^*$.

 \begin{pro}\label{pro:cr} Let $(A, R, S, \mathcal{B})$ be a symmetric Rota-Baxter Frobenius system with $\mathcal{B}$ symmetric. Then for the adjoint operators $\widehat{R}, \widehat{S}$ in Eq.\eqref{eq:cqq}, $(\widehat{R}, \widehat{S})$ is adjoint admissible to $(A, R, S)$, or equivalently, $(A^{*}, \mathcal{R}^{*}, \mathcal{L}^{*}, \widehat{R}^{*}, \widehat{S}^{*})$ is a representation of $(A, R, S)$. Moreover, $(A^{*}, \mathcal{R}^{*}, \mathcal{L}^{*}, \widehat{R}^{*}, \widehat{S}^{*})$ is equivalent to $(A, \mathcal{L}, \mathcal{R}, R, S)$ as representations of $(A, R, S)$.

 Conversely, let $(A, R, S)$ be a symmetric Rota-Baxter system and
 $Q, T : A \rightarrow A$ two linear maps such that $(Q, T)$ is adjoint admissible to $(A, R, S)$. If the resulting representation $(A^{*}, \mathcal{R}^{*}, \mathcal{L}^{*}$, $Q^{*}, T^{*})$ of $(A, R, S)$ is equivalent to $(A, \mathcal{L}, \mathcal{R}, R, S)$, then there exists a symmetric Rota-Baxter Frobenius system $(A, R, S, \mathcal{B})$  such that $Q=\widehat{R}, T=\widehat{S}$.
 \end{pro}

 \begin{proof} For all $a, b, c\in A$, since $(A, \mathcal{B})$ is a symmetric Frobenius algebra, we have $\mathcal{B}(a\cdot b, c)=\mathcal{B}(b, c\cdot a)$. Then
  \begin{eqnarray*}
  0&=&\mathcal{B}(R(a)R(b)-R(R(a)b+aS(b)), c)\\
  &=&\mathcal{B}(R(a), R(b)c)-\mathcal{B}(R(a)b, \widehat{R}(c)) -\mathcal{B}(aS(b), \widehat{R}(c))\\
  &=&\mathcal{B}(a, \widehat{R}(R(b)c)-\widehat{R}(b\widehat{R}(c))- S(b)\widehat{R}(c))
  \end{eqnarray*}
  and
  \begin{eqnarray*}
  0&=&\mathcal{B}(R(a)R(b)-R(S(a)b+aR(b)), c)=\mathcal{B}(a, \widehat{R}(R(b)c)-\widehat{S}(b\widehat{R}(c))- R(b)\widehat{R}(c)),
  \end{eqnarray*}
  produce $\widehat{R}(R(b)c)=\widehat{R}(b\widehat{R}(c))+S(b)\widehat{R}(c)=\widehat{S}(b\widehat{R}(c))+ R(b)\widehat{R}(c)$. One gets Eq.(\ref{eq:ck}). Similarly, we can obtain Eqs.(\ref{eq:ck1})-(\ref{eq:ck3}). Hence $(A^{*}, \mathcal{R}^{*}, \mathcal{L}^{*}, \widehat{R}^{*}, \widehat{S}^{*})$ is a representation of $(A, R, S)$.
 Then we can finish the proof by by an argument similar to that in \cite[Proposition 3.9]{BGM}.
 \end{proof}

 \begin{ex} \label{ex:fb} Let $(A \bowtie A^{*}, \mathbb{R}, \mathbb{S})$ be the symmetric Rota-Baxter system given in Example \ref{ex:fa}. Assume that $(A \bowtie A^{*}, \mathcal{B}_{d})$, where $\mathcal{B}_{d}$ is given by Eq.\eqref{eq:cq}, is a symmetric Frobenius algebra. Then $\widehat{\mathbb{R}}(x+a^{*})=a^{*}=-\mathbb{S}(x+a^{*}),~~ \widehat{\mathbb{S}}(x+a^{*})=-x=-\mathbb{R}(x+a^{*})$. By Proposition \ref{pro:cr}, $(-\mathbb{S}, -\mathbb{R})$ is adjoint admissible to $(A \bowtie A^{*}, \mathbb{R}, \mathbb{S})$, or equivalently, $(A\oplus  A^{*}, \mathcal{L}_{A \bowtie A^{*}}, \mathcal{R}_{A \bowtie A^{*}}, \mathbb{R}, \mathbb{S})$ is a representation of $(A \bowtie A^{*}, \mathbb{R}, \mathbb{S})$, which is exactly the adjoint representation.
 \end{ex}

 \begin{defi}\label{de:cs} Let $(A, R, S)$ and $(A^{*}, Q^{*}, T^{*})$ be two symmetric Rota-Baxter systems. A \textbf{double construction of symmetric Rota-Baxter Frobenius system associated to $(A, R, S)$ and $(A^{*}, Q^{*}, T^{*})$} is a double construction $(A \bowtie A^{*}, \mathcal{B}_{d})$ of Frobenius algebra associated to $(A, \cdot_{A})$ and $(A^{*}, \cdot_{A^{*}})$ such that $(A \bowtie A^{*}, R+Q^{*}, S+T^{*}, \mathcal{B}_{d})$ is a symmetric Rota-Baxter Frobenius system.
 \end{defi}

 \begin{lem} \label{lem:ct} Let $(A \bowtie A^{*}, R+Q^{*}, S+T^{*}, \mathcal{B}_{d})$ be a double construction of symmetric Rota-Baxter Frobenius system associated to $(A, R, S)$ and $(A^{*}, Q^{*}, T^{*})$.
 \begin{enumerate}[(1)]
   \item \label{it:lem:ct1} The adjoint $\widehat{R+Q^{*}}$ of $R+Q^{*}$ with respect to $\mathcal{B}_{d}$ is $Q+R^{*}$, the adjoint $\widehat{S+T^{*}}$ of $S+T^{*}$ with respect to $\mathcal{B}_{d}$ is $T+S^{*}$. Further $(Q+R^{*}, T+S^{*})$ is adjoint admissible to $(A \bowtie A^{*}, R+Q^{*}, S+T^{*})$.

   \item \label{it:lem:ct2} $(Q, T)$ is adjoint admissible to $(A, R, S)$.

   \item \label{it:lem:ct3} $(R^{*}, S^{*})$ is adjoint admissible to $(A^{*}, Q^{*}, T^{*})$.
 \end{enumerate}
 \end{lem}

 \begin{proof} \ref{it:lem:ct1} For all $x, y \in A$, $a^{*}, b^{*}\in A^{*}$, we have
 \begin{eqnarray*}
 \mathcal{B}_{d}((R+Q^{*})(x+a^{*}), y+b^{*})
 \hspace{-3mm}&=&\hspace{-3mm}\mathcal{B}_{d}(R(x)+Q^{*}(a^{*}), y+b^{*})=\langle Q^{*}(a^{*}), y\rangle+\langle R(x), b^{*}\rangle\\
 \hspace{-3mm}&=&\hspace{-3mm}\langle a^{*}, Q(y)\rangle+\langle x, R^{*}(b^{*})\rangle=\mathcal{B}_{d}(x+a^{*}, Q(y)+R^{*}(b^{*}))\\
 \hspace{-3mm}&=&\hspace{-3mm}\mathcal{B}_{d}(x+a^{*}, (Q+R^{*})(y+b^{*})).
 \end{eqnarray*}
 Hence, $\widehat{R+Q^{*}}=Q+R^{*}$ and similarly $\widehat{S+T^{*}}=T+S^{*}$. Furthermore, by Proposition $\ref{pro:cr}$, $(Q+R^{*}, T+S^{*})$ is adjoint admissible to $(A \bowtie A^{*},R+Q^{*}, S+T^{*})$.
	
 \ref{it:lem:ct2} By Item \ref{it:lem:ct1}, $(Q+R^{*}, T+S^{*})$ is adjoint admissible to $(A \bowtie A^{*}, R+Q^{*}, S+T^{*})$, that is to say, for all $x+a^{*}, y+b^{*}\in A \oplus A^{*}$, Eqs.(\ref{eq:ck})-(\ref{eq:ck3}) hold for $Q+R^{*}, T+S^{*}, R+Q^{*}, S+T^{*}$. Taking $a^{*}=b^{*}=0$ in the equations above give the admissibility of $Q, T$ to $(A, R, S)$.

 \ref{it:lem:ct3} Similar to Item \ref{it:lem:ct2} by taking $x=y=0$.
 \end{proof}

 \begin{lem}\label{lem:ctt}{\em \cite{Bai1}} Let $(A, \cdot_{A})$ and $(A^{*}, \cdot_{A^{*}})$ be algebras. Then there is a double construction of Frobenius algebra $(A \bowtie A^{*}, \mathcal{B}_{d})$ associated to $(A, \cdot_{A})$ and $(A^{*}, \cdot_{A^{*}})$ if and only if $((A, \cdot_{A}), (A^{*}, \cdot_{A^{*}}), \mathcal{R}_{A}^{*},  \mathcal{L}_{A}^{*},$ $\mathcal{R}_{A^{*}}^{*}, \mathcal{L}_{A^{*}}^{*})$ is a matched pair of algebras.
 \end{lem}

 Extending Lemma \ref{lem:ctt} to symmetric Rota-Baxter systems, we obtain

 \begin{thm}\label{thm:cttt} Let $(A, R, S)$ and $(A^{*}, Q^{*}, T^{*})$ be symmetric Rota-Baxter systems. Then there is a double construction of symmetric Rota-Baxter Frobenius system $(A \bowtie A^{*}, R+Q^{*}, S+T^{*}, \mathcal{B}_{d})$ associated to $(A, R, S)$ and $(A^{*}, Q^{*}, T^{*})$ if and only if $((A, R, S), (A^{*}, Q^{*}, T^{*}), \mathcal{R}_{A}^{*},  \mathcal{L}_{A}^{*},$ $\mathcal{R}_{A^{*}}^{*}, \mathcal{L}_{A^{*}}^{*})$ is a matched pair of $(A, R, S)$ and $(A^{*}, Q^{*}, T^{*})$.
 \end{thm}

 \begin{proof} ($\Rightarrow$) If $(A \bowtie A^{*}, R+Q^{*}, S+T^{*}, \mathcal{B}_{d})$ is a double construction of symmetric Rota-Baxter Frobenius system associated to $(A, R, S)$ and $(A^{*}, Q^{*}, T^{*})$, then by Definition $\ref{de:cs}$, we know that there is a double construction of Frobenius algebra $(A \bowtie A^{*}, \mathcal{B}_{d})$ associated to $(A, \cdot_{A})$ and $(A^{*}, \cdot_{A^{*}})$. By Lemma \ref{lem:ctt}, $((A, \cdot_{A}), (A^{*}, \cdot_{A^{*}}),$ $\mathcal{R}_{A}^{*},  \mathcal{L}_{A}^{*}, \mathcal{R}_{A^{*}}^{*}, \mathcal{L}_{A^{*}}^{*})$ is a matched pair of algebras $(A, \cdot_{A})$ and $(A^{*}, \cdot_{A^{*}})$. Furthermore, by Lemma $\ref{lem:ct}$, $(A^{*}, \mathcal{R}_{A}^{*}, \mathcal{L}_{A}^{*}, Q^{*}, T^{*})$ is a representation of $(A, R, S)$ and $(A, \mathcal{R}_{A^{*}}^{*}, \mathcal{L}_{A^{*}}^{*}, R, S)$ is a representation of $(A^{*}, Q{^{*}}, T^{*})$.  Therefore, $((A, R, S)$, $(A^{*}$, $Q^{*}, T^{*})$, $\mathcal{R}_{A}^{*},  \mathcal{L}_{A}^{*}, \mathcal{R}_{A^{*}}^{*}, \mathcal{L}_{A^{*}}^{*})$ is a matched pair of $(A, R, S)$ and $(A^{*}, Q^{*}, T^{*})$.

 ($\Leftarrow$) If $((A, R, S), (A^{*}, Q^{*}, T^{*}), \mathcal{R}_{A}^{*},  \mathcal{L}_{A}^{*}, \mathcal{R}_{A^{*}}^{*}, \mathcal{L}_{A^{*}}^{*})$ is a matched pair of $(A, R, S)$ and $(A^{*}$, $Q^{*}, T^{*})$, then by Definition $\ref{de:co}$, $((A, \cdot_{A}), (A^{*}, \cdot_{A^{*}}),$ $\mathcal{R}_{A}^{*},  \mathcal{L}_{A}^{*}, \mathcal{R}_{A^{*}}^{*}, \mathcal{L}_{A^{*}}^{*})$ is a matched pair of $(A, \cdot_{A})$ and $(A^{*}, \cdot_{A^{*}})$. By Lemma \ref{lem:ctt}, there is a double construction of Frobenius algebra $(A \bowtie A^{*}, \mathcal{B}_{d})$ associated to $(A, \cdot_{A})$ and $(A^{*}, \cdot_{A^{*}})$. Furthermore, by Theorem $\ref{thm:cp}$, $(A\bowtie A^{*}, R+Q^{*}, S+T^{*})$ is a symmetric Rota-Baxter system.  Hence, there is a double construction of symmetric Rota-Baxter Frobenius system $(A \bowtie A^{*}, R+Q^{*}, S+T^{*}, \mathcal{B}_{d})$ associated to $(A, R, S)$ and $(A^{*}, Q^{*}, T^{*})$.
 \end{proof}

\subsection{Symmetric Rota-Baxter ASI bisystems}
 Dually, we present the notion of a symmetric Rota-Baxter cosystem based on a Rota-Baxter cosystem given in \cite{MMS}.

 \begin{defi}\label{de:cu} A triple $(C, Q, T)$ consisting of a coalgebra $C$ and two linear maps $Q, T : C \rightarrow C$ is called a \textbf{symmetric Rota-Baxter cosystem} if, for all $c \in C$,
 \begin{eqnarray}
 Q(c_{(1)}) \otimes Q(c_{(2)}) &=& Q(Q(c)_{(1)}) \otimes Q(c)_{(2)} + Q(c)_{(1)} \otimes T(Q(c)_{(2)})\nonumber\\
 &=& T(Q(c)_{(1)}) \otimes Q(c)_{(2)} + Q(c)_{(1)} \otimes Q(Q(c)_{(2)}), \label{eq:cu}\\
 T(c_{(1)}) \otimes T(c_{(2)}) &=& Q(T(c)_{(1)}) \otimes T(c)_{(2)} + T(c)_{(1)} \otimes T(T(c)_{(2)})\nonumber\\
 &=& T(T(c)_{(1)}) \otimes T(c)_{(2)} + T(c)_{(1)} \otimes Q(T(c)_{(2)}). \label{eq:cu1}
 \end{eqnarray}
 Let $(C, Q_{C}, T_{C})$ and $(D, Q_{D}, T_{D})$ be symmetric Rota-Baxter cosystems. A morphism between $(C, Q_{C}, T_{C})$ and $(D, Q_{D}, T_{D})$ is a coalgebra map $f: C \rightarrow D$ such that $f \circ Q_{C} = Q_{D} \circ f$ and $f \circ T_{C} = T_{D} \circ f$.
 \end{defi}

 \begin{rmk} \label{rmk:gb}
 \begin{enumerate}[(1)]
 \item The symmetry of the maps $Q$ and $T$ in Definition \ref{de:cu} means that the conditions are unchanged upon interchanging them.

 \item $(C, Q, Q+\lambda \id_{C})$ is a symmetric Rota-Baxter cosystem if and only if $(C, Q)$ is a Rota-Baxter coalgebra of weight $\lambda$.
  \end{enumerate}
 \end{rmk}

 \begin{ex}\label{ex:cuu} Let $(C, \Delta)$ be a $2$-dimensional coalgebra with basis $e, f$, whose coproduct is given by $\Delta(e)=-e\otimes e, \Delta(f)=-e\otimes f$.
 The following choices of $Q,T$ give symmetric Rota--Baxter cosystems.
 \begin{enumerate}[(1)]
 \item $Q(e)=q_{1}e+q_{2}f,~~~Q(f)=q_{3}e+\frac{q_{2}q_{3}}{q_{1}}f,\\ T(e)=-\frac{q_{2}q_{3}}{q_{1}}e+q_{2}f,~~~T(f)=q_{3}e-q_{1}f~(q_{1}\neq 0).$ 
 \item $Q(e)=0,~~~ Q(f)=q_{1}e+q_{2}f,\quad T(e)=-q_{2}e,~~~ T(f)=q_{1}e.$
 \item $Q(e)=q_{1}f,~~~Q(f)=q_{2}f,\quad T(e)=-q_{2}e+q_{1}f,~~~ T(f)=0 ~(q_{1}\neq 0).$
 \item $Q(e)=0,~~~ Q(f)=q_{1}e,\quad T(e)=0,~~~ T(f)=q_{2}e~(q_{2}\neq q_{1}).$
 \item $Q=0,\quad T(e)=q_{1}e,~~~ T(f)=q_{2}e~(q_{1}\neq 0).$
 \item $Q(e)=q_{1}e,~~~ Q(f)=q_{2}e,\quad T=0~(q_{1}\neq 0).$
 \item $Q=0,\quad T(e)=q_{1}e,~~~ T(f)=q_{1}f~(q_{1}\neq 0).$
 \item $Q(e)=q_{1}e,~~~ Q(f)=q_{1}f,\quad T=0~(q_{1}\neq 0),$
 \end{enumerate}
 where $q_{i}, i=1, 2, 3$ are parameters.
 \end{ex}

 \begin{rmk}
 \begin{enumerate}[(1)]
   \item If $C$ is a finite dimensional vector space, then $((C, \D), Q, T )$ is a symmetric Rota-Baxter cosystem if and only if $((C^{*}, \D^*), Q^{*}, T^{*})$ is a symmetric Rota-Baxter system.

   \item Moreover, for linear maps $R, S : A \rightarrow A$, the condition that $(R^{*}, S^{*})$ is adjoint admissible to the symmetric Rota-Baxter system $(A^{*}, Q^{*}, T^{*})$ can be rewritten in terms of $\Delta$ as
 \begin{eqnarray}
 &&Q(R(x)_{(1)})\otimes R(x)_{(2)}=R(x)_{(1)}\otimes R(R(x)_{(2)})+T(x_{(1)})\otimes R(x_{(2)})\nonumber\\
 &&\hspace{34mm}=S(x)_{(1)}\otimes R(S(x)_{(2)})+ Q(x_{(1)})\otimes R(x_{(2)}),\label{eq:ck5}\\
 &&R(x)_{(1)}\otimes Q(R(x)_{(2)})=R(S(x)_{(1)})\otimes S(x)_{(2)}+R(x_{(1)})\otimes Q(x_{(2)})\nonumber\\
 &&\hspace{34mm}=R(R(x)_{(1)})\otimes R(x)_{(2)}+ R(x_{(1)})\otimes T(x_{(2)}),\label{eq:ck6}\\
 &&T(S(x)_{(1)})\otimes S(x)_{(2)}=R(x)_{(1)}\otimes S(R(x)_{(2)})+T(x_{(1)})\otimes S(x_{(2)})\nonumber\\
 &&\hspace{34mm}=S(x)_{(1)}\otimes S(S(x)_{(2)})+ Q(x_{(1)})\otimes S(x_{(2)}),\label{eq:ck7}\\
 &&S(x)_{(1)}\otimes T(S(x)_{(2)})=S(S(x)_{(1)})\otimes S(x)_{(2)}+S(x_{(1)})\otimes Q(x_{(2)})\nonumber\\
 &&\hspace{34mm}=S(R(x)_{(1)})\otimes R(x)_{(2)}+ S(x_{(1)})\otimes T(x_{(2)}).\label{eq:ck8}
 \end{eqnarray}
 \end{enumerate}
 \end{rmk}

 \begin{defi} \label{de:cv}\cite{Zhe,Bai1,Ag1} An \textbf{ASI bialgebra} is a triple $(A, \cdot, \Delta)$, where the pair $(A, \cdot)$ is an algebra, and the pair $(A, \Delta)$ is a coalgebra, such that, for all $a, b\in A$,
 \begin{eqnarray*}
 &\D(a b)=a_{(1)} b\otimes a_{(2)}+b_{(1)}\otimes a b_{(2)},&\\
 &a_{(1)}\otimes a_{(2)} b + b_{(2)} a\otimes b_{(1)}=b a_{(1)}\otimes a_{(2)}+b_{(2)}\otimes a b_{(1)}.&
 \end{eqnarray*}
 \end{defi}

 \begin{defi}\label{de:cx} A \textbf{symmetric Rota-Baxter ASI bisystem} is a quintuple $((A, \cdot, \Delta), R, S, Q, T)$ ($(A, R, S, Q, T)$ for short), where $(A, \cdot, \Delta)$ is an ASI bialgebra, $(A, R, S)$ is a symmetric Rota-Baxter system and $(A, Q, T)$ is a symmetric Rota-Baxter cosystem such that Eqs.\eqref{eq:ck}-\eqref{eq:ck3} and Eqs.\eqref{eq:ck5}-\eqref{eq:ck8} hold.
 \end{defi}

 \begin{rmk} \label{rmk:fc} $(A, R, S,$ $Q, T)$ is a symmetric Rota-Baxter ASI bisystem if and only if $(A, S, R,$ $T, Q)$ is a symmetric Rota-Baxter ASI bisystem.
 \end{rmk}

 \begin{ex}\label{ex:gc} Let $(A, \cdot, \Delta)$ be a $2$-dimensional ASI bialgebra given in \cite[Example 3.14]{LM} whose product and coproduct are presented in Example \ref{ex:cee} and Example \ref{ex:cuu}, respectively. Define linear maps $R, S:A\rightarrow A$ by
 \begin{eqnarray*}
 &R(e)=q_{1}e-q_{1}f,~~~ R(f)=q_{2}e-q_{2}f,&\\
 &S(e)=q_{2}e-q_{1}f,~~~ S(f)=q_{2}e-q_{1}f,~~~ (q_{1}\neq 0, q_{2}\neq 0).&
 \end{eqnarray*}
 Then $(A, R, S, -S, -R)$ is a symmetric Rota-Baxter ASI bisystem. 
 \end{ex}

 \begin{lem}\label{lem:cy}{\em \cite{Bai1}} Let $(A, \cdot_{A})$ and $(A^{*}, \D^{*})$ be algebras, where $\Delta : A \rightarrow A \otimes A$ is a linear map. Then the following conditions are equivalent.
 \begin{enumerate}[(1)]
 \item \label{it:lem:cy1} The sex-tuple $((A, \cdot_{A}), (A^{*}, \cdot_{A^{*}}), \mathcal{R}_{A}^{*},  \mathcal{L}_{A}^{*}, \mathcal{R}_{A^{*}}^{*}, \mathcal{L}_{A^{*}}^{*})$ is a matched pair of $(A, \cdot_{A})$ and $(A^{*},$ $\cdot_{A^{*}})$.

 \item \label{it:lem:cy2} There is a double construction of Frobenius algebra associated to $(A, \cdot_{A})$ and $(A^{*}, \cdot_{A^{*}})$.

 \item \label{it:lem:cy3} The triple $(A, \cdot_{A}, \Delta)$ is an ASI bialgebra.
 \end{enumerate}
 \end{lem}

 \begin{thm}\label{thm:cz} Let $(A, R, S)$ and $(A^{*}, Q^{*}, T^{*})$ be symmetric Rota-Baxter systems. Then $(A, R, S, Q, T)$ is a symmetric Rota-Baxter ASI bisystem if and only if the sex-tuple $((A, R, S)$, $(A^{*}, Q^{*}, T^{*}), \mathcal{R}_{A}^{*},  \mathcal{L}_{A}^{*}, \mathcal{R}_{A^{*}}^{*}, \mathcal{L}_{A^{*}}^{*})$ is a matched pair of $(A, R, S)$ and $(A^{*}, Q^{*}, T^{*})$.
 \end{thm}

 \begin{proof} ($\Rightarrow$) If $(A, R, S, (Q, T))$ is a symmetric Rota-Baxter ASI bisystem, then by Definition $\ref{de:cx}$, we know that $(A, \cdot, \Delta)$ is an ASI bialgebra and $(Q, T)$, $(R^{*}, S^{*})$ are adjoint admissible $(A, R, S)$ and $(A^{*}, Q^{*}, T^{*})$ respectively. It means that $(A^{*}, \mathcal{R}_{A}^{*}, \mathcal{L}_{A}^{*}, Q^{*}, T^{*})$ is a representation of $(A, R, S)$ and $(A, \mathcal{R}_{A^{*}}^{*}, \mathcal{L}_{A^{*}}^{*}, R, S)$ is a representation of $(A^{*}, Q^{*}, T^{*})$. By Lemma \ref{lem:cy}, $((A, \cdot_{A}), (A^{*}, \cdot_{A^{*}}), \mathcal{R}_{A}^{*}, \mathcal{L}_{A}^{*}, \mathcal{R}_{A^{*}}^{*}, \mathcal{L}_{A^{*}}^{*})$ is a matched pair of $(A, \cdot_{A})$ and $(A^{*}, \cdot_{A^{*}})$ since $(A, \cdot, \Delta)$ is an ASI bialgebra. Therefore, $((A, R, S), (A^{*}, Q^{*}, T^{*}), \mathcal{R}_{A}^{*},  \mathcal{L}_{A}^{*}, \mathcal{R}_{A^{*}}^{*}, \mathcal{L}_{A^{*}}^{*})$ is a matched pair of $(A, R, S)$ and $(A^{*}, Q^{*}, T^{*})$.
	
 ($\Leftarrow$) If $((A, R, S), (A^{*}, Q^{*}, T^{*}), \mathcal{R}_{A}^{*},  \mathcal{L}_{A}^{*}, \mathcal{R}_{A^{*}}^{*}, \mathcal{L}_{A^{*}}^{*})$ is a matched pair of $(A, R, S)$ and $(A^{*}$, $Q^{*}, T^{*})$, then by Definition $\ref{de:co}$, $((A, \cdot_{A}), (A^{*}, \cdot_{A^{*}}), \mathcal{R}_{A}^{*}, \mathcal{L}_{A}^{*}, \mathcal{R}_{A^{*}}^{*}, \mathcal{L}_{A^{*}}^{*})$ is a matched pair of $(A, \cdot_{A})$ and $(A^{*}, \cdot_{A^{*}})$. By Lemma \ref{lem:cy}, $(A, \cdot, \Delta)$ is an ASI bialgebra. By Definition $\ref{de:co}$, $(Q,T)$, $(R^{*}, S^{*})$ are adjoint admissible $(A, R, S)$ and $(A^{*}, Q^{*}, T^{*})$, respectively. Hence, $(A, R, S, Q, T)$ is a symmetric Rota-Baxter ASI bisystem.
 \end{proof}

 Combining Theorems \ref{thm:cttt} and \ref{thm:cz}, we have

 \begin{thm}\label{thm:da} Let $(A, R, S)$ and $(A^{*}, Q^{*}, T^{*})$ be symmetric Rota-Baxter systems. Then the following conditions are equivalent.
 \begin{enumerate}[(1)]

   \item \label{it:thm:da1} The sex-tuple $((A, R, S), (A^{*}, Q^{*}, T^{*}), \mathcal{R}_{A}^{*}, \mathcal{L}_{A}^{*}, \mathcal{R}_{A^{*}}^{*}, \mathcal{L}_{A^{*}}^{*})$ is a matched pair of $(A, R, S)$ and $(A^{*}, Q^{*}, T^{*})$.

   \item \label{it:thm:da2} There is double construction of symmetric Rota-Baxter Frobenius system associated to $(A, R, S)$ and $(A^{*}, Q^{*}, T^{*})$.

   \item \label{it:thm:da3}  $(A, R, S, Q, T)$ is a symmetric Rota-Baxter ASI bisystem.
 \end{enumerate}
 \end{thm}

 \section{Admissible associative Yang-Baxter equations and $\mathcal{O}$-operators} \label{se:aybe}
 In this section, we introduce the notions of the admissible aYBe and $\mathcal{O}$-operators. Building upon these, we present several constructions of symmetric Rota-Baxter ASI bisystems.

 \subsection{Admissible associative Yang-Baxter equations} In this subsection, we discuss the conditions under which $((A, \cdot, \D_r), R, S, Q, T)$ is a symmetric Rota-Baxter ASI bisystem, where $r=r^{1}\otimes r^{2}\in A \otimes A$ and $\Delta_{r} : A \rightarrow A \otimes A$ is given by
 \begin{eqnarray}
 \Delta_{r}(a) := r^{1}\otimes a r^{2}-r^{1} a\otimes  r^{2},\label{eq:db}
 \end{eqnarray} 
 for all $a \in A$. In this case, we call $((A, \cdot, \D_r), R, S, Q, T)$ {\bf coboundary}.

 Let us recall from \cite{Bai1} some basic facts about ASI bialgebra. Let $(A, \cdot)$ be an algebra. Then $(A, \cdot, \Delta_{r})$ is an ASI bialgebra if and only if $r$ satisfies
 \begin{eqnarray}
 (\mathcal{L}(a) \otimes \id - \id \otimes \mathcal{R}(a))(\id \otimes \mathcal{L}(b) - \mathcal{R}(b) \otimes \id)(r + \sigma(r)) = 0,\label{eq:db1}\\
 (\id \otimes \id \otimes \mathcal{L}(a) - \mathcal{R}(a) \otimes \id \otimes \id)(r_{12}r_{13} + r_{13}r_{23} - r_{23}r_{12}) = 0. \label{eq:db2}
 \end{eqnarray}
 for all $a \in A$.
 $r \in A \otimes A$ is a solution of the following \textbf{associative Yang-Baxter equation} (aYBe for short):
 \begin{eqnarray}
 r_{12}r_{13} + r_{13}r_{23} - r_{23}r_{12} = 0, \label{eq:db4}
 \end{eqnarray} where $\bar{r} = r$ and
 \begin{eqnarray*}
 r_{12}r_{13} = r^{1}\bar{r}^{1}\otimes r^{2}\otimes\bar{r}^{2},~~r_{13}r_{23} = r^{1}\otimes\bar{r}^{1} \otimes r^{2}\bar{r}^{2},~~r_{23}r_{12} = r^{1}\otimes\bar{r}^{1}r^{2} \otimes\bar{r}^{2}.
 \end{eqnarray*}

 If $r$ is an antisymmetric solution of aYBe, then $(A, \cdot, \Delta_{r})$ is an ASI bialgebra, where the linear map $\Delta_{r}$ is defined by Eq.\eqref{eq:db}. In this case, we call this bialgebra \textbf{triangular}. 

 \begin{thm}\label{thm:de} Let $(A, R, S)$ be a $(Q, T)$-adjoint admissible symmetric Rota-Baxter system. Define a linear map $\Delta : A \rightarrow A \otimes A$ by Eq.\eqref{eq:db} such that $(A^{*}, \D^{*})$ is an algebra. Then we have
 \begin{enumerate}[(1)]
 \item \label{it:thm:de1} Eq.\eqref{eq:cu} holds if and only if, for all $a \in A$,
 \begin{eqnarray}
 &&(\id\otimes \mathcal{L}(Q(a))-\id\otimes Q\circ \mathcal{L}(a))(Q\otimes \id -\id\otimes R)r+(Q\circ \mathcal{R}(a)\otimes \id)(\id\otimes Q-R\otimes \id)r\nonumber\\
 &&\hspace{70mm}-(\mathcal{R}(Q(a))\otimes \id)(\id\otimes T-S\otimes \id )r=0. \label{eq:de}\\
 &&(\id\otimes \mathcal{L}(Q(a)))(T\otimes \id-\id\otimes S )r-(\id\otimes Q\circ \mathcal{L}(a))(Q\otimes \id -\id\otimes R)r\nonumber\\
 &&\hspace{43mm}+(Q\circ \mathcal{R}(a)\otimes \id-\mathcal{R}(Q(a))\otimes \id)(\id\otimes Q-R\otimes \id)r=0.\label{eq:dede}
 \end{eqnarray}
 \item \label{it:thm:de2} Eq.\eqref{eq:cu1} holds if and only if, for all $a \in A$,
 \begin{eqnarray}
 &&(\id\otimes \mathcal{L}(T(a)))(Q\otimes \id -\id\otimes R)r+(T\circ \mathcal{R}(a)\otimes \id-(\mathcal{R}(T(a))\otimes \id))(\id\otimes T-S\otimes \id)r \nonumber\\
 &&\hspace{70mm}-(\id\otimes T\circ \mathcal{L}(a))(T\otimes \id-\id\otimes S)r=0. \label{eq:de1}\\
 &&(\id\otimes \mathcal{L}(T(a))-\id\otimes T\circ \mathcal{L}(a))(T\otimes \id -\id\otimes S)r+(T\circ \mathcal{R}(a)\otimes \id)(\id\otimes T-S\otimes \id)r\nonumber\\
 &&\hspace{70mm}-(\mathcal{R}(T(a))\otimes \id)(\id\otimes Q-R\otimes \id)r=0.\label{eq:de1de1}
 \end{eqnarray}
 \item \label{it:thm:de3} Eq.\eqref{eq:ck5} holds if and only if, for all $a \in A$,
 \begin{eqnarray}
 &&(\id\otimes \mathcal{L}(R(a))-\mathcal{R}(R(a))\otimes \id-T\circ \mathcal{R}(a)\otimes \id)(Q\otimes \id -\id\otimes R)r\nonumber\\
 &&\hspace{65mm}-(\id\otimes R\circ \mathcal{L}(a))(T\otimes \id-\id\otimes S)r=0.\label{eq:de2}\\
 &&(\id\otimes \mathcal{L}(R(a))-\mathcal{R}(S(a))\otimes \id-Q\circ \mathcal{R}(a)\otimes \id-\id\otimes R\circ \mathcal{L}(a))\nonumber\\
 &&\hspace{93mm}(Q\otimes \id -\id\otimes R)r=0.\label{eq:de2de2}
 \end{eqnarray}
 \item \label{it:thm:de4} Eq.\eqref{eq:ck6} holds if and only if, for all $a \in A$,
 \begin{eqnarray}
 &&(\id\otimes Q\circ \mathcal{L}(a)+\id\otimes \mathcal{L}(S(a))+R\circ \mathcal{R}(a)\otimes \id-\mathcal{R}(R(a))\otimes \id)\nonumber\\
 &&\hspace{93mm}(\id\otimes Q -R\otimes \id)r=0.\label{eq:de3}\\
 &&(\id\otimes T\circ \mathcal{L}(a)+\id\otimes \mathcal{L}(R(a))-\mathcal{R}(R(a))\otimes \id)(\id\otimes Q -R\otimes \id)r\nonumber\\
 &&\hspace{65mm}+(R\circ \mathcal{R}(a)\otimes \id)(\id\otimes T -S\otimes \id)r=0.\label{eq:de3de3}
 \end{eqnarray}
 \item \label{it:thm:de5} Eq.\eqref{eq:ck7} holds if and only if, for all $a \in A$,
 \begin{eqnarray}
 &&(\id\otimes \mathcal{L}(S(a))-\mathcal{R}(R(a))\otimes \id-T\circ \mathcal{R}(a)\otimes \id-\id\otimes S\circ \mathcal{L}(a))\nonumber\\
 &&\hspace{92mm}(T\otimes \id -\id\otimes S)r=0.\label{eq:de4}\\
 &&(\id\otimes \mathcal{L}(S(a))-\mathcal{R}(S(a))\otimes \id-Q\circ \mathcal{R}(a)\otimes \id)(T\otimes \id -\id\otimes S)r\nonumber\\
 &&\hspace{65mm}-(\id\otimes S\circ \mathcal{L}(a))(Q\otimes \id -\id\otimes R)r=0.\label{eq:de4de4}
 \end{eqnarray}
 \item \label{it:thm:de6} Eq.\eqref{eq:ck8} holds if and only if, for all $a \in A$,
 \begin{eqnarray}
 &&(\id\otimes Q\circ \mathcal{L}(a)+\id\otimes \mathcal{L}(S(a))-\mathcal{R}(S(a))\otimes \id)(\id\otimes T -S\otimes \id)r\nonumber\\
 &&\hspace{65mm}+(S\circ \mathcal{R}(a)\otimes \id)(\id\otimes Q -R\otimes \id)r=0.\label{eq:de5}\\
 &&(\id\otimes T\circ \mathcal{L}(a)+\id\otimes \mathcal{L}(R(a))-\mathcal{R}(S(a))\otimes \id+S\circ \mathcal{R}(a)\otimes \id)\nonumber\\
 &&\hspace{93mm}(\id\otimes T -S\otimes \id)r=0.\label{eq:de5de5}
 \end{eqnarray}
 \end{enumerate}
 \end{thm}

 \begin{proof} For all $a\in A$, the results can be proved by the following computation. 
  {\small\begin{eqnarray*}
 \ref{it:thm:de1}~~~0\hspace{-6mm}&=&\hspace{-6mm}-Q(a_{(1)}) \otimes Q(a_{(2)})+Q(Q(a)_{(1)}) \otimes Q(a)_{(2)} + Q(a)_{(1)} \otimes T(Q(a)_{(2)}) \\
 \hspace{-6mm}&=&\hspace{-6mm} Q(r^{1})\otimes Q(a)r^{2}-Q(r^{1}Q(a))\otimes r^{2}+r^{1}\otimes T(Q(a)r^{2})-r^{1}Q(a)\otimes T(r^{2})-Q(r^{1})\otimes Q(a r^{2})\\
 &&+Q(r^{1} a) \otimes Q(r^{2})\\
 \hspace{-3mm}&\stackrel{\eqref{eq:ck}\eqref{eq:ck1}}{=}&\hspace{-3mm} Q(r^{1})\otimes Q(a)r^{2}-Q(R(r^{1})a)\otimes r^{2}+S(r^{1})Q(a)\otimes r^{2}+r^{1}\otimes Q(a R(r^{2}))-r^{1}\otimes Q(a) R(r^{2})\\
 &&-r^{1}Q(a) \otimes T(r^{2})-Q(r^{1})\otimes Q(a r^{2})+Q(r^{1} a) \otimes Q(r^{2})\\
 \hspace{-6mm}&=&\hspace{-6mm}(\id\otimes \mathcal{L}(Q(a))-\id\otimes Q\circ \mathcal{L}(a))(Q\otimes \id -\id\otimes R)r+(Q\circ \mathcal{R}(a)\otimes \id)(\id\otimes Q-R\otimes \id)r\\
 &&-(\mathcal{R}(Q(a))\otimes \id)(\id\otimes T-S\otimes \id )r,\\
 0\hspace{-6mm}&=&\hspace{-6mm}-Q(a_{(1)}) \otimes Q(a_{(2)})+T(Q(a)_{(1)}) \otimes Q(a)_{(2)} + Q(a)_{(1)} \otimes Q(Q(a)_{(2)}) \\
 \hspace{-6mm}&=&\hspace{-6mm} T(r^{1})\otimes Q(a)r^{2}-T(r^{1}Q(a))\otimes r^{2}+r^{1}\otimes Q(Q(a)r^{2})-r^{1}Q(a)\otimes Q(r^{2})-Q(r^{1})\otimes Q(a r^{2})\\
 &&+Q(r^{1} a) \otimes Q(r^{2})\\
 \hspace{-6mm}&=&\hspace{-6mm} T(r^{1})\otimes Q(a)r^{2}-Q(R(r^{1})a)\otimes r^{2}+R(r^{1})Q(a)\otimes r^{2}+r^{1}\otimes Q(a R(r^{2}))-r^{1}\otimes Q(a) S(r^{2})\\
 &&-r^{1}Q(a) \otimes Q(r^{2})-Q(r^{1})\otimes Q(a r^{2})+Q(r^{1} a) \otimes Q(r^{2})\\
 \hspace{-6mm}&=&\hspace{-6mm}(\id\otimes \mathcal{L}(Q(a)))(T\otimes \id-\id\otimes S )r-(\id\otimes Q\circ \mathcal{L}(a))(Q\otimes \id -\id\otimes R)r+(Q\circ \mathcal{R}(a)\otimes \id\\
 &&-\mathcal{R}(Q(a))\otimes \id)(\id\otimes Q-R\otimes \id)r,\\
 \ref{it:thm:de2}~~~0\hspace{-6mm}&=&\hspace{-6mm}-T(a_{(1)}) \otimes T(a_{(2)})+Q(T(a)_{(1)}) \otimes T(a)_{(2)} + T(a)_{(1)} \otimes T(T(a)_{(2)}) \\
 \hspace{-6mm}&=&\hspace{-6mm}Q(r^{1})\otimes T(a)r^{2}-Q(r^{1}T(a))\otimes r^{2}+r^{1}\otimes T(T(a)r^{2})-r^{1}T(a)\otimes T(r^{2})-T(r^{1})\otimes T(a r^{2})\\
 &&+T(r^{1} a) \otimes T(r^{2})\\
 \hspace{-6mm}&=&\hspace{-6mm} Q(r^{1})\otimes T(a)r^{2}-T(S(r^{1})a)\otimes r^{2}+S(r^{1})T(a)\otimes r^{2}+r^{1}\otimes T(a S(r^{2}))-r^{1}\otimes T(a) R(r^{2})\\
 &&-r^{1}T(a) \otimes T(r^{2})-T(r^{1})\otimes T(a r^{2})+T(r^{1} a) \otimes T(r^{2})\\
 \hspace{-6mm}&=&\hspace{-6mm} (\id\otimes \mathcal{L}(T(a)))(Q\otimes \id -\id\otimes R)r+(T\circ \mathcal{R}(a)\otimes \id-(\mathcal{R}(T(a))\otimes \id))(\id\otimes T-S\otimes \id)r\\
 &&-(\id\otimes T\circ \mathcal{L}(a))(T\otimes \id-\id\otimes S)r,\\
 0\hspace{-6mm}&=&\hspace{-6mm}-T(a_{(1)}) \otimes T(a_{(2)})+T(T(a)_{(1)}) \otimes T(a)_{(2)} + T(a)_{(1)} \otimes Q(T(a)_{(2)}) \\
 \hspace{-6mm}&=&\hspace{-6mm}T(r^{1})\otimes T(a)r^{2}-T(r^{1}T(a))\otimes r^{2}+r^{1}\otimes Q(T(a)r^{2})-r^{1}T(a)\otimes Q(r^{2})-T(r^{1})\otimes T(a r^{2})\\
 &&+T(r^{1} a) \otimes T(r^{2})\\
 \hspace{-6mm}&=&\hspace{-6mm} T(r^{1})\otimes T(a)r^{2}-T(S(r^{1})a)\otimes r^{2}+R(r^{1})T(a)\otimes r^{2}+r^{1}\otimes T(a S(r^{2}))-r^{1}\otimes T(a) S(r^{2})\\
 &&-r^{1}T(a) \otimes Q(r^{2})-T(r^{1})\otimes T(a r^{2})+T(r^{1} a) \otimes T(r^{2})\\
 \hspace{-6mm}&=&\hspace{-6mm} (\id\otimes \mathcal{L}(T(a))-\id\otimes T\circ \mathcal{L}(a))(T\otimes \id -\id\otimes S)r+(T\circ \mathcal{R}(a)\otimes \id)(\id\otimes T-S\otimes \id)r\\
 &&-(\mathcal{R}(T(a))\otimes \id)(\id\otimes Q-R\otimes \id)r,\\
 \ref{it:thm:de3}~~~0\hspace{-6mm}&=&\hspace{-6mm}Q(R(a)_{(1)})\otimes R(a)_{(2)}-R(a)_{(1)}\otimes R(R(a)_{(2)})-T(a_{(1)})\otimes R(a_{(2)})\\
 \hspace{-6mm}&=&\hspace{-6mm}Q(r^{1})\otimes R(a) r^{2}-Q(r^{1}R(a))\otimes r^{2}-r^{1}\otimes R(R(a)r^{2})+r^{1}R(a)\otimes R(r^{2})-T(r^{1})\otimes R(a r^{2})\\
 &&+T(r^{1}a)\otimes R(r^{2})\\
 \hspace{-6mm}&=&\hspace{-6mm}Q(r^{1})\otimes R(a) r^{2}-Q(r^{1})R(a)\otimes r^{2}-T(Q(r^{1})a)\otimes r^{2}-r^{1}\otimes R(a)R(r^{2})+r^{1}\otimes R(a S(r^{2}))\\
 &&+r^{1}R(a)\otimes R(r^{2})-T(r^{1})\otimes R(a r^{2})+T(r^{1}a)\otimes R(r^{2})\\
 \hspace{-6mm}&=&\hspace{-6mm} (\id\otimes \mathcal{L}(R(a))-\mathcal{R}(R(a))\otimes \id-T\circ \mathcal{R}(a)\otimes \id)(Q\otimes \id -\id\otimes R)r-(\id\otimes R\circ \mathcal{L}(a))\\
 &&(T\otimes \id-\id\otimes S)r,\\
 0\hspace{-6mm}&=&\hspace{-6mm}Q(R(a)_{(1)})\otimes R(a)_{(2)}-S(a)_{(1)}\otimes R(S(a)_{(2)})-Q(a_{(1)})\otimes R(a_{(2)})\\
 \hspace{-6mm}&=&\hspace{-6mm}Q(r^{1})\otimes R(a) r^{2}-Q(r^{1}R(a))\otimes r^{2}-r^{1}\otimes R(R(a)r^{2})+r^{1}R(a)\otimes R(r^{2})-T(r^{1})\otimes R(a r^{2})\\
 &&+T(r^{1}a)\otimes R(r^{2})\\
 \hspace{-6mm}&=&\hspace{-6mm}Q(r^{1})\otimes R(a) r^{2}-Q(r^{1})S(a)\otimes r^{2}-Q(Q(r^{1})a)\otimes r^{2}-r^{1}\otimes R(a)R(r^{2})+r^{1}\otimes R(a R(r^{2}))\\
 &&+r^{1}S(a)\otimes R(r^{2})-Q(r^{1})\otimes R(a r^{2})+Q(r^{1}a)\otimes R(r^{2})\\
 \hspace{-6mm}&=&\hspace{-6mm} (\id\otimes \mathcal{L}(R(a))-\mathcal{R}(S(a))\otimes \id-Q\circ \mathcal{R}(a)\otimes \id-\id\otimes R\circ \mathcal{L}(a))(Q\otimes \id -\id\otimes R)r,\\
 \ref{it:thm:de4}~~~0\hspace{-6mm}&=&\hspace{-6mm}R(a)_{(1)}\otimes Q(R(a)_{(2)})-R(S(a)_{(1)})\otimes S(a)_{(2)}-R(a_{(1)})\otimes Q(a_{(2)}) \\
 \hspace{-6mm}&=&\hspace{-6mm}r^{1}\otimes Q(R(a) r^{2})-r^{1}R(a)\otimes Q(r^{2})-R(r^{1})\otimes Q(a r^{2})+R(r^{1}a)\otimes Q(r^{2})-R(r^{1})\otimes S(a) r^{2}\\
 &&+R(r^{1}S(a))\otimes r^{2}\\
 \hspace{-6mm}&=&\hspace{-6mm}r^{1}\otimes Q(a Q(r^{2}))+r^{1}\otimes S(a)Q(r^{2})-r^{1}R(a)\otimes Q(r^{2})-R(r^{1})\otimes Q(a r^{2})+R(r^{1}a)\otimes Q(r^{2})\\
 &&+R(r^{1})R(a)\otimes r^{2}-R(R(r^{1})a)\otimes r^{2}-R(r^{1})\otimes S(a)r^{2}\\
 \hspace{-6mm}&=&\hspace{-6mm} (\id\otimes Q\circ \mathcal{L}(a)+\id\otimes \mathcal{L}(S(a))+R\circ \mathcal{R}(a)\otimes \id-\mathcal{R}(R(a))\otimes \id)(\id\otimes Q -R\otimes \id)r,\\
 0\hspace{-6mm}&=&\hspace{-6mm}R(a)_{(1)}\otimes Q(R(a)_{(2)})-R(R(a)_{(1)})\otimes R(a)_{(2)}-R(a_{(1)})\otimes T(a_{(2)}) \\
 \hspace{-6mm}&=&\hspace{-6mm}r^{1}\otimes Q(R(a) r^{2})-r^{1}R(a)\otimes Q(r^{2})-R(r^{1})\otimes T(ar^{2})+R(r^{1}a)\otimes T(r^{2})-R(r^{1})\otimes R(a) r^{2}\\
 &&+R(r^{1}R(a))\otimes r^{2}\\
 \hspace{-6mm}&=&\hspace{-6mm}r^{1}\otimes T(a Q(r^{2}))+r^{1}\otimes R(a)Q(r^{2})-r^{1}R(a)\otimes Q(r^{2})-R(r^{1})\otimes T(a r^{2})+R(r^{1}a)\otimes T(r^{2})\\
 &&+R(r^{1})R(a)\otimes r^{2}-R(S(r^{1})a)\otimes r^{2}-R(r^{1})\otimes R(a)r^{2}\\
 \hspace{-6mm}&=&\hspace{-6mm} (\id\otimes T\circ \mathcal{L}(a)+\id\otimes \mathcal{L}(R(a))-\mathcal{R}(R(a))\otimes \id)(\id\otimes Q -R\otimes \id)r+(R\circ \mathcal{R}(a)\otimes \id)\\
 &&(\id\otimes T -S\otimes \id)r,\\
 \ref{it:thm:de5}~~~0\hspace{-6mm}&=&\hspace{-6mm}T(S(a)_{(1)})\otimes S(a)_{(2)}-R(a)_{(1)}\otimes S(R(a)_{(2)})-T(a_{(1)})\otimes S(a_{(2)}) \\
 \hspace{-6mm}&=&\hspace{-6mm}T(r^{1})\otimes S(a) r^{2}-T(r^{1}S(a))\otimes r^{2}-r^{1}\otimes S(R(a)r^{2})+r^{1}R(a)\otimes S(r^{2})-T(r^{1})\otimes S(a r^{2})\\
 &&+T(r^{1}a)\otimes S(r^{2})\\
 \hspace{-6mm}&=&\hspace{-6mm}T(r^{1})\otimes S(a) r^{2}-T(r^{1})R(a)\otimes r^{2}-T(T(r^{1})a)\otimes r^{2}-r^{1}\otimes S(a)S(r^{2})+r^{1}\otimes S(a S(r^{2}))\\
 &&+r^{1}R(a)\otimes S(r^{2})-T(r^{1})\otimes S(a r^{2})+T(r^{1}a)\otimes S(r^{2})\\
 \hspace{-6mm}&=&\hspace{-6mm} (\id\otimes \mathcal{L}(S(a))-\mathcal{R}(R(a))\otimes \id-T\circ \mathcal{R}(a)\otimes \id-\id\otimes S\circ \mathcal{L}(a))(T\otimes \id -\id\otimes S)r,\\
 0\hspace{-6mm}&=&\hspace{-6mm}T(S(a)_{(1)})\otimes S(a)_{(2)}-S(a)_{(1)}\otimes S(S(a)_{(2)})-Q(a_{(1)})\otimes S(a_{(2)}) \\
 \hspace{-6mm}&=&\hspace{-6mm}T(r^{1})\otimes S(a) r^{2}-T(r^{1}S(a))\otimes r^{2}-r^{1}\otimes S(S(a)r^{2})+r^{1}S(a)\otimes S(r^{2})-Q(r^{1})\otimes S(a r^{2})\\
 &&+Q(r^{1}a)\otimes S(r^{2})\\
 \hspace{-6mm}&=&\hspace{-6mm}T(r^{1})\otimes S(a) r^{2}-T(r^{1})S(a)\otimes r^{2}-Q(T(r^{1})a)\otimes r^{2}-r^{1}\otimes S(a)S(r^{2})+r^{1}\otimes S(a R(r^{2}))\\
 &&+r^{1}S(a)\otimes S(r^{2})-Q(r^{1})\otimes S(a r^{2})+Q(r^{1}a)\otimes S(r^{2})\\
 \hspace{-6mm}&=&\hspace{-6mm} (\id\otimes \mathcal{L}(S(a))-\mathcal{R}(S(a))\otimes \id-Q\circ \mathcal{R}(a)\otimes \id)(T\otimes \id -\id\otimes S)r-(\id\otimes S\circ \mathcal{L}(a))\\
 &&(Q\otimes \id -\id\otimes R)r,\\
 \ref{it:thm:de6}~~~0\hspace{-6mm}&=&\hspace{-6mm}S(a)_{(1)}\otimes T(S(a)_{(2)})-S(S(a)_{(1)})\otimes S(a)_{(2)}-S(a_{(1)})\otimes Q(a_{(2)})\\
 \hspace{-6mm}&=&\hspace{-6mm}r^{1}\otimes T(S(a) r^{2})-r^{1}S(a)\otimes T(r^{2})-S(r^{1})\otimes Q(a r^{2})+S(r^{1}a)\otimes Q(r^{2})-S(r^{1})\otimes S(a) r^{2}\\
 &&+S(r^{1}S(a))\otimes r^{2}\\
 \hspace{-6mm}&=&\hspace{-6mm}r^{1}\otimes Q(a T(r^{2}))+r^{1}\otimes S(a)T(r^{2})-r^{1}S(a)\otimes T(r^{2})-S(r^{1})\otimes Q(a r^{2})+S(r^{1}a)\otimes Q(r^{2})\\
 &&-S(r^{1})\otimes S(a)r^{2}+S(r^{1})S(a)\otimes r^{2}-S(R(r^{1})a)\otimes r^{2}\\
 \hspace{-6mm}&=&\hspace{-6mm} (\id\otimes Q\circ \mathcal{L}(a)+\id\otimes \mathcal{L}(S(a))-\mathcal{R}(S(a))\otimes \id)(\id\otimes T -S\otimes \id)r+(S\circ \mathcal{R}(a)\otimes \id)\\
 &&(\id\otimes Q -R\otimes \id)r,\\
 0\hspace{-6mm}&=&\hspace{-6mm}S(a)_{(1)}\otimes T(S(a)_{(2)})-S(R(a)_{(1)})\otimes R(a)_{(2)}-S(a_{(1)})\otimes T(a_{(2)})\\
 \hspace{-6mm}&=&\hspace{-6mm}r^{1}\otimes T(S(a) r^{2})-r^{1}S(a)\otimes T(r^{2})-S(r^{1})\otimes T(a r^{2})+S(r^{1}a)\otimes T(r^{2})-S(r^{1})\otimes R(a) r^{2}\\
 &&+S(r^{1}R(a))\otimes r^{2}\\
 \hspace{-6mm}&=&\hspace{-6mm}r^{1}\otimes T(a T(r^{2}))+r^{1}\otimes R(a)T(r^{2})-r^{1}S(a)\otimes T(r^{2})-S(r^{1})\otimes T(a r^{2})+S(r^{1}a)\otimes T(r^{2})\\
 &&-S(r^{1})\otimes R(a)r^{2}+S(r^{1})S(a)\otimes r^{2}-S(S(r^{1})a)\otimes r^{2}\\
 \hspace{-6mm}&=&\hspace{-6mm} (\id\otimes T\circ \mathcal{L}(a)+\id\otimes \mathcal{L}(R(a))-\mathcal{R}(S(a))\otimes \id+S\circ \mathcal{R}(a)\otimes \id)(\id\otimes T -S\otimes \id)r,
 \end{eqnarray*}}
 finishing the proof.
 \end{proof}

 Based on the discussion above, we can obtain 
 \begin{thm} \label{thm:df} Let $(A, R, S)$ be a $(Q, T)$-adjoint admissible symmetric Rota-Baxter system. Then $((A, \cdot, \D_r), R, S, Q, T)$, where $\Delta_r$ is defined by Eq.\eqref{eq:db}, is a symmetric Rota-Baxter ASI bisystem if and only if Eqs.\eqref{eq:db1}-\eqref{eq:db2} and \eqref{eq:de}-\eqref{eq:de5de5} are satisfied.
 \end{thm}

 \begin{rmk}\label{rmk:thm:df} If $r$ is antisymmetric, then $(R\otimes \id -\id \otimes Q)r=0$ if and only if $(Q\otimes \id-\id\otimes R)r=0$. And $(S\otimes \id -\id \otimes T)r=0$ if and only if $(T\otimes \id-\id\otimes S)r=0$.
 \end{rmk}

 By Theorem  \ref{thm:df} and Remark \ref{rmk:thm:df}, we have

 \begin{cor}\label{cor:thm:df} If $r$ is antisymmetric, and
 \begin{eqnarray}
 &(Q\otimes \id - \id\otimes R)r=0,&\label{eq:dh}\\
 &(T\otimes \id - \id\otimes S)r=0,& \label{eq:dh1}
 \end{eqnarray}
 or
 \begin{eqnarray*}
 &(S\otimes \id - \id\otimes T)r=0,&\\
 &(R\otimes \id - \id\otimes Q)r=0,&
 \end{eqnarray*}
 hold, then we have Eqs.\eqref{eq:db1} and \eqref{eq:de}-\eqref{eq:de5de5}. Moreover, if $r$ is a solution of the aYBe, then $(A, R, S, Q, T)$ is a symmetric Rota-Baxter ASI bisystem.
 \end{cor}

 \begin{defi}\label{de:dh} Let $(A, R, S)$ be a symmetric Rota-Baxter system, $r \in A \otimes A$ and $Q,T : A\rightarrow A$ linear maps. Then Eq.\eqref{eq:db4} together with Eqs.(\ref{eq:dh}) and (\ref{eq:dh1})
 is called a \textbf{$(Q, T)$-admissible aYBe in $(A, R, S)$}.
 \end{defi}

 By Corollary \ref{cor:thm:df} and Definition \ref{de:dh}, one has

 \begin{pro}\label{pro:cii} Let $(A, R, S)$ be a $(Q, T)$-adjoint admissible symmetric Rota-Baxter system and $r \in A \otimes A$. If $r$ is an antisymmetric solution of the $(Q, T)$-admissible aYBe, then $(A, R, S, Q, T)$ is a symmetric Rota-Baxter ASI bisystem. We call this bialgebra $(A, R, S, Q, T)$ \textbf{triangular}.
 \end{pro}

 \subsection{$\mathcal{O}$-operator}
 \begin{defi}\label{de:dk} Let $(A, R, S)$ be a symmetric Rota-Baxter system, $(M, \ell, r)$ be an $A$-bimodule and $\alpha, \beta: M\rightarrow M$ linear maps. A linear map $\mathfrak{T}:  M \rightarrow A$ is called a \textbf{weak $\mathcal{O}$-operator associated to $(M, \ell, r)$ and $\alpha, \beta$} if the equations below hold:
 \begin{eqnarray*}
 &\mathfrak{T}(m) \mathfrak{T}(n) = \mathfrak{T}(\ell(\mathfrak{T}(m)n)) + \mathfrak{T}(m r(\mathfrak{T}(n))),&  \label{eq:dk}\\
 &\mathfrak{T}\ci \alpha= R\ci \mathfrak{T},& \label{eq:dk1}\\
 &\mathfrak{T}\ci \beta= S\ci \mathfrak{T},&\label{eq:dk2}
 \end{eqnarray*}
 for all $m, n \in M$, if in addition, $(M, \ell, r, \alpha, \beta)$ is a representation of $(A, R, S)$, then $\mathfrak{T}$ is called an \textbf{$\mathcal{O}$-operator associated to $(M, \ell, r, \alpha, \beta)$.}
 \end{defi}

 \begin{lem} \label{lem:dl}{\em \cite{BGN}} Let $A$ be an algebra and $r \in A \otimes A$. Then $r$ is an antisymmetric solution of the aYBe in $A$ if and only if $r^{+}$ is an $\mathcal{O}$-operator, i.e., 
 \begin{eqnarray}
 &r^{+}(a^{*}) \cdot  r^{+}(b^{*})=r^{+}(\mathcal{R}^{*}(r^{+}(a^{*}))b^{*} + a^{*}\mathcal{L}^{*}(r^{+}(b^{*}))).&\label{eq:dl4}
 \end{eqnarray}
 where $r^{+} : A^{*} \rightarrow A$ by
 \begin{eqnarray*}
 \langle r^{+}(\xi), \eta\rangle = \langle r, \xi\otimes\eta\rangle, \forall \xi, \eta \in A^{*}.
 \end{eqnarray*}
 \end{lem}

 \begin{thm}\label{thm:dm} Let $(A, R, S)$ be a symmetric Rota-Baxter system, $Q, T: A \rightarrow A$ linear maps and $r \in A \otimes A$. If $r$ is antisymmetric, then $r$ is a solution of the $(Q, T)$-admissible aYBe in $(A, R, S)$ if and only if $r^{+}$ is a weak $\mathcal{O}$-operator and
 \begin{eqnarray}
 &R \circ r^{+} = r^{+} \circ Q^{*},~~~~S \circ r^{+} = r^{+} \circ T^{*},&\label{eq:dm1}
 \end{eqnarray} hold.
 \end{thm}

 \begin{proof} By Lemma \ref{lem:dl}, if $r$ is antisymmetric, then $r$ is a solution of aYBe in $(A, \cdot)$ if and only if Eq.\eqref{eq:dl4} holds. Moreover, for all $a^{*}, b^{*}\in A^{*}$, we have
 \begin{eqnarray*}
 \langle b^{*}, R(r^{+}(a^{*}))-r^{+}(Q^{*}(a^{*}) \rangle
 &=&\langle  R^{*}(b^{*}), r^{+}(a^{*}) \rangle -\langle Q^{*}(a^{*})\otimes b^{*}, r\rangle\\
 &=&\langle a^{*}\otimes R^{*}(b^{*}), r\rangle-\langle a^{*}\otimes b^{*}, (Q\otimes \id)(r) \rangle\\
 &=&\langle  a^{*}\otimes b^{*}, (\id\otimes R)r -(Q\otimes \id)(r) \rangle,\\
 \langle b^{*}, S(r^{+}(a^{*}))-r^{+}(T^{*}(a^{*}))\rangle
 &=&\langle  S^{*}(b^{*}), r^{+}(a^{*}) \rangle -\langle T^{*}(a^{*})\otimes b^{*}, r\rangle\\
 &=&\langle a^{*}\otimes S^{*}(b^{*}), r\rangle-\langle a^{*}\otimes b^{*}, (T\otimes \id)(r) \rangle\\
 &=&\langle  a^{*}\otimes b^{*}, (\id\otimes S)r-(T\otimes \id)(r) \rangle.
 \end{eqnarray*}
 Hence, Eqs.\eqref{eq:dh}-\eqref{eq:dh1} hold if and only if Eq.\eqref{eq:dm1} holds.
 \end{proof}

 Based on Theorem \ref{thm:dm} and Lemma \ref{lem:dl}, we have

 \begin{cor}\label{cor:dkk} Let $(A, R, S)$ be a symmetric Rota-Baxter system, $ Q, T: A \rightarrow A$  linear maps and $r \in A \otimes A$. If $r$ is antisymmetric, then $r$ is a solution of the $(Q, T)$-admissible aYBe if and only if $r^{+}$ is a weak $\mathcal{O}$-operator associated to $(A^{*}, \mathcal{R^{*}}, \mathcal{L^{*}})$ and $(Q^{*}, T^{*})$.
 \end{cor}

 \begin{thm}\label{thm:dn} Let $(A, R, S)$ be a symmetric Rota-Baxter system, $(M, \ell, r)$ an $A$-bimodule, $Q, T : A \rightarrow A$ and $\alpha, \beta, \xi, \zeta : M \rightarrow M$ linear maps. Then the following conditions are equivalent.
 \begin{enumerate}[(1)]
 \item \label{it:thm:dn1} There is a symmetric Rota-Baxter system $(A \ltimes_{\ell, r} M, R+\alpha, S+\beta)$ such that $(Q + \xi, T+\zeta)$ is adjoint admissible to  $(A \ltimes_{\ell, r} M, R+\alpha, S+\beta)$.

 \item \label{it:thm:dn2} There is a symmetric Rota-Baxter system $(A \ltimes_{r^{*}, \ell^{*}} M^{*}, R+\xi^{*}, S+\zeta^{*})$ such that
     $(Q + \alpha^{*}, T+\beta^{*})$ is adjoint admissible to  $(A \ltimes_{r^{*}, \ell^{*}} M^{*}, R+\xi^{*}, S+\zeta^{*})$.

 \item \label{it:thm:dn3} The following conditions are satisfied:
 \begin{enumerate}[(3a)]
 \item \label{it:thm:dn4} $(M, \ell, r, \alpha, \beta)$ is a representation of $(A, R, S)$, that is, Eqs.\eqref{eq:cf}-\eqref{eq:cf3} hold;

 \item \label{it:thm:dn5} $(Q, T)$ is adjoint admissible to $(A, R, S)$, that is, Eqs. \eqref{eq:ck}-\eqref{eq:ck3} hold;

 \item \label{it:thm:dn6} $(\xi, \zeta)$ is admissible to $(A, R, S)$ with respect to $(M, \ell, r)$, that is, Eqs.\eqref{eq:cj}-\eqref{eq:cj3} hold;

 \item \label{it:thm:dn7} For $a \in A, m\in V$,\vskip-6mm
 \begin{eqnarray}
 &\xi(\alpha(m)r(a))=\xi(m r(Q(a)))+\beta(m)r(Q(a))
 =\zeta(m r(Q(a)))+\alpha(m)r(Q(a)),&\label{eq:dn}\\
 &\xi(\ell(a)\alpha(m))=\zeta(\ell(Q(a))m)+\ell(Q(a))\alpha(m)
 =\xi(\ell(Q(a))m)+\ell(Q(a))\beta(m),&\label{eq:dn1}\\
 &\zeta(\beta(m)r(a))=\xi(m r(T(a)))+\beta(m)r(T(a))
 =\zeta(m r(T(a)))+\alpha(m)r(T(a)),&\label{eq:dn2}\\
 &\zeta(\ell(a)\beta(m))=\zeta(\ell(T(a))m)+\ell(T(a))\alpha(m)
 =\xi(\ell(T(a))m)+\ell(T(a))\beta(m).&\label{eq:dn3}
 \end{eqnarray}
 \end{enumerate}
 \end{enumerate}
 \end{thm}

 \begin{proof} (\ref{it:thm:dn1}$\Leftrightarrow$\ref{it:thm:dn3}) By Proposition \ref{pro:cff}, we have  $(A\ltimes_{\ell, r} M,  R+\alpha, S+\beta)$ is a symmetric Rota-Baxter system if and only if $(M, \ell, r, \alpha, \beta)$ a representation of $(A, R, S)$. Let $x, y \in A$ and $m, n\in M$, we have
 {\small\begin{eqnarray*}
 &&\hspace{-16mm}(Q+\xi)((R+\alpha)(a+m)\cdot_{A\ltimes_{\ell, r} M}(b+n))-(Q+\xi)((a+m) \cdot_{A\ltimes_{\ell, r} M} (Q+\xi)(b+n))\\
 &&-(S+\beta)(a+m) \cdot_{A\ltimes_{\ell, r} M} (Q+\xi)(b+n)\\
 &=&Q(R(a)b)-Q(aQ(b))-S(a)Q(b)+\xi(\ell(R(a))n)-\xi(\ell(a)\xi(n))-\ell(S(a))\xi(n)\\
 &&+\xi(\alpha(m)r(b))-\xi(m r(Q(b)))-\beta(m)r(Q(b)),\\
 &&\hspace{-16mm}(Q+\xi)((R+\alpha)(a+m)\cdot_{A\ltimes_{\ell, r} M}(b +n))-(T+\zeta)((a+m) \cdot_{A\ltimes_{\ell, r} M} (Q+\xi)(b+n))\\
 &&-(R+\alpha)(a+m) \cdot_{A\ltimes_{\ell, r} M} (Q+\xi)(b+n)\\
 &=&Q(R(a)b)-T(aQ(b))-R(a)Q(b)+\xi(\ell(R(a))n)-\zeta(\ell(a)\xi(n))-\ell(R(a))\xi(n)\\
 &&+\xi(\alpha(m)r(b))-\zeta(m r(Q(b)))-\alpha(m)r(Q(b)),\\
 &&\hspace{-16mm}(Q+\xi)((a+m) \cdot_{A\ltimes_{\ell, r} M} (R+\alpha)(b+n)) -(T+\zeta)((Q+\xi)(a+m) \cdot_{A\ltimes_{\ell, r} M} (b+n))\\
 &&-(Q+\xi)(a+m) \cdot_{A\ltimes_{\ell, r} M} (R+\alpha)(b+n)\\
 &=&Q(aR(b))-T(Q(a)b)-Q(a)R(b)+\xi(\ell(a)\alpha(n))-\zeta(\ell(Q(a))n)-\ell(Q(a))\alpha(n)\\
 &&+\xi(m r(R(b)))-\zeta(\xi(m)r(b))-\xi(m)r(R(b)),\\
 &&\hspace{-16mm}(Q+\xi)((a+m) \cdot_{A\ltimes_{\ell, r} M} (R+\alpha)(b+n)) -(Q+\xi)((Q+\xi)(a+m) \cdot_{A\ltimes_{\ell, r} M} (b+n))\\
 &&-(Q+\xi)(a+m) \cdot_{A\ltimes_{\ell, r} M} (S+\beta)(b+n)\\
 &=&Q(aR(b))-Q(Q(a)b)-Q(a)S(b)+\xi(\ell(a)\alpha(n))-\xi(\ell(Q(a))n)-\ell(Q(a))\beta(n)\\
 &&+\xi(m r(R(b)))-\xi(\xi(m)r(b))-\xi(m)r(S(b)),\\
 &&\hspace{-16mm}(T+\zeta)((S+\beta)(a+m) \cdot_{A\ltimes_{\ell, r} M} (b+n))-(Q+\xi)((a+m) \cdot_{A\ltimes_{\ell, r} M} (T+\zeta)(b+n))\\
 &&-(S+\beta)(a+m) \cdot_{A\ltimes_{\ell, r} M} (T+\zeta)(b+n)\\
 &=&T(S(a)b)-Q(aT(b))-S(a)T(b)+\zeta(\ell(S(a))n)-\xi(\ell(a)\zeta(n))-\ell(S(a))\zeta(n)\\
 &&+\zeta(\beta(m)r(b))-\xi(m r(T(b)))-\beta(m)r(T(b)),\\
 &&\hspace{-16mm}(T+\zeta)((S+\beta)(a+m) \cdot_{A\ltimes_{\ell, r} M} (b+n))-(T+\zeta)((a+m) \cdot_{A\ltimes_{\ell, r} M} (T+\zeta)(b+n))\\
 &&-(R+\alpha)(a+m) \cdot_{A\ltimes_{\ell, r} M} (T+\zeta)(b+n)\\
 &=&T(S(a)b)-T(aT(b))-R(a)T(b)+\zeta(\ell(S(a))n)-\zeta(\ell(a)\zeta(n))-\ell(R(a))\zeta(n)\\
 &&+\zeta(\beta(m)r(b))-\zeta(m r(T(b)))-\alpha(m)r(T(b)),\\
 &&\hspace{-16mm}(T+\zeta)((a+m) \cdot_{A\ltimes_{\ell, r} M} (S+\beta)(b+n))-(T+\zeta)((T+\zeta)(a+m) \cdot_{A\ltimes_{\ell, r} M} (b+n))\\
 &&-(T+\zeta)(a+m) \cdot_{A\ltimes_{\ell, r} M} (R+\alpha)(b+n)\\
 &=&T(a S(b))-T(T(a)b)-T(a)R(b)+\zeta(\ell(a)\beta(n))-\zeta(\ell(T(a))n)-\ell(T(a))\alpha(n)\\
 &&+\zeta(m r(S(b)))-\zeta(\zeta(m)r(b))-\zeta(m)r(R(b)),\\
 &&\hspace{-16mm}(T+\zeta)((a+m) \cdot_{A\ltimes_{\ell, r} M} (S+\beta)(b+n))-(Q+\xi)((T+\zeta)(a+m) \cdot_{A\ltimes_{\ell, r} M} (b+n))\\
 &&-(T+\zeta)(a+m) \cdot_{A\ltimes_{\ell, r} M} (S+\beta)(b+n)\\
 &=&T(a S(b))-Q(T(a)b)-T(a)S(b)+\zeta(\ell(a)\beta(n))-\xi(\ell(T(a))n)-\ell(T(a))\beta(n)\\
 &&+\zeta(m r(S(b)))-\xi(\zeta(m)r(b))-\zeta(m)r(S(b)).
 \end{eqnarray*}}
 \noindent Therefore Eq.\eqref{eq:ck} holds (where $Q$ is replaced by $Q + \xi$, $T$ by $T + \zeta$, $R$ by $P + \alpha$, $S$ by $S + \beta$, $a$ by $a + m$, and $b$ by $b + m$) if and only if Eq.\eqref{eq:ck} (corresponding to $m = n = 0$), Eq.\eqref{eq:cj} (corresponding to $b = m = 0$) and Eq.\eqref{eq:dn}, where $a$ is replaced by $b$, (corresponding to $a = n = 0$) hold. Likewise, other cases hold. Hence Item \ref{it:thm:dn1} holds if and only if Item \ref{it:thm:dn3} holds.
 	
 (\ref{it:thm:dn2}$\Leftrightarrow$\ref{it:thm:dn3}) In Item \ref{it:thm:dn1}, take
 $ M=M^{*}, \ell=r^{*}, r=\ell^{*}, \alpha=\xi^{*}, \xi=\alpha^{*}, \beta=\zeta^{*}, \zeta=\beta^{*}$. Then from the above equivalence between Item \ref{it:thm:dn1} and Item \ref{it:thm:dn3}, we have Item \ref{it:thm:dn2} holds if and only if Items \ref{it:thm:dn4}-\ref{it:thm:dn6} hold and the following equations hold:
 {\small\begin{eqnarray}
 \alpha^{*}(\xi^{*}(m^{*})\ell^{*}(a))=\alpha^{*}(m^{*}\ell^{*}(Q(a)))+\zeta^{*}(m^{*})\ell^{*}(Q(a))=\beta^{*}(m^{*}\ell^{*}(Q(a)))+\xi^{*}(m^{*})\ell^{*}(Q(a)),\label{eq:dn4}\\
 \alpha^{*}(r^{*}(a)\xi^{*}(m^{*}))=\beta^{*}(r^{*}(Q(a))m^{*})+r^{*}(Q(a))\xi^{*}(m^{*})=\alpha^{*}(r^{*}(Q(a))m^{*})+r^{*}(Q(a))\zeta^{*}(m^{*}),\label{eq:dn5}\\
 \beta^{*}(\zeta^{*}(m^{*})\ell^{*}(a))=\alpha^{*}(m^{*}\ell^{*}(T(a)))+\zeta^{*}(m^{*})\ell^{*}(T(a))=\beta^{*}(m^{*}\ell^{*}(T(a)))+\xi^{*}(m^{*})\ell^{*}(T(a)),\label{eq:dn6}\\
 \beta^{*}(r^{*}(a)\zeta^{*}(m^{*}))=\beta^{*}(r^{*}(T(a))m^{*})+r^{*}(T(a))\xi^{*}(m^{*})=\alpha^{*}(r^{*}(T(a))m^{*})+r^{*}(T(a))\zeta^{*}(m^{*}).\label{eq:dn7}
 \end{eqnarray}}
 For all $a\in A$, $n\in M$ and $m^{*}\in M^{*}$, we have
 {\small\begin{eqnarray*}
 &&\hspace{-30mm}\big\langle \alpha^{*}(\xi^{*}(m^{*})\ell^{*}(a))-\alpha^{*}(m^{*}\ell^{*}(Q(a)))-\zeta^{*}(m^{*})\ell^{*}(Q(a)), n\big\rangle \\
 &=&\big\langle m^{*}, \xi (\ell(a)\alpha(n))-\ell(Q(a))\alpha(n)-\zeta(\ell(Q(a))n)\big\rangle.\\
 &&\hspace{-30mm}\big\langle \alpha^{*}(\xi^{*}(m^{*})\ell^{*}(a))-\beta^{*}(m^{*}\ell^{*}(Q(a)))-\xi^{*}(m^{*})\ell^{*}(Q(a)), n\big\rangle \\
 &=&\big\langle m^{*}, \xi (\ell(a)\alpha(n))-\ell(Q(a))\beta(n)-\xi(\ell(Q(a))n)\big\rangle.\\
 &&\hspace{-30mm}\big\langle \alpha^{*}(r^{*}(a)\xi^{*}(m^{*}))-\beta^{*}(r^{*}(Q(a))m^{*})-r^{*}(Q(a))\xi^{*}(m^{*}), n\big\rangle \\
 &=& \big\langle m^{*}, \xi (\alpha (n)r (a))-\beta (n)r (Q(a))-\xi (nr (Q(a)))\big\rangle.\\
 &&\hspace{-30mm}\big\langle \alpha^{*}(r^{*}(a)\xi^{*}(m^{*}))-\alpha^{*}(r^{*}(Q(a))m^{*})-r^{*}(Q(a))\zeta^{*}(m^{*}), n\big\rangle \\
 &=& \big\langle m^{*}, \xi (\alpha (n)r (a))-\alpha (n)r (Q(a))-\zeta (nr (Q(a)))\big\rangle.\\
 &&\hspace{-30mm}\big\langle \beta^{*}(\zeta^{*}(m^{*})\ell^{*}(a))-\alpha^{*}(m^{*}\ell^{*}(T(a)))-\zeta^{*}(m^{*})\ell^{*}(T(a)), n\big\rangle \\
 &=&\big\langle m^{*}, \zeta (\ell(a)\beta(n))-\ell(T(a))\alpha(n)-\zeta(\ell(T(a))n)\big\rangle.\\
 &&\hspace{-30mm}\big\langle \beta^{*}(\zeta^{*}(m^{*})\ell^{*}(a))-\beta^{*}(m^{*}\ell^{*}(T(a)))-\xi^{*}(m^{*})\ell^{*}(T(a)), n\big\rangle \\
 &=&\big\langle m^{*}, \zeta (\ell(a)\beta(n))-\ell(T(a))\beta(n)-\xi(\ell(T(a))n)\big\rangle.\\
 &&\hspace{-30mm}\big\langle \beta^{*}(r^{*}(a)\zeta^{*}(m^{*}))-\beta^{*}(r^{*}(T(a))m^{*})-r^{*}(T(a))\xi^{*}(m^{*}), n\big\rangle \\
 &=& \big\langle m^{*}, \zeta (\beta (n)r (a))-\beta (n)r (T(a))-\xi (nr (T(a)))\big\rangle.\\
 &&\hspace{-30mm}\big\langle \beta^{*}(r^{*}(a)\zeta^{*}(m^{*}))-\alpha^{*}(r^{*}(T(a))m^{*})-r^{*}(T(a))\zeta^{*}(m^{*}), n\big\rangle \\
 &=& \big\langle m^{*}, \zeta (\beta (n)r (a))-\alpha (n)r (T(a))-\zeta (nr (T(a)))\big\rangle.
 \end{eqnarray*}}
 Hence, Eqs.\eqref{eq:dn4}-\eqref{eq:dn7} hold if and only if Eqs.\eqref{eq:dn}-\eqref{eq:dn3} hold.
 \end{proof}

 \begin{thm} \label{thm:do} Let $(A, R, S)$ be a $(\xi, \zeta)$-admissible symmetric Rota-Baxter system with respect to the $A$-bimodule $(M, \ell, r)$, $Q, T : A \rightarrow A, \alpha, \beta, \xi, \zeta : M \rightarrow M$ and $\mathfrak{T} : M \rightarrow A$ linear maps which is identified as an element in $(A \ltimes_{r^{*}, \ell^{*}} M^{*}) \otimes (A \ltimes_{r^{*}, \ell^{*}} M^{*})$.

 \begin{enumerate}[(1)]
   \item \label{it:thm:do1} The element $r = \mathfrak{T} - \sigma(\mathfrak{T})$ is an antisymmetric solution of the $(Q + \alpha^{*}, T+\beta^{*})$-admissible aYBe in the symmetric Rota-Baxter system $(A \ltimes_{r^{*}, \ell^{*}} M^{*}, R+\xi^{*}, S+\zeta^{*})$ if and only if $\mathfrak{T}$ is a weak $\mathcal{O}$-operator associated to $(M, \ell, r)$ and $(\alpha, \beta)$ and $\mathfrak{T}\ci \xi = Q\ci \mathfrak{T}, \mathfrak{T}\ci \zeta = T\ci \mathfrak{T}$.

   \item \label{it:thm:do2} Assume that $(M, \ell, r, \alpha, \beta)$ is a representation of $(A, R, S)$. If $\mathfrak{T}$ is an $\mathcal{O}$-operator associated to $(M, \ell, r, \alpha, \beta)$ and $\mathfrak{T}\ci \xi = Q\ci \mathfrak{T}, \mathfrak{T}\ci \zeta = T\ci \mathfrak{T}$, then $r = \mathfrak{T} - \sigma(\mathfrak{T})$ is an antisymmetric solution of the $(Q + \alpha^{*}, T+\beta^{*})$-admissible aYBe in the symmetric Rota-Baxter system $(A \ltimes_{r^{*}, \ell^{*}} M^{*}, R+\xi^{*}, S+\zeta^{*})$. If, in addition, $(A, R, S)$ is $(Q, T)$-admissible and Eqs.\eqref{eq:dn}-\eqref{eq:dn3} hold, then the symmetric Rota-Baxter system $(A \ltimes_{r^{*}, \ell^{*}} M^{*}, R+\xi^{*}, S+\zeta^{*})$ is ($Q + \alpha^{*}, T+\beta^{*}$)-admissible. In this case, there is a symmetric Rota-Baxter ASI bisystem $(A \ltimes_{r^{*}, \ell^{*}} M^{*}, R+\xi^{*}, S+\zeta^{*}, Q + \alpha^{*}, T+\beta^{*})$ where the linear map $\Delta$ is defined by Eq.\eqref{eq:db} with $r = \mathfrak{T} - \sigma(\mathfrak{T})$.
  \end{enumerate}
  \end{thm}

 \begin{proof} \ref{it:thm:do1} Let $\{e_{1}, e_{2}, ..., e_{n}\}$ be a basis of $M$ and  $\{e_{1}^{*}, e_{2}^{*},..., e_{n}^{*}\}$ be the dual basis. Then
 \begin{eqnarray*}
 &&\hspace{-6mm}\mathfrak{T}=\sum_{i} e_{i}^{*}\otimes \mathfrak{T}(e_{i})\in  M^{*}\otimes A\subset (A\oplus M^{*})\otimes  (A\oplus M^{*}),\\
 &&\hspace{-6mm}r=\mathfrak{T}-\sigma(\mathfrak{T})=\sum_{i} (e_{i}^{*}\otimes \mathfrak{T}(e_{i})-\mathfrak{T}(e_{i})\otimes e_{i}^{*}).
 \end{eqnarray*}
 Note that
 \begin{eqnarray*}
 &&\hspace{-6mm}((Q+\alpha^{*})\otimes \id)(r)=\sum_{i} (\alpha^{*}(e_{i}^{*})\otimes \mathfrak{T}(e_{i})-Q(\mathfrak{T}(e_{i}))\otimes e_{i}^{*}),\\
 &&\hspace{-6mm}(\id\otimes(R+\xi^{*}))(r)=\sum_{i} (e_{i}^{*}\otimes R(\mathfrak{T}(e_{i}))-\mathfrak{T}(e_{i})\otimes \xi^{*}(e_{i}^{*})),\\
 &&\hspace{-6mm}((T+\beta^{*})\otimes \id)(r)=\sum_{i} (\beta^{*}(e_{i}^{*})\otimes \mathfrak{T}(e_{i})-T(\mathfrak{T}(e_{i}))\otimes e_{i}^{*}),\\
 &&\hspace{-6mm}(\id\otimes(S+\zeta^{*}))(r)=\sum_{i} (e_{i}^{*}\otimes S(\mathfrak{T}(e_{i}))-\mathfrak{T}(e_{i})\otimes \zeta^{*}(e_{i}^{*})).
 \end{eqnarray*}
 Further,
 \begin{eqnarray*}
 &&\hspace{-6mm}\sum_{i} \mathfrak{T}(e_{i})\otimes \xi^{*}(e_{i}^{*})=\sum_{i} \mathfrak{T}(\xi(e_{i}))\otimes e_{i}^{*},\quad \sum_{i} \alpha^{*}(e_{i}^{*})\otimes \mathfrak{T}(e_{i})=\sum_{i} e_{i}^{*}\otimes \mathfrak{T}(\alpha(e_{i})),\\
 &&\hspace{-6mm}\sum_{i} \mathfrak{T}(e_{i})\otimes \zeta^{*}(e_{i}^{*})=\sum_{i} \mathfrak{T}(\zeta(e_{i}))\otimes e_{i}^{*},\quad\sum_{i} \beta^{*}(e_{i}^{*})\otimes \mathfrak{T}(e_{i})=\sum_{i} e_{i}^{*}\otimes \mathfrak{T}(\beta(e_{i})).
 \end{eqnarray*}
 Therefore $((Q+\alpha^{*})\otimes \id)(r)=(\id\otimes(R+\xi^{*}))(r)$ if and only if $\mathfrak{T}\circ \xi=Q\circ \mathfrak{T}$ and $R\circ \mathfrak{T}=\mathfrak{T}\circ \alpha$. $((T+\beta^{*})\otimes \id)(r)=(\id\otimes(S+\zeta^{*}))(r)$ if and only if $\mathfrak{T}\circ \zeta=T\circ \mathfrak{T}$ and $S\circ \mathfrak{T}=\mathfrak{T}\circ \beta$. Hence the conclusion follows by \cite[Theorem 2.5.5]{Bai1} or \cite[\bf Claim]{Bai}.
	
 \ref{it:thm:do2} It follows Item \ref{it:thm:do1} and Theorem \ref{thm:dn}.
 \end{proof}

\subsection{Quasitriangular case} In this subsection, we assume that Char$K\neq$ 2. For $r \in A \otimes A$, set $r=\Lambda+\Gamma$ with $\Lambda=\frac{r+\sigma(r)}{2}, \Gamma=\frac{r-\sigma(r)}{2}$.

 If $r$ is a solution of the aYBe in $A$ and the symmetric part $\Lambda$ of $r$ is \textbf{$(\mathcal{L}, \mathcal{R})$-invariant} in the sense of
 \begin{equation}
 (\id\otimes \mathcal{L}(a)-\mathcal{R}(a)\otimes \id)\Lambda=0, \forall a\in A, \label{eq:db5}
 \end{equation}
 then $(A, \cdot, \Delta_{r})$ is an ASI bialgebra, where the linear map $\Delta_{r}$ is defined by Eq.\eqref{eq:db}. We call this bialgebra \textbf{quasitriangular}\cite{Bai1}. In fact, in this case, by Eq.\eqref{eq:db5},
 \begin{eqnarray}\label{eq:dbbbbbb}
 \D_r(a)=\frac{1}{2}(r^{1}\otimes a r^{2}-r^{1} a\otimes r^{2}-r^{2}\otimes a r^{1}+r^{2}a\otimes  r^{1}).
 \end{eqnarray}

 \begin{pro}\label{pro:jk} \label{cor:dg} Let $(A, R, S)$ be a $(Q, T)$-adjoint admissible symmetric Rota-Baxter system. If $\Lambda$ is $(\mathcal{L}, \mathcal{R})$-invariant and
 \begin{eqnarray}
 &(Q\otimes \id - \id\otimes R)\Gamma=0,~~~ (T\otimes \id - \id\otimes S)\Gamma=0.& \label{eq:dg}
 \end{eqnarray}
 hold, then we have Eqs.\eqref{eq:cu}-\eqref{eq:ck8} hold for $\D_r$ given by Eq.\eqref{eq:dbbbbbb}. Moreover, if $r$ is a solution of the aYBe in $A$, then $((A, \cdot \D_r), R, S, Q, T)$ is a symmetric Rota-Baxter ASI bisystem.
 \end{pro}

 \begin{proof} By the discussion above, we only need to check that Eqs.\eqref{eq:cu}-\eqref{eq:ck8} hold. For example,
 {\small\begin{eqnarray*}
 &&\hspace{-16mm}Q(a_{(1)}) \otimes Q(a_{(2)})-Q(Q(a)_{(1)}) \otimes Q(a)_{(2)} - Q(a)_{(1)} \otimes T(Q(a)_{(2)}) \\
 \hspace{-3mm}&\stackrel{\eqref{eq:ck}\eqref{eq:ck1}}{=}&\hspace{-3mm} \frac{1}{2}(Q(r^{1})\otimes Q(a r^{2})+Q(r^{1} a) \otimes Q(r^{2})-Q(r^{2})\otimes Q(a r^{1})+Q(r^{2} a) \otimes Q(r^{1})-Q(r^{1})\otimes Q(a)r^{2}\\
 &&-Q(R(r^{1})a)\otimes r^{2}+S(r^{1})Q(a)\otimes r^{2}+Q(r^{2})\otimes Q(a)r^{1}+Q(R(r^{2})a)\otimes r^{1}-S(r^{2})Q(a)\otimes r^{1}\\
 &&+r^{1}\otimes Q(a R(r^{2}))-r^{1}\otimes Q(a) R(r^{2})+r^{1}Q(a)\otimes T(r^{2})-r^{2}\otimes Q(a R(r^{1}))+r^{2}\otimes Q(a) R(r^{1})\\
 &&-r^{2}Q(a)\otimes T(r^{1}))\\
 \hspace{-6mm}&=&\hspace{-6mm}(\id\otimes \mathcal{L}(Q(a))-\id\otimes Q\circ \mathcal{L}(a))(Q\otimes \id -\id\otimes R)\frac{r-\sigma(r)}{2}\\
 &&+(Q\circ \mathcal{R}(a)\otimes \id)(\id\otimes Q-R\otimes \id)\frac{r-\sigma(r)}{2}\\
 \hspace{-6mm}&=&\hspace{-6mm}(\id\otimes \mathcal{L}(Q(a))-\id\otimes Q\circ \mathcal{L}(a))(Q\otimes \id -\id\otimes R)\Gamma+(Q\circ \mathcal{R}(a)\otimes \id)(\id\otimes Q-R\otimes \id)\Gamma
 \end{eqnarray*}
 }
 Then by Eq.\eqref{eq:dg}, the first equal in Eq.\eqref{eq:cu} holds. Similarly for other equalities.
 \end{proof}

\section{Applications}\label{se:appl}
 The symmetric Rota-Baxter ASI bisystem finds applications in a wide range of structures, including Nijenhuis algebras, Rota-Baxter Lie bisystems, Rota-Baxter ASI bialgebras of weight $\lambda$, averaging ASI bialgebras, special apre-perm bialgebras, averaging Lie bialgebras, and Rota-Baxter Lie bialgebras of weight $\lambda$. This diversity underscores the broad applicability of our framework.

 \subsection{To Rota-Baxter Lie systems} In this subsection, we introduce the notion of a Rota-Baxter Lie system. For the concept and terminology of Lie algebras, we refer to  \cite{CP}.

 \begin{defi} \label{de:gh} Let $(g,[,])$ be a Lie algebra. If the linear operators $R, S: g \rightarrow g$ satisfy
 \begin{eqnarray}
 &[R(x), R(y)] = R([R(x), y] + [x, S(y)] ),&\label{eq:gh0}\\
 &[S(x), S(y)] = S([R(x), y] + [x, S(y)] ), & \label{eq:gh1}
 \end{eqnarray}
 for all $x, y \in g$, then $((g, [,]), R, S)$ ($(g, R, S)$ for short) is called a \textbf{Rota-Baxter Lie system}.
 \end{defi}

 \begin{rmk}\label{pro:gi}
 \begin{enumerate}
   \item Let $((g, [,]), R, S)$ be a Rota-Baxter Lie system. Then we have
  \begin{eqnarray*}
 &[R(x), R(y)] = R([R(x), y] + [x, S(y)] )= R([S(x), y] + [x, R(y)] ),&\\
 &[S(x), S(y)] = S([R(x), y] + [x, S(y)] )= S([S(x), y] + [x, R(y)] ). &
 \end{eqnarray*}
   \item In \cite{ZM}, we investigate the bialgebra structures on a class of special Rota-Baxter Lie systems and present some derived results.
 \end{enumerate}
 \end{rmk}

 \begin{pro}\label{pro:gj} Let $((A, \cdot), R, S)$ be a symmetric Rota-Baxter system. Define bilinear map $[,]: A\otimes A \rightarrow A$ by
 \begin{eqnarray}
   [x, y]=x y-y x,\label{eq:gj1}
 \end{eqnarray}
 for all $x, y \in A$, then $((A, [,]), R, S)$ is a Rota-Baxter Lie system.
 \end{pro}

 \begin{proof} Obviously, $(A,[,])$ is a Lie algebra. For all $x, y\in A$, we have
 \begin{eqnarray*}
 &&\hspace{-18mm}[R(x), R(y)] - R([R(x), y] + [x, S(y)] )\\
 &=&R(x)R(y)-R(y)R(x) - R(R(x)y + xS(y)) + R(yR(x) + S(y)x)=0,\\
 &&\hspace{-18mm}[S(x), S(y)] - S([R(x), y] + [x, S(y)] )\\
 &=&S(x)S(y)-S(y)S(x) - S(R(x)y + xS(y)) + S(yR(x) + S(y)x)=0.
 \end{eqnarray*}
 Therefore $((A, [,]), R, S)$ is a Rota-Baxter Lie system.
 \end{proof}

 \begin{ex}\label{ex:ejj} By Example \ref{ex:cee} and Proposition \ref{pro:gj}, $((A, [,]), R, S)$ is a $2$-dimensional Rota-Baxter Lie system with the product $[,]$
 \begin{center}
 \begin{tabular}{r|rr}
 $[,]$ & $e$  & $f$ \\
 \hline
 $e$ & $0$  & $-ke$  \\
 $f$ & $ke$  & $0$  \\
 \end{tabular},
 where $k\neq 0\in \mathfrak{R}$.
 \end{center}\vskip1mm
 and $R, S$ are defined by
 \begin{eqnarray*}
 &R(e)=q_{1}e-q_{1}f,~~~R(f)=q_{2}e-q_{2}f,&\\
 &S(e)=q_{2}e-q_{1}f,~~~S(f)=q_{2}e-q_{1}f~~~(q_{1}\neq 0, q_{2}\neq 0).&
 \end{eqnarray*}
 \end{ex}
 \smallskip

 Dually, we have the notion of a Rota-Baxter Lie cosystem.

 \begin{defi} \label{de:ek} Let $(g, \delta)$ be a Lie coalgebra, write $\delta(x)=x_{[1]}\otimes x_{[2]}$. If linear operators $Q, T: g \rightarrow g$ satisfy
 \begin{eqnarray}
 &Q(x_{[1]})\otimes Q(x_{[2]})=Q(Q(x)_{[1]})\otimes Q(x)_{[1]}+Q(x)_{[1]}\otimes T(Q(x)_{[2]}),&\label{eq:ek0}\\
 &T(x_{[1]})\otimes T(x_{[2]})=Q(T(x)_{[1]})\otimes Q(x)_{[1]}+T(x)_{[1]}\otimes T(T(x)_{[2]}), & \label{eq:ek1}
 \end{eqnarray}
 for all $x \in g$, then $((g, \delta), Q, T)$ is called a \textbf{Rota-Baxter Lie cosystem}.
 \end{defi}

 \begin{rmk}\label{rmk:el} Let $((g, \delta), Q, T)$ be a Rota-Baxter Lie cosystem. Then for all $x\in g$, we have
  \begin{eqnarray*}
 Q(x_{[1]})\otimes Q(x_{[2]})&=&Q(Q(x)_{[1]})\otimes Q(x)_{[2]}+Q(x)_{[1]}\otimes T(Q(x)_{[2]})\\
 &=&T(Q(x)_{[1]})\otimes Q(x)_{[2]}+Q(x)_{[1]}\otimes Q(Q(x)_{[2]}),\\
 T(x_{[1]})\otimes T(x_{[2]})&=&Q(T(x)_{[1]})\otimes Q(x)_{[2]}+T(x)_{[1]}\otimes T(T(x)_{[2]})\\
 &=&T(T(x)_{[1]})\otimes Q(x)_{[2]}+T(x)_{[1]}\otimes Q(T(x)_{[2]}).
 \end{eqnarray*}
 \end{rmk}

 \begin{pro}\label{pro:em} Let $((C, \D), Q, T)$ be a symmetric Rota-Baxter cosystem. Define bilinear map $\d: C\rightarrow C\otimes C$ by
 \begin{eqnarray}
 x_{[1]}\otimes x_{[2]}=x_{(1)}\otimes x_{(2)}-x_{(2)}\otimes x_{(1)}\label{eq:em1}
 \end{eqnarray}
 for all $x\in C$, then $((C, \d), Q, T)$ is a Rota-Baxter Lie cosystem.
 \end{pro}

 \begin{proof}
 Dually to Proposition \ref{pro:gj}.
 \end{proof}

 \begin{ex}\label{ex:emm} By Example \ref{ex:cuu} and Proposition \ref{pro:em}, $((A, \delta), Q, T)$ is a $2$-dimensional Rota-Baxter Lie cosystem with the coproduct
 \begin{equation*}
 \delta(e)=0,\ \ \delta(f)=f\otimes e-e\otimes f
 \end{equation*}
 and $Q, T$ are defined by
 \begin{eqnarray*}
 &Q(e)=-q_{2}e+q_{1}f,~~~Q(f)=-q_{2}e+q_{1}f,&\\
 &T(e)=-q_{1}e+q_{1}f,~~~T(f)=-q_{2}e+q_{2}f~~~  (q_{1}\neq 0, q_{2}\neq 0).&
 \end{eqnarray*}
 \end{ex}

 Now, we present the definition of a Rota-Baxter Lie bisystem by combining a Rota-Baxter Lie system and a Rota-Baxter cosystem.

 \begin{defi}\label{de:eo} Let $((g, [,]), R, S)$ be a Rota-Baxter Lie system, $(M, \rho)$ a representation of Lie algebra $(g, [,])$, $\alpha, \beta: M\rightarrow M$ linear maps. If for all $x\in g, m\in M,$
 \begin{eqnarray*}
 &\rho(R(x))\alpha(m)=\alpha(\rho(R(x))m+\rho(x)\beta(m))=\alpha(\rho(S(x))m+\rho(x)\alpha(m)),&\\
 &\rho(S(x))\beta(m)=\beta(\rho(R(x))m+\rho(x)\beta(m))=\beta(\rho(S(x))m+\rho(x)\alpha(m)),&
 \end{eqnarray*}
 then we call $(M, \rho, \alpha, \beta)$ a \textbf{representation} of $((g, [,]), R, S)$.
 \end{defi}

 \begin{defi}\label{de:emmm} A \textbf{Rota-Baxter Lie bisystem} is a quintuple $((g, [,], \delta), R, S, Q, T)$, where $(g, [,], \delta)$ is a Lie bialgebra, $((g, [,]), R, S)$ is a Rota-Baxter Lie system, $((g, \delta), Q, T)$ is a Rota-Baxter Lie cosystem, such that the following equations hold:
 \begin{eqnarray}
 Q([R(x), y])\hspace{-3mm}&=&\hspace{-3mm}[R(x), Q(y)]+T([x, Q(y)])\nonumber\\
 \hspace{-3mm}&=&\hspace{-3mm}[S(x), Q(y)]+Q([x, Q(y)]),\label{eq:emm1}\\
 T([S(x), y])\hspace{-3mm}&=&\hspace{-3mm}[R(x), T(y)]+T([x, T(y)])\nonumber\\
 \hspace{-3mm}&=&\hspace{-3mm}[S(x), T(y)]+Q([x, T(y)]),\label{eq:emm2}\\
 Q(R(x)_{[1]})\otimes R(x)_{[2]}\hspace{-3mm}&=&\hspace{-3mm}R(x)_{[1]}\otimes R(R(x)_{[2]})+T(x_{[1]})\otimes R(x_{[2]})\nonumber\\
 \hspace{-3mm}&=&\hspace{-3mm}S(x)_{[1]}\otimes R(S(x)_{[2]})+ Q(x_{[1]})\otimes R(x_{[2]}),\label{eq:emm3}\\
 T(S(x)_{[1]})\otimes S(x)_{[2]}\hspace{-3mm}&=&\hspace{-3mm}R(x)_{[1]}\otimes S(R(x)_{[2]})+T(x_{[1]})\otimes S(x_{[2]})\nonumber\\
 \hspace{-3mm}&=&\hspace{-3mm}S(x)_{[1]}\otimes S(S(x)_{[2]})+ Q(x_{[1]})\otimes S(x_{[2]}),\label{eq:emm4}
 \end{eqnarray}
 for all $x, y \in g$.
 \end{defi}

 \begin{rmk}\label{rmk:de:emmm} $(g^{*}, Ad_{g}^{*}, Q^{*}, T^{*})$ is a representation of $((g, [,]), R, S)$ if and only if Eqs.(\ref{eq:emm1}) and (\ref{eq:emm2}) hold; $(g, Ad_{g^{*}}^{*}, R, S)$ is a representation of $((g^{*}, \delta^{*}), Q^{*}, T^{*})$ if and only if Eqs.(\ref{eq:emm3}) and (\ref{eq:emm4}) hold.
 \end{rmk}

 Next we give the equivalent characterization of a Rota-Baxter Lie bisystem via a class of matched pairs of Rota-Baxter Lie systems.

 \begin{defi}\label{de:eii} Let $((g, [,]_{g}), R_{g}, S_{g})$ and $((h, [,]_{h}), R_{h}, S_{h})$ be Rota-Baxter Lie systems. A \textbf{matched pair of $((g, [,]_{g}), R_{g}, S_{g})$ and $((h,[,]_{h}), R_{h}, S_{h})$} is a quadruple $(((g, [,]_{g}), R_{g}, S_{g}),$ $ ((h, [,]_{h}),R_{h}, S_{h}), \rho_{g}, \rho_{h})$, where $(g, \rho_{h}, R_{g}, S_{g})$ is a representation of $((h, [,]_{h}), R_{h}, S_{h})$, $(h, \rho_{g},$ $R_{h}, S_{h})$ is a representation of $((g, [,]_{g}), R_{g}, S_{g})$ and $((g, [,]_{g}), (h, [,]_{h}), \rho_{g}, \rho_{h})$ is a matched pair of $(g, [,]_{g})$ and $(h, [,]_{h})$.
 \end{defi}

 \begin{thm}\label{thm:ep} Let $((g, [,]), R, S)$ and $((g^{*}, \delta^{*}), Q^{*}, T^{*})$ be Rota-Baxter Lie systems. Then $((g, [,],$ $\delta), R, S, Q, T)$ is a Rota-Baxter Lie bisystem if and only if $(((g, [,]), R, S)$, $((g^{*}, [,]_{g^{*}}),$ $ Q^{*}, T^{*}), Ad_{g}^{*},Ad_{g^{*}}^{*})$ is a matched pair of $((g, [,]), R, S)$ and $((g^{*}, \delta^{*}),$ $Q^{*}, T^{*})$.
 \end{thm}

 \begin{proof} ($\Rightarrow$) If $((g, [,], \delta), R, S, Q, T)$ is a Rota-Baxter Lie bisystem, then by Definition $\ref{de:emmm}$ and Remark \ref{rmk:de:emmm}, we know that $(g^{*}, Ad_{g}^{*}, Q^{*}, T^{*})$ is a representation of $((g, [,]), R, S)$ and $(g, Ad_{g^{*}}^{*}, R, S)$ is a representation of $((g^{*}, \delta^{*}), Q^{*}, T^{*})$. By \cite{CP}, $((g, [,])$, $(g^{*}, \delta^{*}), Ad_{g}^{*}, Ad_{g^{*}}^{*})$ is a matched pair of $(g, [,])$ and $(g^{*}, \delta^{*})$ since $(g, [,], \delta)$ is a Lie bialgebra. Therefore, $(((g, [,]), R, S)$, $((g^{*}, \delta^{*}), Q^{*}, T^{*}), Ad_{g}^{*},  Ad_{g^{*}}^{*})$ is a matched pair of $((g, [,]), R, S)$ and $((g^{*}, \delta^{*}), Q^{*}, T^{*})$.
	
 ($\Leftarrow$) If $(((g, [,]), R, S)$, $((g^{*}, \delta^{*}), Q^{*}, T^{*}), Ad_{g}^{*},  Ad_{g^{*}}^{*})$ is a matched pair of $((g, [,]), R, S)$ and $((g^{*},$ $\delta^{*}), Q^{*}, T^{*})$, then by Definition $\ref{de:eii}$, $((g, [,]), (g^{*}, \delta^{*}), Ad_{g}^{*}, Ad_{g^{*}}^{*})$ is a matched pair of $(g, [,])$ and $(g^{*}, \d^{*})$. By \cite{CP}, $(g, [,], \delta)$ is a Lie bialgebra. By Definition $\ref{de:eii}$, $(g^{*}, Ad_{g}^{*}, Q^{*}, T^{*})$ is a representation of $((g, [,]), R, S)$ and $(g, Ad_{g^{*}}^{*}, R, S)$ is a representation of $((g^{*}, \delta^{*}), Q^{*}, T^{*})$. Then by Remark \ref{rmk:de:emmm}, $((g, [,], \delta), R, S, Q, T)$ is a Rota-Baxter Lie bisystem.
 \end{proof}

 \begin{pro}\label{pro:en} Let $((A, \cdot, \D), R, S, Q, T)$ be a symmetric Rota-Baxter ASI bisystem. Define $[,]$ and $\delta$ by Eq.\eqref{eq:gj1} and Eq.\eqref{eq:em1}, respectively. Then $((A, [,], \delta), R, S, Q, T)$ is a Rota-Baxter Lie bisystem.
 \end{pro}

 \begin{proof} By \cite[Theorem 5.7]{Ma}, we know that $(A, [,], \delta)$ is a Lie bialgebra. For all $a, b \in A$, we have
 \begin{eqnarray*}
 &&\hspace{-16mm}Q([R(a), b])-[R(a), Q(b)]-T([a, Q(b)])\\
 &=&Q(R(a)b)-R(a)Q(b)-T(aQ(b))-Q(bR(a))+Q(b)R(a)+T(Q(b)a)=0,\\
 &&\hspace{-16mm}Q(R(a)_{[1]})\otimes R(a)_{[2]}-R(a)_{[1]}\otimes R(R(a)_{[2]})-T(a_{[1]})\otimes R(a_{[2]})\\
 &=&Q(R(a)_{(1)})\otimes R(a)_{(2)}-Q(R(a)_{(2)})\otimes R(a)_{(1)}-R(a)_{(1)}\otimes R(R(a)_{(2)})\\
 &&+R(a)_{(2)}\otimes R(R(a)_{(1)})-T(a_{(1)})\otimes R(a_{(2)})+T(a_{(2)})\otimes R(a_{(1)})=0.
 \end{eqnarray*}
 The rest of Eqs.\eqref{eq:emm1}-\eqref{eq:emm4} can be checked similarly. Hence, $((A, [,], \delta), R, S, Q, T)$ is a Rota-Baxter Lie bisystem.
 \end{proof}

 \begin{ex}\label{ex:enn} Let $((A,[,]), R, S)$ be a Rota-Baxter Lie system defined by Exmaple \ref{ex:ejj}, $((A, \delta), -S, -R)$ be a Rota-Baxter Lie cosystem defined by Exmaple \ref{ex:emm}, then by Proposition \ref{pro:en}, $((A, [,], \delta), R, S, -S, -R)$ is a Rota-Baxter Lie bisystem.
 \end{ex}

\subsection{To Rota-Baxter ASI bialgebra of weight $\lambda$} Let us recall from \cite{BGM} the notion of a Rota-Baxter ASI bialgebra of weight $\lambda$.

 \begin{defi}\label{de:er} A \textbf{Rota-Baxter ASI bialgebra of weight $\lambda$} is a triple $((A, \cdot, \Delta), R, Q)$ consisting of a vector space $A$ and linear maps
 \begin{equation*}
 \cdot: A \otimes A \rightarrow A, \quad \Delta: A \rightarrow A \otimes A, \quad R, Q: A \rightarrow A
 \end{equation*}
 such that
 \begin{enumerate}[(1)]
 \item \label{it:de:er1}$(A, \cdot, \Delta)$ is an ASI bialgebra,
 \item \label{it:de:er2}$((A, \cdot), R)$ is a Rota-Baxter algebra,
 \item \label{it:de:er3}$((A, \Delta), Q)$ is a Rota-Baxter coalgebra, and
 \item \label{it:de:er4} the following compatibility conditions hold. For all $a, b \in A$,
 \begin{eqnarray}
 &Q(aR(b)) = Q(a)R(b) + Q(Q(a)b) + \lambda Q(a)b,&\label{eq:er1}\\
 &Q(R(a)b) = R(a)Q(b) + Q(aQ(b)) + \lambda aQ(b),&\label{eq:er2}\\
 &(\mathrm{id} \otimes Q)\Delta R = (R \otimes Q)\Delta + (R \otimes \mathrm{id})\Delta R + \lambda (R \otimes \mathrm{id})\Delta,&\label{eq:er3}\\
 &(Q \otimes \mathrm{id})\Delta R = (Q \otimes R)\Delta + (\mathrm{id} \otimes R)\Delta R + \lambda (\mathrm{id} \otimes R)\Delta.&\label{eq:er4}
 \end{eqnarray}
 \end{enumerate}
 \end{defi}

 \begin{thm}\label{thm:es} $(A, R, R+\lambda \id_{A}, Q, Q+\lambda \id_{A})$ is a symmetric Rota-Baxter ASI bisystem if and only if $(A, R, Q)$ is a Rota-Baxter ASI bialgebra of weight $\lambda$.
 \end{thm}

 \begin{proof} By Lemma \ref{lem:ec}, $(A, R, R+\lambda \id_{A})$ is a symmetric Rota-Baxter system if and only if $(A, R)$ is a Rota-Baxter algebra of weight $\lambda$. By Remark \ref{rmk:gb}, $(A, Q, Q+\lambda \id_{A})$ is a symmetric Rota-Baxter cosystem if and only if $(A, Q)$ is a Rota-Baxter coalgebra of weight $\lambda$. Hence the rest we only need to verify that Eqs.\eqref{eq:ck}-\eqref{eq:ck3} and Eqs.\eqref{eq:ck5}-\eqref{eq:ck8} hold for $R, R+\lambda \id_{A}, Q, Q+\lambda \id_{A}$ if and only if Eqs.\eqref{eq:er1}-\eqref{eq:er4} hold. In fact, for all $a, b\in A$, we calculate
 {\small\begin{eqnarray*}
 &&\hspace{-10mm}Q(R(a)b)-Q(a Q(b))-(R+\lambda \id_{A})(a)Q(b)=Q(R(a)b)-Q(a Q(b))-R(a)Q(b)-\lambda a Q(b),\\
 &&\hspace{-10mm}Q(R(a)b)-R(a)Q(b)-(Q+\lambda \id_{A})(a Q(b))=Q(R(a)b)-R(a)Q(b)-Q(a Q(b))-\lambda a Q(b),\\
 &&\hspace{-10mm}Q(a R(b))-Q(Q(a)b)-Q(a)(R+\lambda \id_{A})(b)=Q(a R(b))-Q(Q(a)b)-Q(a)R(b) -\lambda Q(a)b,\\
 &&\hspace{-10mm}Q(a R(b))-Q(a)R(b)-(Q+\lambda \id_{A})(Q(a)b)=Q(a R(b))-Q(a)R(b)-Q(Q(a)b)-\lambda Q(a)b.
 \end{eqnarray*}
 }
 We can draw the conclusion through similar calculations.
 \end{proof}

 \subsection{To averaging ASI bialgebras} Let us recall from \cite{HC} the notion of an averaging ASI bialgebra.

 \begin{defi}\label{de:et} An \textbf{averaging ASI bialgebra} is a vector space $A$ together with linear maps
 \begin{equation*}
 \cdot: A \otimes A \rightarrow A, \quad \Delta: A \rightarrow A \otimes A, \quad R, Q: A \rightarrow A
 \end{equation*}
 such that the following conditions are satisfied.
 \begin{enumerate}[(1)]
 \item \label{it:de:et1}The triple $(A, \cdot, \Delta)$ is an ASI bialgebra.
 \item \label{it:de:et2}The triple $((A, \cdot), R)$ is an averaging algebra. i.e.,
 \begin{equation}
 R(a)R(b) = R(R(a)b)= R(a R(b)),~\forall a, b\in A.\label{eq:et1}
 \end{equation}
 \item \label{it:de:et3} The triple $((A, \Delta), Q)$ is an averaging coalgebra. i.e.,
 \begin{equation}
 (Q \otimes Q)\Delta = (Q \otimes \mathrm{id})\Delta Q= ( \mathrm{id} \otimes Q )\Delta Q.\label{eq:et2}
 \end{equation}
 \item \label{it:de:et4} The following equalities hold. For all $a, b \in A$.
 \begin{eqnarray}
 &R(a)Q(b)=Q(R(a)b)=Q(a Q(b)), &\label{eq:et3}\\
 &Q(a)R(b)=Q(a R(b))=Q(Q(a)b),& \label{eq:et4}\\
 &(Q \otimes R)\Delta = (Q \otimes \mathrm{id})\Delta R = (\mathrm{id} \otimes R)\Delta R,&   \label{eq:et5}\\
 &(R \otimes Q)\Delta = (R \otimes \mathrm{id})\Delta R = (\mathrm{id} \otimes Q)\Delta R.&  \label{eq:et6}
 \end{eqnarray}
 \end{enumerate}
 We denote this bialgebra by $((A, \cdot, \Delta), R, Q)$.
 \end{defi}

 \begin{rmk}\label{rmk:de:et} We note here that when $R=Q$, Eqs.\eqref{eq:et3}-\eqref{eq:et4} are exactly Eqs.\eqref{eq:et1} and Eqs.\eqref{eq:et5}-\eqref{eq:et6} are exactly Eqs.\eqref{eq:et2}.
 That is to say, $((A, \cdot, \Delta), R, R)$ is an averaging ASI bialgebra if and only if Items \ref{it:de:et1}-\ref{it:de:et3} hold.
 \end{rmk}

 \begin{thm}\label{thm:eu} $(A, R, 0, Q, 0)$ is a symmetric Rota-Baxter ASI bisystem if and only if $(A, R, Q)$ is an averaging ASI bialgebra.
 \end{thm}

 \begin{proof} By Remark \ref{rmk:eb}, $(A, R, 0)$ is a symmetric Rota-Baxter system if and only if $(A, R)$ is an averaging algebra. Dually, $(A, Q, 0)$ is a symmetric Rota-Baxter cosystem if and only if $(A, Q)$ is an averaging coalgebra. And Eqs.\eqref{eq:ck}-\eqref{eq:ck3}, \eqref{eq:ck5}-\eqref{eq:ck8} for $(A, R, 0, Q, 0)$ are exactly Eqs.\eqref{eq:et3}-\eqref{eq:et6}. Thus we finish the proof.
 \end{proof}

 Let $Q=-S, T=-R$ in Definition \ref{de:cx}:

 \begin{cor}\label{cor:cxx} Let $(A, \cdot, \D)$ be an ASI bialgebra, $(A, R, S)$ be a symmetric Rota-Baxter system  and $(A, -S, -R)$ be a symmetric Rota-Baxter cosystem. Then $(A, R, S, -S, -R)$ is a symmetric Rota-Baxter ASI bisystem if and only if for all $a, b\in A$, the following equations hold:
 \begin{eqnarray}
 &R(a)S(b)=R(a S(b))+S(R(a)b),&\label{eq:cxx1}\\
 &S(a)R(b)=S(aR(b))+R(S(a)b),&\label{eq:cxx2}\\
 &R(a_{(1)})\otimes S(a_{(2)})= R(S(a)_{(1)})\otimes S(a)_{(2)}+R(a)_{(1)}\otimes S(R(a)_{(2)}),&\label{eq:cxx3}\\
 &S(a_{(1)})\otimes R(a_{(2)})= S(R(a)_{(1)})\otimes R(a)_{(2)}+S(a)_{(1)}\otimes R(S(a)_{(2)}).&\label{eq:cxx4}
 \end{eqnarray}
 Particularly, for a commutative and cocomutative ASI bialgebra $(A, \cdot, \D)$, then Eq.\eqref{eq:cxx1} is exactly Eq.\eqref{eq:cxx2}, and Eq.\eqref{eq:cxx3} is exactly Eq.\eqref{eq:cxx4}.
 \end{cor}

 \begin{pro}\label{pro:fe} Let $(A, R, S, -S, -R)$ be a symmetric Rota-Baxter ASI bisystem. Then $(A, R-S, R-S)$ is an averaging ASI bialgebra.
 \end{pro}

 \begin{proof} For all $a, b\in A$, we have
 \begin{eqnarray*}
 &&\hspace{-26mm}(R-S)(a)(R-S)(b)-(R-S)((R-S)(a)b)\\
 &=&\hspace{-8mm}\underbrace{R(a)R(b)-R(R(a)b)}\underbrace{-R(a)S(b)+S(R(a)b)}
 \underbrace{+S(a)S(b)-S(S(a)b)}\\
 &&\hspace{-8mm}\underbrace{-S(a)R(b)+R(S(a)b)}\\
 &\stackrel{\eqref{eq:ea0}\eqref{eq:ea1}\eqref{eq:cxx1}\eqref{eq:cxx2}}{=}&\hspace{-3mm}R(a S(b))-R(a S(b))+S(aR(b))-S(aR(b))\\
 &=&\hspace{-8mm}0,
 \end{eqnarray*}
 This establishes the first equality in Eq.\eqref{eq:et1}, a similar argument applies to the second.

 On the other hand,
 \begin{eqnarray*}
 &&\hspace{-27mm}(R-S)(a_{(1)})\otimes (R-S)(a_{(2)})-(R-S)((R-S)(a)_{(1)})\otimes (R-S)(a)_{(2)}\\
 &=&\hspace{-8mm}\underbrace{R(a_{(1)})\otimes R(a_{(2)})-R(R(a_{(1)}))\otimes R(a)_{(2)}}\underbrace{-R(a_{(1)})\otimes S(a_{(2)})+R(S(a_{(1)}))\otimes S(a)_{(2)}}\\
 &&\hspace{-8mm}\underbrace{+S(a_{(1)})\otimes S(a_{(2)})-S(S(a_{(1)}))\otimes S(a)_{(2)}}\underbrace{-S(a_{(1)})\otimes R(a_{(2)})+S(R(a_{(1)}))\otimes R(a)_{(2)}}\\
 &\stackrel{\eqref{eq:cu}\eqref{eq:cu1}\eqref{eq:cxx3}\eqref{eq:cxx4}}{=}&\hspace{-3mm}R(a)_{(1)}\otimes S(R(a)_{(2)})-R(a)_{(1)}\otimes S(R(a)_{(2)})+S(a)_{(1)}\otimes R(S(a)_{(2)})\\
 &&\hspace{-8mm}-S(a)_{(1)}\otimes R(S(a)_{(2)})\\
 &=&\hspace{-8mm}0,
 \end{eqnarray*}
 Thus, the first equality in Eq.\eqref{eq:et2} holds, and analogously for the second equality. Hence, by Remark \ref{rmk:de:et}, $(A, R-S, R-S)$ is an averaging ASI bialgebra.
 \end{proof}

 Let us recall from \cite{CGM} some related definitions and results about Nijenhuis algebras.

 \begin{defi}\label{de:ew} A \textbf{Nijenhuis algebra} is a pair $((A, \cdot), N)$, where $(A, \cdot)$ is an algebra and $N: A\to A$ is a linear map such that for all $a, b\in A$, the condition below holds:
 \begin{eqnarray}
 N(a)\cdot N(b)+N^2(a\cdot b)=N(N(a)\cdot b+a\cdot N(b)).\label{eq:n}\label{eq:ew1}
 \end{eqnarray}
 \end{defi}

 \begin{rmk}\label{rmk:de:ew}
 Let $(A, \cdot)$ be an algebra. Then $((A, \cdot), \id_A)$ is a Nijenhuis algebra.
 \end{rmk}

 \begin{pro}\label{pro:fg} Let $((A, \cdot), R, S)$ be a symmetric Rota-Baxter system and Eqs.\eqref{eq:cxx1} and \eqref{eq:cxx2} hold, then $((A, \star), R-S)$ and $((A, \star'), R-S)$ are Nijenhuis algebras, where $a\star b= R(a)b+a S(b), a\star' b= S(a)b+a R(b)$.
 \end{pro}

 \begin{proof} By Remark \ref{rmk:ef}, $(A, \star)$ is an algebra. For all $a, b\in A$, we have
 \begin{eqnarray*}
 &&\hspace{-16mm}(R-S)(a)\star(R-S)(b)-(R-S)((R-S)(a)\star b+ a \star (R-S)(b)-(R-S)(a\star b))\\
 &=&\hspace{-3mm}\underbrace{R(R(a))R(b)-R(R(R(a))b)}\underbrace{+R(a)S(R(b))-R(a S(R(b)))}\\
 &&\underbrace{-R(R(a))S(b)+S(R(R(a))b)}\underbrace{-R(a)S(S(b))+R(a S(S(b)))}\\
 &&\underbrace{-R(S(a))R(b)+R(R(S(a))b)}\underbrace{-S(a)S(R(b))+S(a S(R(b)))}\\
 &&\underbrace{+R(S(a))S(b)-S(R(S(a))b)}\underbrace{+S(a)S(S(b))-S(a S(S(b)))}\\
 &\stackrel{\eqref{eq:cxx1}\eqref{eq:ea0}}{=}&\hspace{-3mm}R(R(a)S(b))+S(R(a)R(b))-R(R(a)S(b))-S(R(a)S(b))-R(S(a)S(b))\\
 &&-S(R(a)R(b))+R(S(a)S(b))+S(R(a)S(b))\\
 &=&\hspace{-6mm}0.
 \end{eqnarray*}
 Hence $((A, \star), R-S)$ is a Nijenhuis algebra. Similarly $((A, \star'), R-S)$ is also a Nijenhuis algebra.
 \end{proof}

 \begin{ex}\label{ex:fh} By Examples \ref{ex:fa} and \ref{ex:fb}, $(A \bowtie A^{*}, \mathbb{R}, \mathbb{S})$ is a symmetric Rota-Baxter system and $(-\mathbb{S}, -\mathbb{R})$ is adjoint admissible to $(A \bowtie A^{*}, \mathbb{R}, \mathbb{S})$, then by Proposition \ref{pro:fg}, $((A \bowtie A^{*}, \star_{A \bowtie A^{*}}(=\star'_{A \bowtie A^{*}})), \mathbb{R}-\mathbb{S})$
 is a Nijenhuis algebra, where
 \begin{eqnarray*}
 &(x+a^*) \star_{A \bowtie A^{*}} (y+b^*)
 =x\cdot_{A}y-a^*\cdot_{A^{*}}b^*,&\\ &(\mathbb{R}-\mathbb{S})(x+a^*)=x+a^*.&
 \end{eqnarray*}
 \end{ex}

 \subsection{To special apre-perm bialgebras} Let us recall from \cite{BGLZ} some related definitions and results about special apre-perm bialgebras.

 \begin{defi}\label{de:hf}A \textbf{special apre-perm algebra} is a triple $(A, \triangleright_A, \triangleleft_A)$, such that $A$ is a vector space, $\triangleright_A, \triangleleft_A: A \otimes A \rightarrow A$ are multiplications on $A$ and, for all $x, y, z \in A$, the following conditions hold:
 \begin{enumerate} [(1)]
 \item \label{it:de:hf1} the multiplication $\triangleleft_A$ is commutative.
 \item \label{it:de:hf2} $(A, \circ_A)$ is a perm algebra, where the multiplication $\circ_A: A \otimes A \to A$ is given by
 \begin{equation*}
 x \circ_A y = x \triangleright_A y + x \triangleleft_A y.
 \end{equation*}
 \item \label{it:de:hf3} the following equation holds:
 \begin{equation*}
 (x \circ_A y) \triangleleft_A z = x \circ_A (y \triangleleft_A z) = -x \triangleleft_A (y \triangleleft_A z).
 \end{equation*}
 \end{enumerate}
 \end{defi}

 \begin{defi}\label{de:hg} A \textbf{special apre-perm coalgebra} is a triple $(A, \vartheta, \theta)$, such that $A$ is a vector space and $\vartheta, \theta: A \to A \otimes A$ are co-multiplications satisfying the following equations:
 \begin{eqnarray*}
 &(\eta \otimes \mathrm{id})\eta(x) = (\mathrm{id} \otimes \eta)\eta(x),&\\
 &(\mathrm{id} \otimes \eta)\eta(x) = (\sigma \otimes \mathrm{id})(\eta \otimes \mathrm{id})\eta(x),& \\
 &\theta(x) = \sigma \theta(x),& \\
 &(\eta \otimes \mathrm{id})\theta(x) = (\mathrm{id} \otimes \theta)\eta(x),&\\
 &(\mathrm{id} \otimes \theta)(\eta + \theta)(x) = 0,&
 \end{eqnarray*}
 where $\forall~ x \in A$ and $\eta = \theta + \vartheta$.
 \end{defi}

 \begin{defi}\label{de:hi} Let $(A, \triangleright_A, \triangleleft_A)$ be a special apre-perm algebra and $(A, \vartheta, \theta)$ be a special apre-perm coalgebra. Suppose that for all $x, y \in A$, the following equations hold:
 \begin{eqnarray*}
 &\eta(x \circ_A y) = (\mathcal{L}_{\circ_A}(x) \otimes \mathrm{id})\eta(y) - (\mathrm{id} \otimes \mathcal{R}_{\circ_A}(y))\theta(x),&\\
 &\eta(x \circ_A y) = (\mathrm{id} \otimes \mathcal{R}_{\circ_A}(y))\eta(x) - (\mathcal{L}_{\triangleleft_A}(x) \otimes \mathrm{id})\eta(y),& \\
 &\eta(x \circ_A y) = (\mathrm{id} \otimes \mathcal{L}_{\circ_A}(x))\eta(y) + (\mathcal{L}_{\circ_A}(y) \otimes \mathrm{id})\theta(x),& \\
 &\eta(x \triangleleft_A y) = (\mathrm{id} \otimes \mathcal{L}_{\triangleleft_A}(x))\eta(y) + \sigma(\mathrm{id} \otimes \mathcal{L}_{\triangleleft_A}(y))\eta(x),& \\
 &\eta(x \triangleleft_A y) = \sigma \eta(x \triangleleft_A y),& \\
 &\theta(x \circ_A y) = (\mathrm{id} \otimes \mathcal{L}_{\circ_A}(x))\theta(y) + (\mathcal{L}_{\circ_A}(y) \otimes \mathrm{id})\theta(x),&\\
 &\theta(x \circ_A y) = \theta(y \circ_A x),&
 \end{eqnarray*}
 then $(A, \triangleright_A, \triangleleft_A, \vartheta, \theta)$ is a bialgebra, called a \textbf{special apre-perm bialgebra}.
 \end{defi}

 \begin{lem}\label{cor:fd}{\em \cite[Proposition 4.13]{BGLZ}}
 Let $((A, \cdot, \Delta), R, Q)$ be an averaging commutative and cocommutative ASI bialgebra. Then there is a special apre-perm bialgebra $(A, \triangleright_A, \triangleleft_A, \vartheta, \theta)$, where $\triangleright_A, \triangleleft_A$ are defined by
 \begin{eqnarray}
 x \triangleright_A y = R(x) y + Q(x y), \quad x \triangleleft_A y = -Q(x y), \quad \forall~x, y \in A.\label{eq:aprepermp}
 \end{eqnarray}
 and $\vartheta, \theta$ are defined by
 \begin{eqnarray}
 \vartheta(x) = (Q \otimes \mathrm{id})\Delta(x) + \Delta(R(x)), \quad \theta(x) = -\Delta(R(x)), \quad \forall x \in A. \label{eq:aprepermcp}
 \end{eqnarray}
 \end{lem}

 \begin{pro}\label{pro:hj} Let $(A, \cdot, \D)$ be a commutative and cocommutative ASI bialgebra. Then both symmetric Rota-Baxter ASI bisystems $(A, R, S, -S, -R)$ and $(A, R, 0, Q, 0)$ can induce special apre-perm bialgebras.
 \end{pro}

 \begin{proof} By Theorem \ref{thm:eu}, when $(A, R, 0, Q, 0)$ is a symmetric Rota-Baxter ASI bisystem, $(A, R, Q)$ is an averaging ASI bialgebra. By Lemma \ref{cor:fd}, $(A, R, Q)$ induces a special apre-perm bialgebra, where $\triangleright_A, \triangleleft_A, \vartheta, \theta$ are given by Eqs.\eqref{eq:aprepermp} and \eqref{eq:aprepermcp}.
 Similarly, by Proposition \ref{pro:fe}, when $(A, R, S, -S, -R)$ is a symmetric Rota-Baxter ASI bisystem, $(A, R-S, R-S)$ is an averaging ASI bialgebra. By Lemma \ref{cor:fd}, $(A, R-S, R-S)$ induces a special apre-perm bialgebra, where $\triangleright_A, \triangleleft_A, \theta, \vartheta$ are defined by
 \begin{eqnarray*}
 &x \triangleright_A y = R(x)y-S(x)y+R(xy)-S(xy), \ \ x \triangleleft_A y=-R(xy)+S(xy),&\\
 &\vartheta(x)= R(x_{(1)})\otimes x_{(2)}-S(x_{(1)})\otimes x_{(2)}+ R(x)_{(1)})\otimes R(x)_{(2)}-S(x)_{(1)})\otimes S(x)_{(2)},&\\
 &\theta(x)=-R(x)_{(1)})\otimes R(x)_{(2)}+S(x)_{(1)})\otimes S(x)_{(2)}.&
 \end{eqnarray*}
 These finish the proof.
 \end{proof}

 \subsection{To averaging Lie bialgebras} Let us recall from \cite{BGLZ1} the notion of an averaging Lie bialgebra.

 \begin{defi}\label{de:ev}
 An \textbf{averaging Lie bialgebra} is a vector space $g$ together with linear maps
 \begin{equation*}
 [,]: g \otimes g \rightarrow g, \quad \delta: g \rightarrow g \otimes g, \quad R,Q: g \rightarrow g
 \end{equation*}
 such that the following conditions are satisfied:
 \begin{enumerate}[(1)]
 \item \label{it:de:ev1} the triple $(g, [,], \delta)$ is a Lie bialgebra.
 \item \label{it:de:ev2}$((g, [,]), R, Q)$ is an admissible averaging Lie algebra. i.e.
 \begin{eqnarray*}
 &[R(x), R(y)] = R([R(x), y]),&\label{eq:ev1}\\
 &[R(x),Q(y)]=Q([R(x), y])=Q([x, Q(y)]), \forall x,y \in g &\label{eq:ev2}
 \end{eqnarray*}
 \item \label{it:de:ev3}$((g, \delta), Q, R)$ is an admissible averaging Lie coalgebra. i.e.
 \begin{eqnarray*}
 &(Q \otimes Q)\Delta = (Q \otimes \mathrm{id})\Delta Q, & \label{eq:ev3}\\
 &(Q \otimes R)\Delta = (Q \otimes \mathrm{id})\Delta R = (\mathrm{id} \otimes R)\Delta R.&\label{eq:ev4}
 \end{eqnarray*}
 \end{enumerate}
 We denote it by $((g, [,], \delta), R, Q)$.
 \end{defi}

  \begin{pro}\label{pro:ex}
  Let $((A, \cdot, \D), R, Q)$ be an averaging  ASI bialgebra, then $((A, [,], \delta), R, Q)$ is an averaging Lie bialgebra, where $[,]$ and $\delta$ defined by Eq.\eqref{eq:gj1} and Eq.\eqref{eq:em1}, respectively.
 \end{pro}

 \begin{proof}
 Similar to Proposition \ref{pro:en}.
 \end{proof}

 \begin{thm}\label{thm:eu-1}
 $((A, [,], \delta), R, 0, Q, 0)$ is a Rota-Baxter Lie bisystem if and only if $((A, [,], \delta),$ $R, Q)$ is an averaging Lie bialgebra.
 \end{thm}

 \begin{proof} When $((A, [,], \delta), R, 0, Q, 0)$ is a Rota-Baxter Lie bisystem, Definition \ref{de:emmm} implies that Eqs.\eqref{eq:gh0} and \eqref{eq:emm1} are equivalent to Item \ref{it:de:ev2} in Definition \ref{de:ev}, while Eqs.\eqref{eq:ek0} and \eqref{eq:emm3} are equivalent to Item \ref{it:de:ev3}.
 \end{proof}

 By Theorem \ref{thm:eu-1} and Proposition \ref{pro:en}, we have
 \begin{cor}\label{cor:ey} Let $((A, \cdot, \Delta), R, 0, Q, 0)$ be a symmetric Rota-Baxter ASI bisystem. Then $((A, [,], \delta), R, Q)$ is an averaging Lie bialgebra, where $[,]$ and $\delta$ defined by Eq.\eqref{eq:gj1} and Eq.\eqref{eq:em1}, respectively.
 \end{cor}

 \begin{pro}\label{pro:ez}
 Let $((A, [,], \delta),R, S,  -S, -R)$ be a  Rota-Baxter Lie bisystem, then, $((A, [,], \delta),$ $R-S,  R-S)$ is an averaging Lie bialgebra.
 \end{pro}

 \begin{proof} For all $a, b\in A$, we only need to verify the following equations hold:
 \begin{eqnarray*}
 &&\hspace{-10mm}[(R-S)(a), (R-S)(b)]-(R-S)([(R-S)(a),b])\\
 &=&[R(a),R(b)]-[R(a),S(b)]-[S(a),R(b)]+[S(a),S(b)]-R([R(a),b])+R([S(a),b])\\
 &&+S([R(a),b])-S([S(a),b])\\
 &=&0,\\
 &&\hspace{-10mm}(R-S)(a_{[1]})\otimes (R-S)(a_{[2]})-(R-S)((R-S)(a)_{[1]})\otimes (R-S)(a)_{[2]}\\
 &=&R(a_{[1]})\otimes R(a_{[2]})-R(a_{[1]})\otimes S(a_{[2]})-S(a_{[1]})\otimes R(a_{[2]})+S(a_{[1]})\otimes S(a_{[2]})\\
 &&-R(R(a_{[1]}))\otimes R(a)_{[2]}+S(R(a_{[1]}))\otimes R(a)_{[2]}+R(S(a_{[1]}))\otimes S(a)_{[2]}-S(S(a_{[1]}))\otimes S(a)_{[2]}\\
 &=&0,
 \end{eqnarray*}
 finishing the proof.
 \end{proof}

 By Proposition \ref{pro:ez} and Proposition \ref{pro:en}, we have
 \begin{cor}\label{cor:ha}
 Let $(A, R, S, -S, -R)$ be a symmetric Rota-Baxter ASI bisystem. Then $((A, [,],$ $\delta), R-S, R-S)$ is an averaging Lie bialgebra, where $[,]$ and $\delta$ defined by Eq.\eqref{eq:gj1} and Eq.\eqref{eq:em1}, respectively.
 \end{cor}

 \subsection{To Rota-Baxter Lie bialgebras of weight $\lambda$} Let us recall from \cite{BGLM} the notion of a Rota-Baxter Lie bialgebras of weight $\lambda$.

 \begin{defi}\label{de:he} A \textbf{Rota-Baxter Lie bialgebra of weight $\lambda$} is a vector space $\mathfrak{g}$ together with linear maps
 \begin{equation*}
 [,]: \mathfrak{g} \otimes \mathfrak{g} \rightarrow \mathfrak{g}, \quad \delta: \mathfrak{g} \rightarrow \mathfrak{g} \otimes \mathfrak{g}, \quad R, Q: \mathfrak{g} \to \mathfrak{g},
 \end{equation*}
 such that
 \begin{enumerate}[(1)]
 \item \label{it:de:he1} the triple $(\mathfrak{g}, [,], \delta)$ is a Lie bialgebra.
 \item \label{it:de:he2} the pair $((\mathfrak{g}, [,]), R)$ is a Rota-Baxter Lie algebra of weight $\lambda$. i.e.
 \begin{equation*}
 [R(x), R(y)] = R([R(x), y]+[x, R(y)]+ \lambda [x, y]),~~\forall x,y\in \mathfrak{g}.
 \end{equation*}
 \item\label{it:de:he3} the pair $((\mathfrak{g}, \delta), Q)$ is a Rota-Baxter Lie coalgebra of weight $\lambda$. i.e.
 \begin{equation*}
 (Q\otimes Q) \delta (x)= (Q\otimes \id+ \id \otimes Q)\delta (Q(x))+\lambda \delta (Q(x)),~~\forall x\in \mathfrak{g}.
 \end{equation*}
 \item \label{it:de:he4} the following equations hold:
 \begin{eqnarray*}
 &&Q([R(x), y])-[R(x), Q(y)] -Q([x, Q(y)]) -\lambda[x, Q(y)] = 0,\\
 &&(R \otimes Q) \delta + (R \otimes \id -\id \otimes Q)\delta(R(x)) + \lambda(R\otimes \id)\delta(x) = 0.
 \end{eqnarray*}
 \end{enumerate}
 We denote this bialgebra by $((\mathfrak{g}, [,], \delta), R, Q)$.
 \end{defi}

 \begin{pro}\label{pro:hb}
 Let $((A, \cdot, \D), R, Q)$ be a Rota-Baxter ASI bialgebra of weight $\lambda$, then $((A, [,], \delta), R, Q)$ is a Rota-Baxter Lie bialgebra of weight $\lambda$, where $[,]$ and $\delta$ defined by Eq.\eqref{eq:gj1} and Eq.\eqref{eq:em1}, respectively.
 \end{pro}

 \begin{proof} Similar to Proposition \ref{pro:en}.
 \end{proof}

 \begin{thm}\label{thm:bc}
 $((A, [,], \delta), R, R+\lambda id_A, Q, Q+\lambda \id_A)$ is a Rota-Baxter Lie bisystem if and only if $((A, [,], \delta), R, Q)$ is a Rota-Baxter Lie bialgebra of weight $\lambda$.
 \end{thm}

 \begin{proof} Similarly to Theorem \ref{thm:es}.
 \end{proof}

 By Theorem \ref{thm:bc} and Proposition \ref{pro:en}, we have

 \begin{cor}\label{cor:hd} Let $((A, \cdot, \D), R, R+\lambda id_A, Q, Q+\lambda id_A)$ be a symmetric Rota-Baxter ASI bisystem, then $((A, [,],$ $\delta), R, Q)$ is a Rota-Baxter Lie bialgebra of weight $\lambda$, where $[,]$ and $\delta$ defined by Eq.\eqref{eq:gj1} and Eq.\eqref{eq:em1}, respectively.
 \end{cor}

 \section{Conclusion} In \cite{Br1}, Brzezi\'{n}ski introduced the covariant bialgebra as a generalization of the infinitesimal bialgebra and established its relation to the Rota-Baxter system.

  \begin{defi}\mlabel{de:1.1} A {\bf covariant bialgebra} is a quadruple $(A,\delta_{1},\delta_{2},\Delta)$ such that
 \begin{enumerate}
   \item $A$ is an algebra,
   \item $(A,\Delta)$ is a coalgebra,
   \item Let $\delta_{i}: A\lr A\otimes A~(i=1,2)$~(write $\d_i(a)=a^i_{<1>)}\otimes a^i_{<2>}$) be two derivations, i.e., $\delta_{i}(a b)=a^{i}_{<1>}\otimes a^{i}_{<2>} b+a b^{i}_{<1>}\otimes b^{i}_{<2>}, i=1,2$. $\Delta$ is a covariant derivation with respect to $(\delta_{1},\delta_{2})$ in the sense of
   \begin{eqnarray*}
    &\Delta(a b)=a^{2}_{<1>}\otimes a^{2}_{<2>} b+a b_{(1)}\otimes b_{(2)}=a_{(1)}\otimes a_{(2)} b+a b^{1}_{<1>}\otimes b^{1}_{<2>}.&
   \end{eqnarray*}
 \end{enumerate}
 \end{defi}

 \begin{defi}\label{de:eh-1} Let $A$ be an algebra. An \textbf{associative Yang-Baxter pair} is a pair $(\mathfrak{r}, \mathfrak{s})$, where $\mathfrak{r}, \mathfrak{s} \in A\otimes A$ such that the following equations hold:
 \begin{eqnarray*}
 \mathfrak{r}_{12}\mathfrak{r}_{23}=\mathfrak{r}_{13}\mathfrak{r}_{12}
 +\mathfrak{s}_{23}\mathfrak{r}_{13},\\
 \mathfrak{s}_{12}\mathfrak{s}_{23}=\mathfrak{s}_{13}\mathfrak{r}_{12}
 +\mathfrak{s}_{23}\mathfrak{s}_{13}.
 \end{eqnarray*}
 \end{defi}

 By \cite[Proposition 3.4]{Br1}, a Rota-Baxter system can be constructed from an associative Yang-Baxter pair $(\mathfrak{r}, \mathfrak{s})$ via Eq.~\eqref{eq:ej1}. Furthermore, \cite[Proposition 3.15 and Corollary 3.17]{Br1} establish that if $(\mathfrak{r}, \mathfrak{s})$ is such a pair on an algebra $A$, and we define the maps $\d_{\mathfrak{r}}, \d_{\mathfrak{s}}, \D_{\mathfrak{r}, \mathfrak{s}}$ by
 \begin{align*}
 \d_{\mathfrak{r}}(a) &= a \mathfrak{r}^1 \otimes \mathfrak{r}^2 - \mathfrak{r}^1 \otimes \mathfrak{r}^2 a, \\
 \d_{\mathfrak{s}}(a) &= a \mathfrak{s}^1 \otimes \mathfrak{s}^2 - \mathfrak{s}^1 \otimes \mathfrak{s}^2 a, \\
 \D_{\mathfrak{r},\mathfrak{s}}(a) &= a \mathfrak{r}^1 \otimes \mathfrak{r}^2 - \mathfrak{s}^1 \otimes \mathfrak{s}^2 a,
 \end{align*}
 then the quadruple $(A, \d_{\mathfrak{r}}, \d_{\mathfrak{s}}, \D_{\mathfrak{r}, \mathfrak{s}})$ forms a covariant bialgebra.

 According to Remark~\ref{rmk:ei}, a pair $(\mathfrak{r}, \mathfrak{s})$ is symmetric if and only if both $(\mathfrak{r}, \mathfrak{s})$ and $(\mathfrak{s}, \mathfrak{r})$ are associative Yang-Baxter pairs. Consequently, a symmetric associative Yang-Baxter pair $(\mathfrak{r}, \mathfrak{s})$ induces two covariant bialgebra structures: $(A, \d_{\mathfrak{r}}, \d_{\mathfrak{s}}, \D_{\mathfrak{r}, \mathfrak{s}})$ and $(A, \d_{\mathfrak{s}}, \d_{\mathfrak{r}}, \D_{\mathfrak{s}, \mathfrak{r}})$. This leads to a natural question: how can the notion of a covariant bialgebra be adapted so that it corresponds uniquely to a suitably modified notion of a symmetric associative Yang-Baxter pair?

\section*{Acknowledgments} 
Ma and Zhao are supported by the National Natural Science Foundation of China (No. 12471033).
Li is supported by the National Natural Science Foundation of China (No. 12471130).

\smallskip

\noindent
{\bf Declaration of interests.} The authors have no conflicts of interest to disclose.

\noindent
{\bf Data availability.} Data sharing is not applicable as no new data were created or analyzed.


 \end{document}